\documentclass[utf8,12pt,nothms,a4paper,english]{aclart}

\usepackage{hyperref}

 \DeclareUnicodeCharacter{00A0}{ }

 \usepackage{lmodern}
 \usepackage{color}

\numberwithin{equation}{subsection}
\theoremstyle{plain}

\newtheorem{theo}[equation]{\theoname}
\newtheorem{lemm}[equation]{\lemmname}
\newtheorem{prop}[equation]{\propname}
\newtheorem{coro}[equation]{\coroname}
\newtheorem{setting}{Setting}
\theoremstyle{remark}
\newtheorem{rema}[equation]{\remaname}

\makeatletter
\let\c@paragraph\c@equation

\let\cl@paragraph\cl@equation
\makeatother
\def\sozat{\,;\,}

\usepackage[all]{xy}

\theoremstyle{plain}
\newtheorem*{theo*}{\theoname}
\theoremstyle{remark}

\def\abs#1{\left|#1\right|}
\def\norm#1{\left\|#1\right\|}
\def\C{\mathbf C}
\def\R{\mathbf R}
\def\N{\mathbf N}
\def\Z{\mathbf Z}
\def\Q{\mathbf Q}
\def\P{\mathbf P}
\let\ra\rightarrow \let\hra\hookrightarrow

\let\bar\overline
\let\phi\varphi 
\let\emptyset\varnothing
\def\lpar{\mathopen{(\!(}} \def\rpar{\mathclose{)\!)}}
\def\lbra{\mathopen{[\![}} \def\rbra{\mathclose{]\!]}}

\def\Tr{\operatorname{Tr}}
\def\Res{\operatorname{Res}}

\makeatletter
\def\Cl{\operatorname{Cl}}
\newcommand\Clan[1][]{\operatorname{Cl}^{\mathrm{an}%
   \@ifempty{#1}{}{,#1}}}
\def\Clanmax{\operatorname{Cl}^{\mathrm{an},\mathrm{max}}}

\newcommand\ga[1][]{\mathbf G_{\mathrm a\@ifempty{#1}{}{,#1}}}
\newcommand\gm[1][]{\mathrm G_{\mathrm m\@ifempty{#1}{}{,#1}}}
\makeatother
\def\Card{\operatorname{Card}}
\def\Hilb{\operatorname{Hilb}}
\def\Pic{\operatorname{Pic}}
\def\ord{\operatorname{ord}}
\def\Spec{\operatorname{Spec}}
\def\Proj{\operatorname{Proj}}
\def\pr{\operatorname{pr}}
\def\div{\operatorname{div}}
\def\ac{\operatorname{ac}}

\def\AD{\mathbb A}

\def\eff{\mathrm{eff}}
\def\GVar{\mathrm{KVar}}
\def\GVarM{\mathscr M}
\def\GExpVarM{\mathscr{E}\!\mathit{xp}\mathscr M}
\def\GExpVar{\mathrm{KExpVar}}
\def\Lef{\mathbf{L}}
\makeatletter
\def\ie{\emph{ie.}\@ifnextchar,{\xspace}{,\xspace}}
\def\eg{\emph{eg.}\@ifnextchar,{\xspace}{,\xspace}}
\makeatother
\def\C{\mathbf C}
\def\Z{\mathbf Z}

\def\Aff{\mathbf A}

\def\ExpVar{\mathrm{ExpVar}}
\let\ExpMot\GExpVarM
\let\Mot\GVarM
\def\Spec{\operatorname{Spec}}
\def\Lef{\mathbf L}
\def\ord{\operatorname{ord}}
\def\res{\operatorname{res}}
\def\Id{\operatorname{Id}}
\def\Hom{\operatorname{Hom}}
\def\pr{\operatorname{pr}}
\let\div\relax
\usepackage{mathabx}
\let\PC\boxdot
\def\div{\operatorname{div}}
\let\ra\rightarrow

\def\ie{\emph{i.e.}\xspace}
\def\resp{\emph{resp.}\xspace}
\iffalse
\def\tcb{\textcolor{blue}}
\def\tcr{\textcolor{blue}}
\else
\let\tcb\relax\let\tcr\relax
\def\color#1{}
\fi
\begin{document}

\title {Motivic height zeta functions}
\author{Antoine Chambert-Loir}
\address{D\'epartement de math\'ematiques d'Orsay \\ B\^atiment 425 \\
Facult\'e des sciences d'Orsay \\ Universit\'e Paris-Sud \\ F-91405 Orsay Cedex}
\email{Antoine.Chambert-Loir@math.u-psud.fr}
\author{Fran\c cois Loeser}
\address{Sorbonne Universit\'es, UPMC Univ Paris 06, UMR 7586 CNRS, Institut Math\'ematique de Jussieu, F-75005 Paris}
\email{Francois.Loeser@upmc.fr}

\begin{abstract}
We consider a motivic analogue of the height zeta function for integral points
of equivariant partial compactifications of affine spaces.
We establish its rationality and determine its largest pole.
\end{abstract}

\dedicatory{Dedicated to the memory of Professor Jun-ichi Igusa}

\maketitle

 \def\theequation{\arabic{equation}}



A major problem of Diophantine geometry is to understand the distribution
of rational or integral points of algebraic varieties defined over
number fields. For example, a well-studied question put
forward by Manin in~\cite{batyrev-m90}
is that of
an asymptotic expansion for the number of rational/\allowbreak integral
points of bounded height. A basic tool is the height zeta function
which is a Dirichlet series.

Around 2000, E. Peyre suggested to consider the analogous problem over function
fields, which has then an even more geometric flavor since it
translates as a problem of enumerative geometry, namely
counting algebraic curves of given degree and establishing properties
of the corresponding generating series. 
In view of the developments of motivic integration
by Kontsevich, Denef--Loeser~\cite{denef/loeser1999b}, etc.,
it is natural to look at a more general generating series
which not only counts the \emph{number} of such algebraic curves, but
takes into account the \emph{space} they constitute 
in a suitable Hilbert scheme. 

A natural  coefficient ring for the generating
series is the Grothendieck ring of varieties~$\GVar_k$: if $k$
is a field, this ring
is generated as a group by the isomorphism
classes of $k$-schemes of finite type, 
the addition being subject to obvious
cut-and-paste relations, and the product is induced by
the product of $k$-varieties. One interest of this generalization
is that it also makes sense in a purely geometric context, where no counting
is available.

Such a situation has been first studied  in
a paper by Kapranov in~\cite{kapranov2000},
where the analogy with the mentioned diophantine problem is not
pointed out. Later, Bourqui made some progress on the
motivic analogue of Manin's problem,
see~\cite{bourqui2009}, as well as his survey
report~\cite{bourqui2011}.


\medskip

In this article, we consider the following situation.\par\nobreak

\begin{setting}
Let $k$ be an algebraically closed
 field of characteristic zero. 
Let $C_0$ be a quasi-projective smooth connected curve over~$k$
and let~$C $ be its smooth projective compactification; 
we let $S=C \setminus C_0$.
Let $F=k(C)$ be the function field of~$C $ and $g$ be its genus.

Let $X$ be a projective irreducible $k$-scheme together with a
non-constant morphism $\pi\colon X\ra C$. 
Let $G$ and~$U$ be Zariski open subsets of~$X$ such that $G\subset U\subset X$.
Let $\mathscr L$ be a line bundle on~$X$;
we assume that there exists an effective $\Q$-divisor~$D$ 
supported on~$(X\setminus U)_F$ such that $\mathscr L(-D)$ is ample
on~$X_F$. 
\end{setting}

We are interested in sections $\sigma\colon C \ra X$
of~$\pi$ such that $\sigma(C_F)\subset G_F$ and $\sigma(C_0)\subset U$. 
\tcr{As in the Hasse principle, existence of \emph{local}  such sections
is a necessary condition to the existence of global sections~$\sigma$.}

\bgroup\color{blue}
\begin{setting}
We assume that for every $v\in C_0$,
$G(F_v)\cap U(\mathfrak o_v)\neq\emptyset$,
where $\mathfrak o_v$ is the completion of~$\mathscr O_{C,v}$
and $F_v$ is its field of functions.
\end{setting}
\egroup

We then want to study the family of such sections~$\sigma$
with prescribed degree $n=\deg\sigma^*\mathscr L$.
This is a geometric/motivic analogue of the  variant of Manin's problem
for \emph{integral points}.

By Proposition~\ref{lemm.hilbert}, these conditions define a 
constructible set~$M_{U,n}$ (in some $k$-scheme); moreover, the hypothesis
on~$\mathscr L$ implies that there exists~$n_0\in\Z$ 
such that $M_{U,n}$ is empty  for $n\leq n_0$.
Considering the classes~$[M_{U,n}]$ of these sets in~$\GVar_k$,
we form the generating Laurent series
\begin{equation*}\label{eq.def-ZU}
  Z_U(T) = \sum_{n\in\Z} [M_{U,n}]  T^{n}   \end{equation*}
and ask about its  properties.
 

Precisely, we investigate in this paper the motivic counterpart
of the situation studied 
recently in the paper~\cite{chambert-loir-tschinkel:2012} 
by Y.~Tschinkel and the first named author.

\begin{setting}
In this paper, we consider the particular case
where $G_F$ is the additive group~$\ga[F]^n$, $U_F=G_F$ and $X_F$
admits an action of~$G_F$ which extends the group action
of~$G_F$ over itself.
We also assume that 
the irreducible components
of the divisor at infinity~$\partial X=X\setminus G$
are smooth and meet transversally.
Finally, we restrict ourselves to the case where the restriction
of~$\mathscr L$ to the generic fiber~$X_F$ is equal
to $-K_{X_F}(\partial X_F)$, the log-anticanonical line bundle. 
\end{setting}
(As explained at the end of Section~\ref{s.setup-ag}, 
this line bundle satisfies the previous ampleness assumption.)
 
Let $\Lef$ be the class of the affine line~$\Aff^1_k$ in~$\GVar_k$
and let $\GVarM_k$ be the localized ring of~$\GVar_k$
with respect to the multiplicative subset~$S$
generated by $\Lef$ and the elements $\Lef^a-1$, for $a\in\N_{>0}$.
An element of~$\GVar_k$ is said to be \emph{effective} if it can be
written as a sum of classes of algebraic varieties;
similarly, an element of $\GVarM_k$ is effective if its
product by some element of~$S$ is the image of an effective
element of~$\GVar_k$. For example, $1-\Lef^{-a}=\Lef^{-a}(\Lef^a-1)$ 
is effective for every $a>0$.

Let $\GVarM_k\{T\}^\dagger$ and $\GVarM_k\{T\}$
be the subrings of~$\GVarM_k\lbra T\rbra [T^{-1}]$
generated by $\GVarM_k[T, T^{-1}]$ and the inverses of the polynomials
$1-\Lef^aT^b$, \tcb{where $(a,b)\in\N\times\N_{>0}$ are integers 
such that $b>a$, respectively $b\geq a$.}
For $b>a\geq 0$, $1-\Lef^{a-b}=\Lef^{a-b}(\Lef^{b-a}-1)$ is
invertible in~$\GVarM_k$, 
so that every element~$P$ of $\GVarM_k\{T\}^\dagger$ has a value $P(\Lef^{-1})$
at $T=\Lef^{-1}$
which is an element of~$\GVarM_k$.

\vskip0pt plus 4\baselineskip\penalty-30\vskip0pt plus-4\baselineskip

The following theorem is the main result of this paper.
\begin{theo}\label{theo.main}
\tcr{Assume the notation and hypotheses of Settings~1, 2, and 3, are in force.}

\nobreak
The Laurent series $Z_U(T)$ 
belongs to $\GVarM_k\{T\}$.
More precisely, there exists an integer~\mbox{$a\geq 1$},
an element $P_U(T)\in\GVarM_k\{T\}^\dagger$
such that $P_U(\Lef^{-1})$ is an effective non-zero element of $\GVarM_k$,
and a positive integer~$d$
such that
\[ (1-\Lef^a T^a)^d Z_U(T) =P_U(T). \]
\end{theo}
Any $k$-constructible set~$M$ can be written as a finite disjoint
union  of integral $k$-varieties; we let $\dim(M)$ be the
maximal dimension of these varieties and $\kappa(M)$ be the number of 
such varieties of maximal dimension; they do not
depend on the chosen partition.
\begin{coro}\label{coro.main}
For every integer~$p\in\{0,\dots,a-1\}$, one of the following cases
occur when $n$ tends to infinity in the congruence class of~$p$ modulo~$a$:
\begin{enumerate}
\item Either $\dim (M_{U,n})=\mathrm o(n)$,
\item Or $\dim (M_{U,n})-n$ has a finite limit and $\log(\kappa(M_{U,n}))/\log(n)$
converges to some integer in~$\{0,\dots,d-1\}$.
\end{enumerate}
Moreover, the second case happens at least for one integer~$p$.
\end{coro}
 
We observe that this condition on congruence classes is unavoidable in general.
For example, if $\mathscr L$ is a multiple $\mathscr L_0^a$ 
of a class in~$\Pic(\mathscr X)$, then $M_{U,n}=\emptyset$
for $n\nmid a$.

In the arithmetic case,
the corresponding question 
consists in establishing the analytic
property of the height zeta function (holomorphy for $\Re(s)>1$,
meromorphy on a larger half-plane, pole of order~$d$ at $s=1$)
as well as showing that the  number of points of height~$\leq B$
grows as $B(\log(B))^{d-1}$.
Its proof in~\cite{chambert-loir-tschinkel:2012} relies on 
the Poisson summation formula for the discrete cocompact subgroup $G(F)$
of the adelic group~$G(\AD_F)$.
In the present work, we take advantage
of the motivic Poisson formula recently established 
by E.~Hrushovski and D.~Kazhdan in~\cite{hrushovski-kazhdan:2009}
to prove new results in the geometric setting.

However, in its present form, this motivic Poisson formula  
suffers two limitations. Firstly, the functions it takes as input 
may only depend on finitely many places of the given function field.
For this reason, the question we solve in this paper is a geometric analogue
of Manin's problem for integral points, rather than for rational points.
Secondly, the Poisson formula only applies to vector groups,
and this is why our varieties are assumed to be equivariant compactifications
of such groups.

\medskip

The plan of the paper is the following. 

We begin the paper by an exposition, in a self-contained geometric
language, of the motivic Poisson formula of Hrushovski--Kazhdan.
We then gather in Section~\ref{s.prelim} some preliminary results
needed for the proof.
In particular, we show in Proposition~\ref{prop.tauber}
that Corollary~\ref{coro.main} is a consequence
from Theorem~\ref{theo.main}.
For eventual reference, we also prove there
a general existence theorem for the moduli spaces
which we study here,  see Proposition~\ref{lemm.hilbert}.
We end this Section by recalling some notation
on Clemens complexes, and on functions on arc spaces
with values in $\Mot_k$.

In Section~\ref{s.setup}, we lay out the foundations for
the proof of Theorem~\ref{theo.main}. Its main goal
consists in describing the moduli spaces as adelic subsets
of the group~$G$.

The core of the proof of Theorem~\ref{theo.main} begins with 
Section~\ref{s.proof}. We first apply the motivic Poisson summation
formula of Hrushovski and Kazhdan. We show that this 
formula gives an expression
$Z(T)$  as a ``sum'' (in the sense of motivic integration) 
over $\xi\in G(F)$ of rational functions $Z(T,\xi)$ whose denominators
are products of factors of the form $1-\Lef^a T^b$ for $b\geq a$.
The point is  that the term corresponding to
the parameter~$\xi=0$ is the one which involves the largest number
of such factors with $a=b$; intuitively, 
the ``order of the pole of $Z(T,\xi)$ at $T=\Lef^{-1}$'' 
is larger for $\xi=0$
than for $\xi\neq 0$. Admitting these facts,
it is therefore a simple matter to conclude the proof
of Theorem~\ref{theo.main}. 

The proof of these facts are the subject of Sections~\ref{s.oscill}
and~\ref{sec.local}. In fact, once rewritten as a motivic  integral,
the Laurent series $Z(T,0)$ is a kind of ``geometric'' motivic
Igusa zeta function. Its analysis, using embedded resolution
of singularities, would be  classical; in fact, our geometric setting
is so strong that we even do not need to resolve singularities in this case.
For general $\xi$, however, what we obtain is a sort of ``motivic
oscillatory integral''. Such integrals are studied in a coordinate
system in Section~\ref{s.oscill}. Finally, in Section~\ref{sec.local},
we establish the three propositions that we had temporarily admitted
in Section~\ref{s.proof}.

\medskip 

In this paper, 
an important role is played by variants of the local zeta functions 
that Igusa had introduced in~\cite{igusa1975}
and which are studied by refining Igusa's initial analysis.
We are honored to dedicate this work to the memory of late Professor Igusa. 
\tcb{The second author had the privilege to first meet Professor Igusa more than thirty years ago. He would like
to acknowledge the profound impact 
of Professor Igusa's vision 
on his own research during all these years.}


\subsubsection*{Acknowledgments}
The research leading to this paper was initiated during a visit 
of the second author  to the  first author
when he was visiting
the Institute for Advanced Study in Princeton for a year. 
We would like to thank that institution for its warm hospitality.
The first author was supported by the Institut universitaire de France, 
by the National Science Foundation under agreement No. DMS-0635607,
as well as by the \emph{Positive} project of Agence nationale de la recherche,
under agreement ANR-2010-BLAN-0119-01.
The second author
was partially supported  by the European Research Council 
under the European Community's Seventh Framework Programme (FP7/2007-2013) 
/ ERC Grant Agreement nr.~246903 NMNAG.

\def\theequation{\thesubsection.\arabic{equation}}

\section{The motivic Poisson formula of Hrushovski--Kazhdan}
\label{s.poisson}

For the convenience of the reader, we begin this paper
with an exposition of
Hrushovski-Kazhdan's motivic Poisson summation formula.
We follow closely
the relevant sections from~\cite{hrushovski-kazhdan:2009},
but adopt a self-contained geometric language.
In the  rest of the paper, we will make an essential use
of the formalism recalled here.

To motivate the definitions,
let us discuss rapidly the dictionary with the Poisson summation
formula for the adele groups of global fields.
So assume that $F$ is a global field. Let $\AD_F$ be the
ring of adeles of~$F$; it is the restricted product of
the completions $F_v$ at all places~$v$ of~$F$ and 
is endowed with a natural structure of a locally compact
abelian group. The field~$F$ embeds diagonally in~$\AD_F$
and its image is a discrete cocompact subgroup.
Fix a Haar measure~$\mu$ on~$\AD_F$ as well as a non-trivial
character $\psi\colon\AD_F\ra\C^*$.
For every Schwartz-Bruhat function~$\phi$ on~$\AD_F^n$,
its Fourier transform is the function~$\mathscr F\phi$ 
on $\AD_F^n$ defined by
\[\mathscr F\phi(y) =  \int_{\AD_F^n} \phi(x) \psi(xy)\,\mathrm d\mu(x);\]
it is again a Schwartz-Bruhat function.
Moreover, the global Haar measure, additive character
and Fourier transform can be written as products of similar local  objects.
Then, one has
\[\sum_{x\in F^n}\phi(x) 
= \mu(\AD_F/F)^{-n} \sum_{y\in F^n} \mathscr F\phi(y). \]

The motivic Poisson summation formula provides an analogue of this formalism,
when $F$ is the function field of a curve~$C $ 
over an algebraically closed field. Integrals 
belong to the Grothendieck ring of varieties, more precisely,
to a (suitably localized) variant ``with exponentials'' of this ring.
They are constructed using \emph{motivic integration} at
the ``local'' level of completions $F_v$; here $F_v$ is identified
with the field $k\lpar t\rpar$ of Laurent series,
so that $F_v^n$ can be considered as an infinite dimensional
$k$-variety, more precisely, an inductive limit of arc spaces
$t^{-m} k\lbra t\rbra^n\simeq \mathscr L(\Aff^n_k)$.
Motivic Schwartz-Bruhat functions are elements of relative
Grothendieck rings. The  possibility to define the
``sum over $F^n$'' of a motivic function
follows from the fact that it is zero outside of a finite
dimensional subvariety of this ind-arc space.
The Poisson summation formula then appears as a reformulation
of the Riemann--Roch theorem for curves combined with
the Serre duality theorem, as formulated in~\cite{serre60}.

\Subsection{The Grothendieck ring of varieties with exponentials}

\paragraph{}
Let $k$ be a 
field. 
The \emph{Grothendieck group of varieties} $\GVar_k$ is defined
by generators and relations; generators are $k$-varieties~$X$
(=$k$-schemes of finite type); relations are the following:
\[ X-Y, \]
whenever
$X$ and $Y$ are isomorphic $k$-varieties;
\[ X-Y-U, \]
whenever $X$ is $k$-variety, $Y$ a closed subscheme of~$X$ and $U=X\setminus Y$
is the complementary open subscheme.
Every $k$-constructible set~$X$ has a class~$[X]$ in the group~$\GVar_k$.

The \emph{Grothendieck group of varieties
with exponentials} $\GExpVar_k$ is defined by generators and relations
(cf. \cite{cluckers-loeser:2010,hrushovski-kazhdan:2009}).
Generators are pairs $(X,f)$,
where $X$ is a $k$-variety 
and $f\colon X\ra\Aff^1=\Spec(\Z[T])$ is a morphism.
Relations are the following:
\[ (X,f)-(Y,f\circ u) \]
whenever $X$, $Y$ are $k$-varieties, $f\colon X\ra\Aff^1$
a morphism, and $u\colon Y\ra X$ a $k$-isomorphism;
\[ (X,f)-(Y,f|_Y)-(U,f|_U) \]
whenever $X$ is a $k$-variety, $f\colon X\ra\Aff^1$ a morphism,
$Y$ a closed  subscheme of~$X$ and $U=X\setminus Y$ the complementary
open subscheme;
\[ (X\times_\Z \Aff^1,\pr_2) \]
where $X$ is a $k$-variety and $\pr_2$ is the second projection.
We will write~$[X,f]$ to denote the class in $\GExpVar_k$ 
of a pair~$(X,f)$.

There is a morphism of Abelian groups
$\iota\colon \GVar_k \ra \GExpVar_k$
which sends
the class of~$X$ to the class~$[X,0]$.

Any pair $(X,f)$ consisting of a constructible set~$X$ 
and of a piecewise morphism $f\colon X\ra\Aff^1$
has a class $[X,f]$ in $\GExpVar_k$.

\paragraph{}
One endows $\GVar_k$ with a ring structure by setting
\[ [X] [Y]= [X\times_k Y] \]
whenver $X$ and~$Y$ are $k$-varieties. The unit element
is the class of the point~$\Spec (k)$.

One endows $\GExpVar_k$ with a ring structure by setting
\[ [X,f][Y,g] = [X\times_k Y, \pr_1^* f+\pr_2^*g], \]
whenever $X$ and~$Y$ are $k$-varieties, $f\colon X\ra\Aff^1$
and $g\colon Y\ra\Aff^1$ are $k$-morphisms;
$\pr_1^*f+\pr_2^*g$ is the morphism from~$X\times_k Y$ to~$\Aff^1$
sending $(x,y)$ to $f(x)+g(y)$.
The unit element for this ring structure 
is the class $[\Spec (k),0]=\iota([\Spec (k)])$.

The morphism $\iota\colon \GVar_k\ra\GExpVar_k$
is a morphism of rings.

One writes $\Lef$ for the class of $\Aff^1_k$ in $\GVar_k$,
or for the class of $ (\Aff^1_k,0)$
in $\GExpVar_k$.
Let $S$ be the multiplicative subset of~$\GVar_k$
generated by $\Lef$ and the elements $\Lef^n-1$, for $n\geq 1$.
The localizations of the rings $\GVar_k$ and $\GExpVar_k$
with respect to~$S$ are denoted $\GVarM_k$ and $\GExpVarM_k$ respectively.
There is a morphism of rings $\iota\colon \GVarM_k\ra\GExpVarM_k$.

\begin{lemm}[{\cite[Lemma 3.1.3]{cluckers-loeser:2010}}]\label{lemm.injective}
The two ring morphisms $\iota \colon \GVar_k\ra\GExpVar_k$
and $\iota\colon \GVarM_k \ra \GExpVarM_k$  are injective.
\end{lemm}
\begin{proof}
For $t\in\Aff^1_k(k)$, let $j_t$ be the map
that sends a pair $(X,f)$ to the class in~$\GVar_k$
of the $k$-variety $[f^{-1}(t)]$.
One observes that $j_0-j_1$ defines a morphism
of groups $j\colon \GExpVar_k\ra \GVar_k$.
Indeed, for every $t\in\Aff^1_k(k)$, 
$j_t$ maps the additivity relations in $\GExpVar_k$ 
to additivity relations in $\GVar_k$.
Moreover, $j_t(Y\times\Aff^1_k,\pr_2)=[Y]$ for every $k$-variety~$Y$,
so that $j((Y\times\Aff^1_k,\pr_2))=0$.
This proves the existence of~$j$.
\tcr{By construction, $\iota$  is a section of~$j$,
hence} $\iota$ is injective.
\end{proof}

\begin{lemm}\label{lemm.vanish}
Let $X$ be a $k$-variety with a $\ga$-action and
let $f\colon X\ra\Aff^1$ be a morphism.
Let $\bar k$ be an algebraic closure of~$k$.
Assume that $f(t+x)=t+f(x)$
for every $t\in\ga(\bar k)$ and every $x\in X(\bar k)$.
Then, the class of~$(X,f)$ is zero in~$\GExpVar_k$.
\end{lemm}
\begin{proof}
By a theorem of Rosenlicht~\cite{rosenlicht1963},
there exists a $\ga$-stable dense open subset~$U$
and a quotient map $U\ra Y$ which is a $\ga$-torsor.
Every such torsor is locally trivial for the Zariski topology.
Consequently,
up to shrinking~$U$ (and~$Y$ accordingly),
this $\ga$-torsor is trivial, 
so that there exists a $\ga$-equivariant
isomorphism $u\colon \ga\times Y\simeq U$.
Let $g\colon Y\ra \Aff^1$ be the morphism given by $y\mapsto f(u(0,y))$.
For $y\in Y(\bar k)$ and $t\in \Aff^1_k(\bar k)$, 
one has  $f(u(t,y))=f(t+u(0,y))=t+f(u(0,y))$.
This shows that the class of~$(Y,f\circ u)$ equals
the product of the classes of~$(\Aff^1_k,\Id)$ and~$(Y,g)$.
It is zero in $\GExpVar_k$, so that the class
of $(U,f|_U)$ is zero too. One concludes the proof by Noetherian
induction.
\end{proof}

\paragraph{Relative variant}
Let $S$ be a $k$-variety. There are similar
rings $\GVar_S$, $\GExpVar_S$, $\GVarM_S$
and $\GExpVarM_S$ defined by replacing $k$-varieties
by $S$-varieties in the above definitions.
We write $[X,f]_S\in\GExpVar_S$ for the class
of a pair $(X,f)$, where $X$
is an $S$-variety and $f\colon X\to\Aff^1$ is a morphism.

Any morphism $u\colon S\ra T$  of $k$-varieties induces
morphisms $u_!$ and $u^*$ between the corresponding Grothendieck groups.
The definitions are similar; let us explain the case of $\GExpVar$.

Let $X$ be an $S$-variety and let $f\colon X\to \Aff^1$ be a morphism.
Via the morphism $u\colon S\to T$, we may view $X$ as a $T$-variety,
so that $(X,f)$ gives rise to a class $[X,f]_T$ in $\GExpVar_T$.  
This induces a morphism of groups
\[ u_!\colon \GExpVar_S\ra \GExpVar_T , \qquad [X,f]_S\mapsto [X,f]_T. \]
If $u$ is an immersion, then $u_!$ is a morphism of rings.

In the other direction, there is a unique morphism of rings
\[ u^*\colon \GExpVar_T\ra \GExpVar_S \]
such that $u^*([X,f]_T)=(X\times_T S,f\circ \pr_1)$ 
for every pair $(X,f)$ consisting of  a $T$-variety~$X$ and
of a morphism $f\colon X\to\Aff^1$.

\begin{rema}
Let $A=\Z[T]$ and $B$ be the localization of~$A$ with
respect to the multiplicative subset generated
by $T$ and the~$T^n-1$, for $n\geq 1$.
The unique ring morphism from~$A$ to~$\GVar_k$
which sends~$T$ to~$\Lef$ endowes~$\GVar_k$ and $\GExpVar_k$
with structures of $A$-algebras. Moreover,
$\GVarM_k\simeq B\otimes_A \GVar_k$ and
$\GExpVarM_k\simeq B\otimes_A \GExpVar_k$.

More generally, for every $k$-variety~$S$, $\GVar_S$ and $\GExpVar_S$
are $A$-algebras, and one has natural isomorphisms
\[ \GVarM_S\simeq B\otimes_A \GVar_S \simeq \GVarM_k\otimes_{\GVar_k} \GVar_S\]
and
\[ \GExpVarM_S\simeq B\otimes_A \GExpVar_S \simeq \GExpVarM_k\otimes_{\GExpVar_k} \GExpVar_S. \]

Thanks to this remark, we will often allow ourselves to write formulas
or proofs at the level of $\GExpVar_S$, when the generalization
to $\GExpVarM_S$ follows directly by localization.
\end{rema}

\paragraph{Functional interpretation of the relative Grothendieck rings}
Elements of $\GExpVar_S$ can be thought of as \emph{motivic 
functions} with source~$S$.
In particular, for $\phi\in \GExpVar_S$
and a point $s\in S$, 
considered as a morphism $\Spec (k(s))\ra S$,
one writes $\phi(s)$ for the
element~$s^*\phi$ of $\GExpVar_{k(s)}$.
By Lemma~\ref{lemm.motivic-function} below, a motivic function
is determined by its values.

Let $u\colon S\to T$ be a morphism of $k$-varieties.
The ring morphism $u^*\colon\GExpVar_T\to\GExpVar_S$ 
then corresponds to composition of functions.

If $u$ is an immersion, the morphism of rings
$u_!\colon\GExpVar_S\to\GExpVar_T$ corresponds in this
interpretation to extension by zero. In the general case,
we shall see that it corresponds to ``summation over rational points''
in the fibers of~$u$.

\begin{lemm}\label{lemm.motivic-function}
Let $\phi\in \GVar_S$
(\resp $\GVarM_S$, \resp $\GExpVar_S$, \resp $\GExpVarM_S$).
If $\phi(s)=0$ for every $s\in S$, then $\phi=0$.
\end{lemm}
As a corollary, Lemmas~\ref{lemm.injective} and~\ref{lemm.vanish} hold 
for relative Grothendieck groups.

\begin{proof}
We give the proof for $\GExpVar_S$, the other three cases are similar.
Let us fix a representative~$M$ of~$\phi$ in $\Z[\ExpVar_S]$,
the free Abelian group generated by pairs $(X,f)$, 
where $X$ is an $S$-scheme and $f\colon X\ra \Aff^1$
is a morphism.
Let $s$ be a generic point of~$S$; since $\phi(s)=0$,
the object $M_{k(s)}$ is a linear combination of 
elementary relations.
By spreading out the varieties and the morphisms
expressing these relations, there exists a dense open subset~$U$
of~$S$ such that the object $M_U$ in~$\Z[\ExpVar_U]$
is a linear combination of the corresponding elementary
relations, \tcb{hence one has $[M_U]=0$. On the other hand,
we have $[M_T] (s) =0$ for every point $s$ in 
$T=S\setminus U$. By Noetherian descending induction it follows that
 $[M_T]=0$. Thus $[M]=0$, and $\phi=0$.}
\end{proof}

\paragraph{Exponential sums}
The class~$\theta$ of a pair $(X,f)$  in~$\GExpVar_k$ 
can be thought of as an analogue of the exponential sum
\[ \sum_{x\in X(k)} \psi(f(x)), \]
when $k$ is a finite field and
$\psi\colon k\ra\C^*$ is a fixed {non-trivial} additive character.
This justifies the \emph{notation} $\sum_{x\in X} \psi(f(x))$
for the class~$[X,f]$ in~$\GExpVar_k$.

More generally, let $S$ be a $k$-variety,
let~$\theta\in\GExpVarM_S$ and 
let $u\colon S\ra\Aff^1$ be a morphism.
We \emph{define}
\begin{equation}\label{eq.1.1.10}
  \sum_{s\in S} \theta(s)\psi(u(s)) = \theta\cdot [S,u]_S, 
\end{equation}
the product being taken in~$\GExpVarM_S$, and its
result being viewed in~$\GExpVarM_k$.
Let us make this definition explicit, assuming that $\theta=[X,f]_S$,
where $X$ is a $S$-variety and $f\colon X\ra\Aff^1$ is a morphism; 
in this case,
\[
  \sum_{s\in S} \theta(s)\psi(u(s)) 
= [X,f]_S [S,u]_S 
= [ X\times_S S , f\circ\pr_1 + u\circ\pr_2]_S
= [ X, f+u\circ g].
\]
To support this notation, observe that when $k$ is a finite field
and $s\in S(k)$, \tcb{denoting by $g$ the morphisms $X\to S$}, one has
\[ \theta(s) = \sum_{\substack{x\in X(k) \\ g(x)=s}} \psi(f(x)), \]
so that
\begin{align*}
 \sum_{s\in S(k)} \theta(s) \psi(u(s)) & = 
\sum_{s\in S(k)} \left(\sum_{\substack{x\in X(k) \\ g(x)=s} }
            \psi(f(x)) \right)  \psi(u(s)) \\
& = \sum_{x\in X(k)}  \psi(f(x)+u(g(x)) . \end{align*}

Let $u\colon S\to T$ be a morphism of $k$-varieties.
This notation of ``summation over rational points'' 
is consistent with the functional interpretation  of
the morphism $u_!\colon\GExpVar_S\to\GExpVar_T$.
Indeed, for every $\phi\in\GExpVar_S$ and every $t\in T$, one has
\[ u_!\phi( t) = \sum_{s\in u^{-1}(t)} \phi(s),\]
\tcb{with notation similar to (\ref{eq.1.1.10}).}

\begin{lemm}\label{lemm.linear}
Let $V$ be a finite dimensional
$k$-vector space, let $f$ be a linear form
on~$V$. Then, 
\[ \sum_{x\in V} \psi(f(x))  = \begin{cases}
  \Lef^{\dim(V)} & \text{if $f=0$;} \\ 0 & \text{otherwise.} \end{cases}
\]
\end{lemm}
\begin{proof}
By definition, 
the left hand side is the class of $(V,f)$
in $\GExpVar_k$.
This equals $[V]=\Lef^{\dim(V)}$ if $f=0$.
Otherwise, let $a\in V$ be such that $f(a)=1$
and let us consider the action of the additive
group~$\ga$ on~$V$ given by $(t,v)\mapsto v+ta$.
Since $f(v+ta)=f(v)+t$, it follows from 
Lemma~\ref{lemm.vanish} that $[V,f]=0$.
\end{proof}

\Subsection{Local Fourier transforms}

\paragraph{Schwartz-Bruhat functions of given level and their integral}
\label{ss-sbfunctions}
Let $F^\circ$ be a complete discrete valuation
ring, with field of fractions~$F$ and perfect residue field~$k$;
we write $\ord\colon F\to\Z$ for the (normalized) valuation on~$F$.
We assume that $F$ and~$k$ have the same characteristic;
let us fix a section of the morphism $F^\circ\to k$,
so that $F^\circ$ is a $k$-algebra.
Every local parameter in~$F$, \ie, every element $t\in F$ of valuation~$1$,
then gives rise to  isomorphisms 
$k\lbra t\rbra\simeq F^\circ$ and $k\lpar t\rpar \simeq F$. 

Fix such a local parameter~$t$.
For every two integers $M \leq N$, we can identify the quotient
set $\{x\sozat \ord(x)\geq M\}/\{x\sozat \ord(x)\geq N\}=
t^M F^\circ/t^{N}F^\circ$
of the elements~$x$ in~$F$ satisfying $\ord(x)\geq M$,
modulo those satisfying~$\ord(x)\geq N$,
with the $k$-rational points of the affine space 
$\Aff_k^{(M,N)}=\Aff^{N-M}_k$, via the formula
\[ x=\sum_{j=M}^{N-1} x_i t^i \pmod {t^{N}} \mapsto (x_{M},\dots,x_{N-1}). \]

For every integer~$n\geq 0$,
let $\mathscr S(F^n;M,N)$ be the ring $\ExpMot_{\Aff_k^{n(M,N)}}$;
its elements are called \emph{motivic Schwartz-Bruhat function} 
of level~$(M,N)$ on~$F^n$.
We define the integral of such a function~$\phi\in\mathscr S(F^n;M,N)$
by the formula
\begin{equation}
 \int_{F^n} \phi(x)\,\mathrm dx = \Lef^{-nN} \sum_{x\in\Aff_k^{n(M,N)}} \phi(x) .\end{equation}

\paragraph{Compatibilities}
\label{ss-sbfunctions-general}
The natural injection $t^M F^\circ/t^NF^\circ\ra t^{M-1}F^\circ/t^N F^\circ$
is turned into a closed immersion 
\[ \iota\colon \Aff_k^{(M,N)}\ra\Aff_k^{(M-1,N)},
\quad (x_M,\dots,x_{N-1}) \mapsto (0,x_M,\dots,x_{N-1}). \]
This gives rise to ring morphisms 
$\iota^*\colon\mathscr S(F^n;M-1,N)\ra\mathscr S(F^n;M,N)$
(\emph{restriction}) and 
$\iota_*\colon \mathscr S(F^n;M,N)\ra\mathscr S(F^n;M-1,N)$
(\emph{extension by zero}).  
One has $\iota^*\iota_*=\Id$.

Similarly, the natural projection $t^M F^\circ/t^{\tcb{N+1}} F^\circ\ra t^M F^\circ/t^{\tcb{N}}F^\circ$ induces a morphism 
\[ \pi\colon \Aff_k^{(M,N+1)}\ra\Aff_k^{(M,N)}, \quad
 (x_M,\dots,x_N) \mapsto (x_M,\dots,x_{N-1}) \]
which is a trivial fibration with fiber~$\Aff^1_k$.
This gives rise to a ring morphism 
$\pi^*\colon\mathscr S(F^n;M,N)\ra\mathscr S(F^n;M,N+1)$
and to a group morphism 
$\pi_*\colon\mathscr S(F^n;M,N+1)\to\mathscr S(F^n;M,N)$
(\emph{integration over the fiber}).
One has $\pi_*\pi^*(\phi)=\Lef^n\phi$ for every $\phi\in\mathscr S(F^n;M,N)$.

The space of  \emph{motivic smooth functions}  on~$F^n$
is then defined by
\begin{equation}
 \mathscr D(F^n)
 = \varprojlim_{M,\iota^*} \varinjlim_{N,\pi^*} \mathscr S(F^n;M,N) ,
\end{equation}
while the space of \emph{motivic Schwartz-Bruhat functions} on~$F^n$ is defined by
\begin{equation}
 \mathscr S(F^n)
 = \varinjlim_{M,\iota_*}\varinjlim_{N,\pi^*} \mathscr S(F^n;M,N). 
\end{equation}
These spaces have a ring structure, but $\mathscr S(F^n)$ has no unit element; 
the natural injection $\mathscr S(F^n)\subset\mathscr D(F^n)$ 
is a morphism of rings.
We denote by $\mathbf 1_{(F^\circ) ^n}$ the class in~$\mathscr S(F^n)$ of
the unit element of $\mathscr S(F^n; 0,0)$.

Observe that $\iota_*$ commutes with the sum over points,
while $\pi^*$ only commutes up to multiplication by~$\Lef^n$. 
Consequently, the integral of a Schwartz-Bruhat function
does not depend on the choice of a level~$(M,N)$ at which it is defined.
This gives rise to an additive map $\mathscr S(F^n)\ra \ExpMot_k$,
denoted $\phi\mapsto \int_{F^n}\phi$.
For every subset $W$ of~$F^n$ whose characteristic function~$\mathbf 1_W$
is a motivic Schwartz-Bruhat function, one also writes
$\int_W \phi=\int_{F^n}\phi\mathbf 1_W$.

\paragraph{The Fourier kernel}
Let $r\colon F\ra k$ a non-zero $k$-linear map
which vanishes on $t^{a}F^\circ$ for some integer~$a$.
We define  the conductor $\nu$ of~$r$ as the smallest integer~$a$
such that $r$ vanishes on~$t^{a}F^\circ$.

In the sequel, our main source of such a linear form
will be  given by residues of differential forms.
Assume that $F$ is the completion at
a closed point~$s$ of a function field in one variable over~$k$, 
and let $\res_s\colon \Omega_{F/k}\to k$ be the residue map 
at the closed point~$s$ (\cite{tate1968}, p.~154).
Then fix some non-zero meromorphic differential form $\omega\in\Omega_{F/k}$
and set $r_s\colon F\to k$, $x\mapsto \res_s(x\omega)$.
In this case, the conductor of~$r$ 
is equal to the order of the pole of~$\omega$
(Theorem~2 of~\cite{tate1968}; see also~\S\ref{residue} below).

The kernel of the Fourier transform is the element
of~$\mathscr D(F^2)$ 
informally written 
\[ (x,y)\mapsto \psi(r(xy))= \mathrm e(xy) . \]
Let us explicit this definition.

Let $x\in F$, let us write $x=\sum_n x_n t^n$, where $x_n=0$
for $n<\ord(x)$. One has $r(x)=\sum\limits_{n=\ord(x)}^{\nu-1} x_n r(t^n)$.
Consequently, restricted to the subset of~$F$
consisting of elements~$x$ such that $\ord(x)\leq M$,
$r$ can be interpreted as a linear morphism  $r^{(M,N)}\colon
\Aff_k^{(M,N)}\ra\Aff^1_k$,
for every integer~$N$ such that $N\geq \nu$.

Let $N''=M+M'+\min(N-M,N'-M')=\min(M'+N,M+N')$.
The product map $F\times F\ra F$ gives rise to a morphism
\[ \Aff_k^{(M,N)}\times \Aff_k^{(M',N')}
            \ra \Aff_k^{(M+M',N'')} . \]
Let us assume that $N''\geq \nu$.
Composing with $r^{(M+M',N'')}$, we obtain a morphism
\[  \Aff_k^{(M,N)}\times\Aff_k^{(M',N')} \to \Aff^1_k,\]
hence an element 
of $\ExpMot_{\Aff_k^{(M,N)}\times\Aff_k^{(M',N')}}$,
whose class in~$\mathscr D(F^2)$ is our kernel.

\paragraph{Fourier transformation}
The Fourier transform 
of a Schwartz-Bruhat function $\phi\in\mathscr S(F;M,N)$ 
is defined formally as
\[ \mathscr F\phi(y) = \int_F \phi(x)\mathrm e(xy)\,\mathrm dx. \]
More generally, we write $\langle x,y\rangle=\sum_{j=1}^nx_jy_j$ for the 
self-pairing of $F^n$ and define the Fourier transform of 
a Schwartz-Bruhat function $\phi\in\mathscr S(F^n;M,N)$  by
\[ \mathscr F\phi(y) 
= \int_{F^n} \phi(x)\mathrm e(\langle x,y\rangle)\,\mathrm dx, \]
where, we recall,
$\mathrm e(\cdot)$ is a short-hand notation for $\psi(r(\cdot))$.

Observe that $\phi\mapsto \mathscr F\phi$ is
$\GExpVarM_k$-linear.

Let us make the definition explicit, assuming
that $n=1$, $\phi$ is of the form~$[U,f]$,
where $(U,g)$ is a $\Aff_k^{(M,N)}$-variety and $f\colon U\ra\Aff^1$
is  a morphism. 
Then,  $ \mathscr F\phi $ is represented by
\[  \Lef^{-N} [U\times_g \Aff_k^{(M,N)}\times\Aff_k^{(M',N')}, f(u)+r(xy)] \]
in the Grothendieck group $\ExpMot_{\Aff_k^{(M',N')}}$, where 
we define $U\times_g \Aff_k^{(M,N)}\times\Aff_k^{(M',N')}$ 
as the fiber product of the $\Aff_k^{(M,N)}$-varieties
$(U,g)$ and $(\Aff_k^{(M,N)}\times\Aff_k^{(M',N')},\pr_1)$,
viewed as an $\Aff_k^{(M',N')}$-variety,
the structural morphism 
\[ U\times_g \Aff_k^{(M,N)}\times \Aff_k^{(M',N')}
\ra\Aff_k^{(M',N')} \] being the projection to the third factor.
For this to make sense, we only need to take \tcr{$M'\leq \nu-N$}
and $N'\geq \nu-M$.

\begin{prop}\label{prop.local-vanishing}
Let $\nu$ be the conductor of~$r$.
Then for  every $\phi\in \mathscr S(F^n;M,N)$, 
one has \mbox{$\mathscr F\phi \in\mathscr S(F^n;\nu-N,\nu-M)$.}
\end{prop}

\begin{theo}[Fourier inversion]
\label{theo.fourier-local}
Let $\nu$ be the conductor of~$r$.
Then for  every $\phi\in \mathscr S(F^n;M,N)$, 
one has $\mathscr F\mathscr F\phi(x)=\Lef^{-n\nu}\phi(-x)$.
\end{theo}
\begin{proof}
For simplicity of notation, we assume that $n=1$.
We may assume that $\phi$ is represented by $[U,f]$ as above.
To compute $\mathscr F\mathscr F\phi$, we may set
$(M',N')=(\nu-N,\nu-M)$ and $(M'',N'')=(M,N)$, so that
$\mathscr F\mathscr F\phi$
is represented
by 
\[ \Lef^{-N-N'}[U\times_g \Aff_k^{(M,N)}\times\Aff_k^{(M',N')}\times \Aff_k^{(M'',N'')}, f(u)+r(xy)+r(yz)].\]
The contribution of the part where $x+z\neq 0$ is zero,
because of Lemma~\ref{lemm.vanish}.
The part where $x+z=0$ is  equal to
\begin{align*} 
\Lef^{-N-N'} [U\times_g\Aff_k^{(M',N')}\times\Aff_k^{(M'',N'')},f(u)] & =
\Lef^{-N-M'} [U\times_g\Aff_k^{(M'',N'')},f(u)] \\
&=
  \Lef^{-\nu} [U,f] ,\end{align*}
where $U$ is viewed as an $\Aff_k^{(M'',N'')}$-variety
via the morphism $-g$. This proves the theorem.
\end{proof}

\paragraph{}
This theory is extended in a straightforward way to products of local fields.
Let $(F_s)_{s\in S}$ be a finite family of fields as above,
fields of fractions of complete discrete valuation rings~$F_s^\circ$,
with local parameters~$t_s$ and residue fields~$k_s$.
Assume that for each~$s$, $k_s$ is a finite extension of~$k$.
In practice, one will start from the function field~$F$ 
of a (projective, smooth, geometrically
connected) $k$-curve~$C$, $S$ will be a set of closed points of~$C$,
and for every~$s\in S$, the field~$F_s$ will be the completion of~$F$ 
at the point~$s$.

For every $s\in S$, write $\Res_{k_s/k}$ for the functor of Weil restriction
of scalars; one has $\Res_{k_s/k}(\Aff^m_{k_s})\simeq \Aff^{m[k_s:k]}_k$.
For every family $(M_s,N_s)_{s\in S}$ of integers such that $M_s \leq N_s$,
one then sets
\[ V(n(M_s,N_s)) = \prod_{s\in S} \Res_{k_s/k} \Aff^{n(N_s-M_s)}_{k_s}, \]
and  defines the spaces of Schwartz-Bruhat functions
of levels $(M_s,N_s)$ on $\prod_{s\in S} F_s^n$ by 
\begin{equation}
\mathscr S (\prod_{s\in S} F_s^n; (M_s,N_s)) = \ExpMot_{V(n(M_s,N_s))} .
\end{equation}
One then sets
\begin{equation}
 \mathscr S(\prod_{s\in S} F_s^n) = \varinjlim_{(M_s,\iota_*)}\varinjlim_{(N_s,\pi^*)} \mathscr S (\prod_{s\in S} F_s^n; (M_s,N_s)).
\end{equation}

There is a natural morphism of rings 
\[ \bigotimes_{s\in S} \mathscr S(F_s^n) \to \mathscr S (\prod_{s\in S}F_s^n) .\]
Contrary to the classical arithmetic case, it is not surjective in general.

The definition of the integral extends to a linear map
$\mathscr S(\prod_{s\in S} F_s^n) \ra \ExpMot_k$.

Let $(r_s\colon F_s\to k)$ be a family of non-trivial \tcb{$k$-linear} maps, let $\nu_s$
be the conductor of~$r_s$.
Then the definition of the Fourier Transform $\mathscr F$
extends naturally to $\mathscr S(\prod_{s\in S}F_s^n)$.
For every $\phi\in\mathscr S(\prod_{s\in S}F_s^n)$, one sets
\[ \mathscr F\phi((y_s)) = \int_{\prod_{F_s^n}} \phi(x) e(\sum \langle x_s,y_s\rangle) \prod\mathrm dx_s. \]
Fourier inversion still holds, with the same proof:
\[ \mathscr F\mathscr F\phi(x) = \Lef^{-\tcb{n}\sum [k_s:k]\nu_s} \phi(-x).
\]

\Subsection{Global Fourier Transforms}

\paragraph{}
Let $k$ be a perfect field,
let $C $ be a projective, geometrically connected, smooth curve over~$k$,
and let $F=k(C )$ be its field of functions.
We fix a non-zero meromorphic 
differential form~$\omega\in\Omega^1_{F/k}$.

One can interpret the field $F=k(C )$ as the $k$-points
of an ind-$k$-variety. The simplest way to do so consists
maybe in considering the 
family 
of all Riemann--Roch
spaces $\mathscr L(D)=\mathrm H^0(C , \mathscr O(D))$, 
indexed by effective divisors~$D$ on~$C $.
Concretely, $\mathscr L(D)$ is the set
of non-zero rational functions~$f$ on~$C $  such that 
$\div(f)+D\geq 0$, together with the~$0$ function. It is
a finite dimensional $k$-vector space and we 
view it as a $k$-variety.
The natural  inclusions from $\mathscr L(D)$ to~$\mathscr L(D')$,
where $D$ and~$D'$ are effective divisors such that $D'-D$ is effective,
give this family the structure of an inductive system, the limit
of which is interpreted as~$k(C )$.

\paragraph{Global Schwartz-Bruhat functions}\label{par.global-sb}
For every closed point~$s\in C$, write $\ord_s$ for the corresponding
normalized valuation on~$F$, $F_s$ for the its completion,
and $F_s^\circ$ for the valuation ring of~$F_s$;
we also fix a local parameter $t_s$ at~$s$.

The adele ring~$\AD_F$ of~$F$ is the subring of $\prod_{s\in C} F_s$
consisting of families $(x_s)$ such that $x_s\in F_s^\circ$
for all but finitely many $s$. (By abuse of notation,
the condition ``$s\in C$'' means that $s$ belongs to the set of
closed points of~$C$.)

In the classical arithmetic setting,
the ring~$\AD_F$ has a locally compact totally disconnected 
topology, and the space of Schwartz-Bruhat functions on~$\AD_F^n$  is
the ring of  real valued locally constant with compact support on~$\AD_F^n$. 

We now describe its geometric analogue $\mathscr S(\AD_F^n)$.

Let $S$ and $S'$ be finite sets of closed points of~$C$ 
such that $S\subset S'$.
There is a natural morphism of rings:
\[ j_S^{S'}\colon  \mathscr S(\prod_{s\in S} F_s^n) \to \mathscr S(\prod_{s\in S'} F_s^n), \quad \phi\mapsto \phi\otimes \bigotimes_{s\in S'\setminus S} \mathbf 1_{(F_s^\circ)^n}. \]
The ring $\mathscr S(\AD_F^n)$ of 
global motivic Schwartz-Bruhat function on~$\AD^n_F$ is defined by
\[ \mathscr S(\AD_F^n) = \varinjlim_{S\subset C,j_S^{S'}} \mathscr S(\prod_{s\in S}F_s^n). \]

It is important to observe that the global motivic Schwartz-Bruhat functions
on~$\AD_F^n$ induce the characteristic function of $(F_s^\circ)^n$ 
at all but finitely many closed points $s\in C$. This is a notable
difference with the arithmetic setting.

\tcr{\paragraph{Simple functions}}\label{par.simple-functions}
In the classical arithmetic case, 
simple functions are  characteristic
functions of a ball, or of products of balls. 
Let us describe their analogues in the motivic setting.
Let $S$ be  a finite subset of closed points of~$C$,
let $a=(a_s)_{s\in S}\in \prod_{s\in S}F_s$,
let $(M_s,N_s)_{s\in S}$ be a family of pairs of integers
such that $\ord(a_s)\geq M_s$ for every $s\in S$.
Let $W=\prod_{s\in S}\Res_{k_s/k}\Aff_{k_s}^{n(M_s,N_s)}$,
let $W_a=\Spec (k)$ and let $W_a\to W$  be the canonical map
induced, for every $s\in S$, by the $t_s$-adic expansion of~$a_s$.
The motivic function on~$W$ associated with the pair $(W_a\ra W,0)$
is called a \emph{simple function}. The corresponding
Schwartz-Bruhat function on~$\prod_{s\in S}F_s^n$
represents the characteristic
function of the product of the balls of centers~$a_s$
and radius $N_s$ in~$F_s^n$.

More generally, let us consider a $k$-variety~$Z$ and 
a morphism $u=(u_s)\colon Z\to W$; 
let $\phi\in\ExpMot_{W\times_k Z}$ be  
the motivic function associated with $(Z,0)$, where
$Z$ is considered as a $W\times_kZ$-variety through
the morphism $u\times\Id_Z$.
For each $z\in Z$, we write $\phi_z$ for the motivic function
on~$W_{k(z)}$ deduced from~$\phi$. When $z\in Z(k)$,
the corresponding Schwartz-Bruhat function on $\prod_{s\in S}F_s^n$
represents the characteristic function of the product
of the polydiscs of radius~$N_s$ and centers $u_s(z)$.
Consequently, we call~$\phi$ 
a \emph{family of simple functions parameterized by the
$k$-variety~$Z$}.

Let $\chi\in\ExpMot_Z$ be a motivic function on~$Z$, represented 
by $[X\xrightarrow g Z,f]_Z$, 
where $X$ is a $Z$-variety
and $f\colon X\ra\Aff^1$ is a morphism.
We then define the Schwartz-Bruhat  function 
$\sum_{z\in Z}\chi(z)\phi_z$
on~$\AD^n_F$
as the one represented by the pair
$[X\xrightarrow {u\circ g} W, f]$.
By linearity, this definition is extended to every element~$\chi$
of~$\ExpMot_Z$.

\begin{lemm}\label{lemm.simple}
Any global Schwartz-Bruhat function on $\AD^n_F$
can be written in this way.
\end{lemm}
\begin{proof}
Let $\Phi$ be a global Schwartz-Bruhat function on $\AD^n_F$,
represented by a pair $[Z,f]$, where 
$Z$ is a variety over $W=\prod_s \Res_{k_s/k}\Aff_{k_s}^{n(M_s,N_s)}$,
for some finite set~$S$ of closed points of~$C$ and integers $(M_s,N_s)$,
and $f\colon Z\ra\Aff^1$.
Let $\phi$ be the family of simple functions
parameterized by~$W$ given by the pair $(W\xrightarrow{\Id} W,0)$.
One checks readily that 
$\Phi = \sum_{w\in W} \Phi\phi_w  $.
\end{proof}

\paragraph{Summation over rational points}
Let $\phi$ be a global Schwartz-Bruhat function on~$\AD^n_F$,
represented by a class $\phi_S$ 
in~$\ExpMot_k\big(\prod_{s\in S}\Res_{k_s/k}\Aff_{k_s}^{n(M_s,N_s)}\big)$,
for some finite set $S$ of closed points of~$C$
and some family $(M_s,N_s)_{s\in S}$.

Consider the divisor~$D=-\sum M_s [s]$ on~$C $.
For every $s\in S$, the natural embedding of~$F=k(C )$
into the field~$F_s$ maps $\mathscr L(D)$ into $t^{M_s}F_s^\circ$.
This gives rise to a  morphism of algebraic varieties
$\alpha\colon \mathscr L(D)^n\ra \big(\prod_{s\in S}\tcb{\Res_{k_s/k}}\Aff_k^{(M_s,N_s)}\big)^n$.
We then 
define $\sum_{x\in F^n}\phi(x)$ as the image
in~$\ExpMot_k$
of the element  $\alpha^* \phi_S $ of $ \ExpMot_{\mathscr L(D)^n}$.
It does  not depend on the choice of the set~$S$
nor on the choice of the integers~$(M_s,N_s)$
and of the class~$\phi_S$.

Let us give a more explicit formula, 
assuming that $\phi_S$ is of the form $[X,f]$,
where $W= \prod_{s\in S}\Res_{k_s/k}\Aff_{k_s}^{n(M_s,N_s)}$,
$X$ is a $W$-variety
and $f\colon X\ra\Aff^1$ is a morphism.
In that case, one has
\begin{equation}
\sum_{x\in F^n}\phi(x)
= [ \mathscr L(D)^n\times_{W} X,f\circ\pr_2] .
\end{equation}

\paragraph{Reminders on residues and duality for curves}\label{residue}
We need to recall a few results
concerning residues, duality and the Riemann--Roch theorem
on smooth curves.

We fix a non-zero meromorphic  differential form $\omega\in\Omega_{F/k}$.
Let $\nu_s$ be the order of the pole,
or minus the order of the zero, of~$\omega$ at~$s$,
and let $\nu$ be the divisor~$\sum\nu_s [s]$ on~$C $.
One has $\deg(\nu)=2-2g$, where $g$ is the genus of~$C $.

For every closed point $s\in C$, we define a map $r_s\colon F_s\to k$
by $r_s(x)=\res_{C,s}(x\omega)$,
where $\res_{C,s}\colon\Omega_{F/k}\to k$ is Tate's residue 
(\cite{tate1968}) on the curve~$C$ at~$s$;
since the field~$k$ is perfect,
it is non-zero and its conductor is equal to~$\nu_s$.
This follows from Theorem~2 of~\cite{tate1968} if $s$ is a rational point
of the curve~$C$; in the general case, one checks that 
\[ \res_{C,s} (\omega) = \Tr_{k(s)/k} ( \res_{C_{k(s)},s}(\omega)), \]
where we indicated the curve as in index.

Let $D$ be a divisor on~$C $.
Let $\mathscr L(D)$ be the set of rational functions~$y\in F^\times$
such that $\div(y)+D\geq 0$ to which we adjoin~$0$; this
is a finite dimensional $k$-vector space.

Let $\Omega(D)$ be the set of 
meromorphic forms $\alpha\in\Omega_{F/k}$
such that $\div(\alpha)\geq D$ (together with $\alpha=0$).
The map $y\mapsto y\omega$ from~$F$ to~$\Omega^1_{F/k}$
induces an isomorphism
$\mathscr L(\div(\omega)-D)\to\Omega(D)$.

We embed~$F$ diagonally in~$\AD_F$.
For every divisor~$D$, let~$\AD_F(D)$ be the subspace of~$\AD_F$
consisting of families $(x_s)$ such that \mbox{$\div(x_s)+\ord_s(D)\geq 0$}
for every closed point~$s\in C$.  
There is an isomorphism of $k$-vector spaces (see~\cite{serre60},
Chapitre~II, \S5, proposition~3; see also~\cite{tate1968},
p.~157)
\[ \mathrm H^1(\mathscr L(D)) \simeq \AD_F/ (\AD_F(D)+F). \]
According to Serre's duality theorem
(\cite{serre60}, Chapitre~II, \S8, théorème~2;
see also~\cite{tate1968}, Theorem~5), 
the morphism
\[ \theta\colon \Omega_{F/k} \ra \Hom(\AD_F,k), \qquad  \alpha\mapsto  \left(
   (x_s)\mapsto \sum_s \res_s(x_s\alpha)\right) \]
identifies
$\Omega(D)$ with the orthogonal of $\AD_F(D)+F$ in~$\Hom(\AD_F,k)$,
\ie, with the dual of $\mathrm H^1(\mathscr L(D))$.
This contains the theorem of residues according to which
\[ \sum_{s\in C} \res_s(x\omega) = 0 \]
for every $x\in F$.

\paragraph{Global Fourier transformation}
Observe that if $s$ is any closed point of~$C$
such that $\nu_s=0$, then $\mathbf 1_{F_s^\circ}$
is its own Fourier transform.
Consequently, we may define the Fourier transform~$\mathscr F\phi$ 
of every global Schwartz-Bruhat function~$\phi\in\mathscr S(\AD_F^n)$
as the image in~$\mathscr S(\AD_F^n)$ 
of~$\mathscr F\phi_S$,
where $S$ is any finite set of places
such that $\nu_s=0$ for $s\not\in S$,   
and $\phi_S\in\mathscr S( \prod_{s\in S}F_s^n)$ is a representative of~$\phi$.
By construction, $\mathscr F\phi$
is itself a global Schwartz-Bruhat function on the ``dual'' 
space~$\AD^n_F$.

\begin{theo}[Fourier inversion formula]
For every $\phi\in\mathscr S(\AD_F^n)$, one has 
\[\mathscr F\mathscr F \phi (x)= \Lef^{n(2g-2)}  \phi(-x). \]
\end{theo}
\begin{proof}
When $\phi$ is a simple function, this is nothing but
the Fourier inversion formula~\ref{theo.fourier-local}.
The general case follows from Lemma~\ref{lemm.simple}.
\tcr{Indeed, if $\phi$ is written as a sum of simple functions $\sum \phi(z)\psi_z$,
it follows from the definitions
that $\mathscr F\phi=\sum\phi(z) \mathscr F\psi_z$,
so that 
\[ \mathscr F\mathscr F\phi(x)=\sum\phi(z) \mathscr F\mathscr F\psi_z(x)
=\sum\phi (z) \Lef^{n(2g-2)} \psi_z(-x)
= \Lef^{n(2g-2)}\phi(-x).\qedhere\]}
\end{proof}

\begin{theo}[Motivic Poisson formula]
\label{theo.poisson}
Let $\phi\in\mathscr S(\AD_F^n)$. 
Then, 
\[ \sum_{x\in F^n} \phi(x) 
= \Lef^{(1-g)n} \sum_{y\in F^n}\mathscr F\phi(y). \]
\end{theo}
\begin{proof}
For simplicity of notation, we assume that $n=1$.
By Lemma~\ref{lemm.simple}, we
may also assume that $\phi$ 
is a simple function $\otimes_{s\in S}\phi_s$, where 
for each $s\in S$,
$\phi_s$ is the characteristic function of the ball
of center $a_s\in F_s$ and radius $N_s$.
Let $D$ be the divisor $\sum_{s\in S} N_s s$ on~$C $.

For every~$s\in S$,
$\mathscr F\phi_s$ is a Schwartz-Bruhat function 
on $F_s$ 
and 
\[ \mathscr F\phi(y_s) = 
\begin{cases}
\psi( \res_s(a_s y_s\omega))\Lef^{-N_s}  
       & \text{if $\ord_s(\tcb{y_s})+\ord_s(D)\geq 0$;} \\
0 & \text{otherwise.}
\end{cases}\]
Then $\mathscr F\phi$ is a global Schwartz-Bruhat function
on $\AD_F$, represented by 
$\bigotimes_{s\in S} \mathscr F\phi_s$
and
\[ \mathscr F\phi(y) = 
\begin{cases}
\psi( \sum_{s\in S}\res_s(a_s y_s\omega))\Lef^{-\deg(D)}  
       & \text{if $\div(y\omega)+D\geq 0$;} \\
0 & \text{otherwise.}
\end{cases}\]
Recall that the map $y\mapsto y\omega$ identifies
$\mathscr L(\div(\omega)+D)$ with  $\Omega(-D)$.
Let $f\colon \mathscr L(\div(\omega)+D)\to k$
be the linear map $y\mapsto \langle \theta(y\omega),(a_s)\rangle$;
it is identically zero if and only if $(a_s)$ belongs
to the orthogonal~$\Omega(-D)^\perp$ of $\Omega(-D)$ with respect
to the Serre duality pairing.
By Lemma~\ref{lemm.linear}, we thus have
\begin{align*}
 \sum_{y\in F} \mathscr F\phi(y)
& =\Lef^{-\deg(D)} \sum_{y\in\mathscr L(\div(\omega)+D)} \psi(f(y))  \\
& =
\begin{cases} \Lef^{-\deg(D)+\dim\mathscr L(\div(\omega)+D)}
& \text{if $(a_s)\in \Omega(-D)^{\perp}$,} \\
0 & \text{otherwise.} 
\end{cases}
\end{align*}
Moreover, the Riemann--Roch formula asserts that
\begin{align*}
 \dim\mathscr L(-D) & = \dim \mathrm H^1(C ,-D)-\deg(D)+1-g \\
& = \dim \Omega(-D) - \deg(D)+1-g \\
& = \dim \mathscr L(\div(\omega)+D)-\deg(D) + 1-g.\end{align*}
Consequently,
\[ \Lef^{1-g} \sum_{y\in F} \mathscr F\phi(y) = 
\begin{cases}
   \Lef^{\dim\mathscr L(-D)} & \text{if $(a_s)\in\Omega(-D)^\perp$,}  \\
0 & \text{otherwise}.
\end{cases}\]

Let us now compute the left hand side of the Poisson formula.
In the case where
\[ (a_s)\in\Omega(-D)^\perp = \AD_F(-D)+F , \]
there exists $a\in F$ such that $\ord_s(a-a_s)\geq N_s$
for all~$s$.
Then, 
\[ \phi(x) = \begin{cases} 1 & \text{if $x-a\in \mathscr L(-D)$,}\\
0 & \text{otherwise} \end{cases} \]
so that
\[ \sum_{x\in F} \phi(x) = \sum_{x\in F}\phi(x-a)
 = \sum_{x\in \mathscr L(-D)} 1
= \Lef^{\dim\mathscr L(-D)}.
\]
In the other case, there does not exist any $a\in k(C )$
such that $\ord_s(a-a_s)\geq N_s$ for all~$s$. Then,
$ \phi(x)=0$ for all $x\in F$ and $\sum_{x\in F}\phi(x)=0$.
In both cases, this concludes the proof of the motivic Poisson formula.
\end{proof}

\begin{rema}
By Fourier inversion, we have
$\mathscr F\mathscr F\phi(x)
=\Lef^{-n\deg(\nu)}\phi(-x)=\Lef^{(2-2g)n}\phi(-x)$.
Consequently, 
if we apply the Poisson formula to $\mathscr F\phi$, 
we obtain
\[ \sum_{y\in F^n} \mathscr F\phi(y)
= \Lef^{(1-g)n} \sum_{x\in F^n} \mathscr F\mathscr F\phi(x)
=\Lef^{(g-1)n} \sum_{x\in F^n} \phi(x), \]
as expected.
\end{rema}

\section{Further preliminaries}\label{s.prelim}

\Subsection{Motivic invariants}
\label{ss.motivic}

Let $k$ be a field.
For every~$m\geq0$, let $\GVar_k^{\leq m}$ be the subgroup of~$\GVar_k$
generated by classes of varieties of dimension~$\leq m$.
If $x\in\GVar_k^{\leq m}$ and $y\in\GVar_k^{\leq n}$, 
then $xy\in\GVar_k^{\leq m+n}$.
Let $(\GVarM_k^{\leq m})_{m\in\Z}$ be the similar filtration
on~$\GVarM_k$; explicitly, $\GVarM_k^{\leq m}$ is generated
by fractions $[X][Y]^{-1}$ where $X$ is a $k$-variety,
$Y$ is a product of varieties of the form $\Aff^1$, $\Aff^a\setminus\{0\}$ (for $a\geq 1$), and $\dim(X)-\dim(Y)\leq m$.
For every class~$x\in\GVarM_k$, let $\dim(x)\in\Z\cup\{-\infty\}$ 
be the infimum of the integers~$m\in\Z$ such that $x\in \GVarM_k^{\leq m}$.
For every $x,y\in\GVarM_k$, 
one has $\dim(x+y)\leq\max(\dim(x),\dim(y))$ and $\dim(xy)\leq \dim(x)+\dim(y)$;
moreover, $\dim(x\Lef^n)=\dim(x)+n$ for every $n\in\Z$.

Assume that $k$ is algebraically closed.
For every $k$-variety~$X$, 
we denote by $\mathrm H^p(X)$ (resp. $\mathrm H^p_{\mathrm c}(X)$) 
its $p$th singular cohomology group (resp. with compact support)
and $\Q$-coefficients (if $k=\C$),
or its $p$th étale cohomology group (resp. with proper support)
and $\Q_\ell$-coefficients
(for some fixed prime number~$\ell$ distinct from the characteristic
of~$k$).
There is a unique ring morphism~$\PC$ from~$\GVar_k$
to the polynomial ring~$\Z[t]$ such that for every
variety~$X$, $\PC([X])$ is the Poincaré polynomial of~$X$.
Its definition relies on the weight filtration on the
cohomology groups with compact support of~$X$.
If $k$ has characteristic zero (which will be the case below),
the morphism~$\PC$ is characterized
by its values on projective smooth varieties:
for every such~$X$, one has
\[ \PC([X]) =
\PC_X(t) 
= \sum_{p=0}^{2\dim (X)} \dim(\mathrm H^p(X)) t^p. \]
This implies that for every variety~$X$, the leading term of~$\PC([X])$
is given by $\kappa(X) t^{2\dim(X)}$, where
$\kappa(X)$ is the number  
of irreducible components of~$X$ of dimension~$\dim(X)$.

One has $\PC(\Lef)=\PC([\P^1])-1=t^2$; for every $a\geq 1$,
$\PC(\Lef^a-1)=t^{2a}-1=t^{2a}(1-t^{-2a})$ is invertible
in the ring~$\Z\lbra t^{-1}\rbra [t]$, with inverse
\[\sum_{m\geq 1} t^{-2ma}  . \]
Consequently, the morphism~$\PC$
extends uniquely to a ring morphism
from~$\GVarM_k$ to the ring~$\Z\lbra t^{-1}\rbra [t]$.
For every element $x\in\GVarM_k$, one has
\[ \dim(x)\geq \frac12 \deg(\PC(x)). \]

\begin{lemm}\label{lemm.elements-simples}
Let $A$ be a ring and let $P_1,\dots,P_r\in A[T]$ be polynomials with
coefficients in~$A$. Assume that for every~$i$,
the leading coefficient of~$P_i$ is a unit in~$A$, and that for every
distinct~$i,j$, the resultant of~$P_i$ and~$P_j$
is a unit in~$A$.
Then, for every polynomial $P\in A[T]$ and every family $(n_1,\dots,n_r)$
of nonnegative integers, there exists a unique family
$(Q_{i,j})$ of polynomials in~$A[T]$,
indexed by pairs of integers~$(i,j)$ 
such that $1\leq i\leq r$ and $1\leq j\leq n_i$,
and a unique polynomial~$Q\in A[T]$ such that
$ \deg(Q_{i,j}) \leq \deg(P_i)-1$  for all $i,j$ and
\[
P(T) = Q(T)\prod_{i=1}^r P_i(T)^{n_i}
+ \sum_{i=1}^r \sum_{j=1}^{n_i}Q_{i,j}(T) P_i(T)^{n_i-j}\prod_{k\neq i}P_k(T)^{n_k}  .\]
\end{lemm}
Since the leading coefficient of~$P_i$ is a unit, $P_i$
is not a zero divisor in~$A[T]$.
Observe that the last equality is a decomposition
into partial fractions
\[\frac{P(T)}{\prod_{i=1}^r P_i(T)^{n_i}} = Q(T)  
+ \sum_{i=1}^r \sum_{j=1}^{n_i} \frac{Q_{i,j}(T)}{  P_i(T)^{j}} \]
in the total ring of fractions of~$A[T]$.
\begin{proof}
First assume that $r=1$. In this case, the desired assertion follows
from considering the Euclidean divisions by~$P_1$
of the polynomial~$P$, of its quotient, etc. 
\begin{align*}
P(T) & = Q_1(T) P_1(T)+R_1(T) \\
& = Q_2(T)P_1(T)^2 + R_2(T) P_1(T)+R_1(T) \\
& = \dots \\
& = Q_{n_1}(T)P_1(T)^{n_1} + R_{n_1}(T)P_1(T)^{n_1-1}+\dots+R_1(T), \end{align*}
where $Q_1,\dots,Q_{n_1}\in A[T]$ and $R_1,\dots,R_{n_1}$ are polynomials
of degrees~$\leq \deg(P_1)-1$.

Now assume $r\geq 2$.
Let $i,j$ be distinct integers in~$\{1,\dots,r\}$.
By the assumption and basic properties of the resultant, 
there exist polynomials~$U,V\in A[T]$ such that $1=UP_i+VP_j$.
Consequently,
the ideals $(P_i)$ and~$(P_j)$ generate the unit ideal of~$A[T]$.
By induction on~$n_1,\dots,n_r$,
it follows that the ideals $(\prod_{k\neq i} P_k^{n_k})$, for $1\leq i\leq r$,
are comaximal in~$A[T]$.
Therefore, there exist polynomials $U_1,\dots,U_r\in A[T]$ such that
\[ 1 = \sum_{i=1}^r U_i(T) \prod_{k\neq i}P_k(T)^{n_k}. \]
By  the case $r=1$ applied to the polynomials $U_i(T)$
and $P_i(T)$, 
we obtain the desired decomposition.

Uniqueness is left to the reader.
\end{proof}

\begin{lemm}\label{lemm.resultants}
Let  \tcb{$a,a'$ be nonnegative integers and $b,b'$} be positive integers. Let $d=\gcd(b,b')$.
Then, 
\[ \Res(1-\Lef^a T^b, 1-\Lef^{a'}T^{b'})
 = (-1)^{b'} \Lef^{ab'} (1-\Lef^{(a'b-ab')/d})^d. \]
In particular, this resultant is a unit in~$\GVarM_k$ if $(a,b)$ and $(a',b')$
are not proportional.
\end{lemm}
\begin{proof}
It is sufficient to prove this formula when the ring~$\GVarM_k$
is replaced by the ring $A=\C[\Lef^{\pm 1/bb'}]$ of Laurent polynomials
in an indeterminate $\Lef^{1/bb'}$. Then, the polynomial
$1-\Lef^aT^b$ is split in~$A$; this leads to
the explicit elementary computation
\begin{align*}
\Res(1-\Lef^aT^b,1-\Lef^{a'}T^{b'})
& = (-\Lef^a)^{b'} \prod_{\zeta^b=1} (1-\zeta^{b'} \Lef^{a'-ab'/b}) \\
&= (-1)^{b'} \Lef^{ab'}
\prod_{\xi^{b/d}=1} (1-\xi \Lef^{a'-ab'/b})^d \\
&= (-1)^{b'} \Lef^{ab'} (1-\Lef^{(a'-ab'/b)b/d})^d \\
& = (-1)^{b'} \Lef^{ab'} (1-\Lef^{(a'b-ab')/d})^d.\qedhere
\end{align*}
\end{proof}

\begin{prop}\label{prop.tauber}
Let $Z(T)=\sum_{n\in\Z} [M_n]T^n\in\GVar^+_k\lbra T\rbra[T^{-1}]$
be a Laurent series with effective coefficients in~$\GVar_k$.

Let $a$ and~$d$ be positive integers and let 
$P(T)=(1-\Lef^a T^a)^d Z(T)$. Assume that $P(T)$
belongs to $\GVarM_k\{T\}^\dagger$
and that $P(\Lef^{-1})$ is an effective non-zero element of~$\GVarM_k$.
Then, for every $p\in\{0,\dots,a-1\}$, one of the following cases
occur when $n$ tends to infinity in the congruence class of~$p$ modulo~$a$:
\begin{enumerate}
\item  Either $\dim(M_n) = \mathrm o(n)$, 
\item Or $\dim(M_n)-n$ has a finite limit and 
$\frac{\log(\kappa(M_n))}{\log(n)}$ converges to some integer in~$\{0,\dots,d-1\}$.
\end{enumerate}
Moreover, the second case happens at least once.
\end{prop}

\begin{proof}
Without lack of generality, we assume that $Z(T)$ is a power series.
Set $a_1=b_1=a$ and $d_1=d$.
By assumption, there exist a finite family $(a_i,b_i)_{2\leq i\leq r}$ 
of integers such that $\tcb{0\leq a_i<b_i}$ for all $i\geq 2$, 
and integers $d_i$ such that
\[Q(T) = Z(T) \prod_{i=1}^r (1-\Lef^{a_i}T^{b_i})^{d_i}  \]
is a polynomial in~$\GVarM_k[T]$.
Using the fact that $1-\Lef^mT^n$ divides $1-\Lef^{mp}T^{np}$ for every positive
integer~$p$, we may assume that no two pairs
$(a_i,b_i)$ and $(a_j,b_j)$ are proportional.

For every~$i\in\{1,\dots,r\}$, 
set $P_i(T)=1-\Lef^{a_i}T^{b_i}$; its leading coefficient
is invertible in~$\GVarM_k$. 
Moreover, for $i$ and $j$ such that $1\leq i<j\leq r$, it follows
from Lemma~\ref{lemm.resultants} that the resultant of~$P_i$
and~$P_j$ is a unit in~$\GVarM_k$.
Thus, by decomposition in partial fractions (Lemma~\ref{lemm.elements-simples}),
there exist polynomials $Q_0$ and $Q_{i,j}$ in $\GVarM_k[T]$ such that
\begin{equation}\label{eq.partialfractions}
Z(T) = Q_0(T) + \sum_{i=1}^r \sum_{j=1}^{d_i}\frac{Q_{i,j}(T)}{(1-\Lef^{a_i}T^{b_i})^{j}} \end{equation}
and $\deg(Q_{i,j})\leq b_i-1$ for every~$i\in\{1,\dots,r\}$.

For $i\in\{1,\dots,r\}$ and $j\in\{1,\dots,d_i\}$, write $Q_{i,j}=\sum_{n=0}^{b_i-1} q_{i,j,n} T^n$, for some elements $q_{i,j,n}\in\GVarM_k$.
This leads to the
following power expansion in~$\GVarM_k\lbra T\rbra $:
\begin{align*}
Z(T) & =Q_0(T)  + \sum_{i=1}^{r} \sum_{j=1}^{d_i}
 \sum_{n=0}^{b_i-1}q_{i,j,n}  T^n \sum_{m=0}^\infty \binom{j+m-1}{j-1} \Lef^{a_im}T^{b_im}  \\
&= Q_0(T)+ \sum_{n=0}^{\infty} \left(\sum_{i=1}^r \sum_{j=1}^{d_i} 
   \binom{j+\lfloor n/b_i\rfloor -1}{j-1}q_{i,j,n\bmod b_i} \Lef^{a_i\lfloor n/b_i\rfloor} \right) T^n,
\end{align*}
so that for every $n> \deg(Q_0)$, one has
\begin{equation} [M_n] = 
\sum_{i=1}^r \sum_{j=1}^{d_i} 
   \binom{j+\lfloor n/b_i\rfloor -1}{j-1}q_{i,j,n\bmod b_i} \Lef^{a_i\lfloor n/b_i\rfloor} . \end{equation}
For every~$i,j$, define
\begin{equation}
[M_n]^{i,j} =  \binom{j+\lfloor n/b_i\rfloor -1}{j-1}q_{i,j,n\bmod b_i} \Lef^{a_i\lfloor n/b
_i\rfloor}  \end{equation} 
and
\begin{equation}
[M_n]^i = \sum_{j=1}^{d_i} [M_n]^{i,j}, 
\end{equation}
so that
\begin{equation}
[M_n] = \sum_{i=1}^r [M_n]^i.
\end{equation}
It follows directly from the definitions 
that for every~$i\geq 1$ and every~$n\geq 1$,
$\dim([M_n]^i)\leq (a_i/b_i) n+\mathrm O(1)$. 
Since $a_i\leq b_i$ for all~$i$, this implies $\dim([M_n])\leq n+\mathrm O(1)$.
We will now show that when $n$
belongs to appropriate congruence classes modulo~$a$,
one has the equality $\dim([M_n]^1)=n+\mathrm O(1)$.
Since $a_i<b_i$ for $i\geq 2$ and $a_1=b_1=a$,
this will imply the relations
$\dim([M_n])=n+\mathrm O(1)$ and $\kappa([M_n])=\kappa([M_n]^1)$
(for $n$ large enough in this congruence class).

Let $n$ be any integer~$>\deg(Q_0)$,
let $n=am+\bar n$ be the Euclidean division of~$n$ by~$a$.
It follows from the definition of~$[M_n]^{1}$ that
\[ [M_n]^{1} \Lef^{-n} 
=  \sum_{j=1}^d \binom{j+m-1}{j-1} q_{1,j,\bar n}\Lef^{am-n}
= \sum_{j=1}^d \binom{j+m-1}{j-1} q_{1,j,\bar n}\Lef^{-\bar n}.
\]

It follows from Equation~\eqref{eq.partialfractions} 
that
\[ P(\Lef^{-1}) = \left[Z(T) (1-\Lef^aT^a)^d\right]_{T=\Lef^{-1}}
= Q_{1,d}(\Lef^{-1})
= \sum_{p=0}^{a-1} q_{1,d,p}\Lef^{-p}.\]
Since $P(\Lef^{-1})$ is effective and non-zero, 
its Poincaré polynomial $\PC(P(\Lef^{-1}))$ is non-zero.
Consequently, there must exist an integer $p\in\{0,\dots,a-1\}$
such that $\dim(q_{1,d,p})\neq -\infty $.
We now restrict the analysis to integers~$n$
congruent to~$p$ modulo~$a$.
%
Set
\begin{equation*}
d_p = \max_{1\leq j\leq d}  \dim(q_{1,j,p})-p 
\end{equation*}
and let $j_p$ be the largest integer~$j$ such that $d_p=\dim(q_{1,j,p})-p$.
Looking at Poincaré polynomials and using that 
for $j\neq j_p$, either $\dim(q_{1,j,p})<\dim(q_{1,j_p,p})$, or the
binomial coefficient $\binom{j_p+m-1}{j_p-1}$ goes to infinity faster 
than $\binom{j+m-1}{j-1}$ when $m\to\infty$, we get
the following asymptotic expansions
\[\dim([M_n]^1\Lef^{-n}) =  d_p \quad\text{and}\quad \kappa([M_n]) \sim \binom{j_p+m-1}{j_p-1} \kappa(q_{1,j_p,p}) \]
for $n$ large enough and congruent to~$p$ modulo~$a$.
In particular,
\[\dim([M_n]) = n + \mathrm O(1) 
\quad\text{and}\quad
  \frac{\log(\kappa([M_n]))}{\log(n)} \ra j_p-1.
\]
This concludes the proof of the proposition.
\end{proof}

\Subsection{Existence of the moduli spaces}

In this Section, we prove a general proposition that asserts
existence of moduli schemes of sections of bounded height 
in a general context.

Let $k$ be a field, let $C $ be \tcb{an irreducible} projective smooth~$k$-curve;
let $\eta$ be its generic point and let $F=k(C)$ be the function field of~$C$.
Let $C_0$ be a non-empty Zariski open subset of~$C $.

Let $X$ be \tcb{an irreducible} projective $k$-variety together with a surjective flat
morphism $\pi\colon X\ra C $.
Let~$G$ be a Zariski open subset of~$X_F$, assumed to be affine.
Let~$U$ be a Zariski open subset of~$X$ such that $G\subset U_F$
and $\pi(U)\supset C_0$.

Let $(D_\alpha)_{\alpha\in\mathscr A}$ be a finite family
of Cartier divisors on~$X$ such that, for each~$\alpha$,
the restriction of~$D_\alpha$  to~$X_F$ is effective 
and
$X_F\setminus G=\bigcup \abs{D_\alpha}_F$.
For each~$\alpha$, we also let $\mathscr L_\alpha$
be the line bundle~$\mathscr O_X(D_\alpha)$.
Finally, we assume that there exists a linear
combination with positive coefficients
$\mathscr L=\sum\lambda_\alpha\mathscr L_\alpha$,
as well as a \tcr{Cartier $\Q$-divisor~$D$} on~$X$ such that $D_F$ is effective,
supported by~$(X\setminus G)_F$ 
and such that $\mathscr L(-D)$ is ample.

\begin{rema}\label{rema.lower-bound}
Let $\mathscr L$ be a line bundle on~$X$ and
let $f\in\Gamma(X_F,\mathscr L)$ be a non-zero global section.
Let us show that there exists an integer~$m$ such that 
for every section $\sigma\colon C\to X$ of~$\pi$
satisfying $\sigma(\eta) \notin \div(f)$, one has
$\deg(\sigma^*\mathscr L)\geq -m$.

There exists an effective Cartier divisor~$E$ on~$C $
such that $f$ extends to a global section of $\mathscr L \otimes\pi^*(E)$.
Indeed, viewing~$f$ as a meromorphic section of~$\mathscr L$
on~$X$, let us decompose its  divisor 
as the sum $H+V$ of
its horizontal (\emph{i.e.}, faithfully flat over~$C $) 
and vertical (mapping to a point) irreducible components.
By construction, the components of~$H$ are the Zariski closures
of the components of the divisor of~$f$, viewed as a section
of~$\mathscr L$ on~$X_F$; consequently, $H$ is effective
by hypothesis.
Still by definition, $V$ is a linear combination
of irreducible components of closed fibers.
Consequently, there exists an effective divisor~$E$ on~$C $
such that $V\geq -\pi^*E$; then $f$ extends to a global section
of $\mathscr L\otimes\pi^*(E)$.

In particular, 
for every section $\sigma\colon C\to X$ of~$\pi$
satisfying $\sigma(\eta) \notin \div(f)$, one has
$\deg(\sigma^*\mathscr L)\geq -\deg(\sigma^*\pi^*E)=-\deg(E)$.
\end{rema}


\begin{prop}\label{lemm.hilbert}
For every $\mathbf n$ in~$\Z^{\mathscr A}$, 
there is a quasi-projective $k$-scheme $M_{G,\mathbf n}$ parameterizing
sections $\sigma\colon C \ra X$ of~$\pi$
satisfying the following properties:
\begin{itemize}
\item under~$\sigma$, the generic point~$\eta$ of~$C $ 
is mapped to a point of~$G$;
\item for each $\alpha\in\mathscr A$, 
$\deg_{C }\sigma^*\mathscr L_\alpha=n_\alpha$.
\end{itemize}

In that scheme, the sections~$\sigma$ such that $\sigma(C_0)\subset U$
constitute a constructible set~$M_{U,\mathbf n}$.
Moreover, there exists~$n_0\in\Z$ such that $M_{U,\mathbf n}$ is empty  
if $n_\alpha<n_0$ for some~$\alpha\in\mathscr A$.
\end{prop}
\begin{proof}
As a standard
consequence of the existence of Hilbert schemes, there
exists a $k$-scheme $M_{X,\mathbf n}$ which parameterizes
sections $\sigma\colon C \ra X$ such that
$\deg_{C }\sigma^*\mathscr L_\alpha=n_\alpha$ for each~$\alpha$.
Indeed, the functor of \tcb{sections}~$\sigma\colon C \ra X$ is represented
by the open subscheme of the Hilbert scheme $\Hilb_{X}$ which parameterizes
the closed subschemes of~$X$ which are mapped isomorphically by~$\pi$.
By flatness, 
each of the condition~$\deg_{C }\sigma^* \mathscr L_\alpha=n_\alpha$ 
is open and closed in the Hilbert scheme. 

The condition that the
generic point of~$C $ is mapped to a point of~$G$
means that $\sigma(C )\not\subset \abs{X\setminus G}$,
while the condition $\sigma(C )\subset\abs{X\setminus G}$
defines an closed subscheme of $M_{X,\mathbf n}$. 
Let $M_{G,\mathbf n}$ be its complement.
By construction, this scheme represents the given functor,
and we have to prove that it is quasi-projective.

First of all, since the restriction to~$X_F$ of the
divisor~$D_\alpha$ is effective and disjoint from~$G$,
\tcr{Remark~\ref{rema.lower-bound} asserts that} there exists an integer~$m$
such that
$\deg(\sigma^*\mathscr L_\alpha)\geq -m$
for every~$\mathbf n$ and every section~$\sigma$ in $M_{G,\mathbf n}$.

Let $\mathscr M$ be an ample $\Q$-line bundle on~$X$
of the form~$\mathscr L(-D)$,
where \tcr{$D$ is a Cartier $\Q$-divisor such that
$D_F$ is effective and disjoint from~$G$.}
For every $\sigma\in M_{G,\mathbf n}$, one has
\[ \deg \sigma^*(\mathscr M+\mathscr O(D))=\deg\sigma^*\mathscr M+ \deg\sigma^*\mathscr O(D). \]
\tcr{By Remark~\ref{rema.lower-bound},} there exists an integer~$m'$ such that
\[ \deg\sigma^*\mathscr O(D) \geq -m' \]
for all sections~$\sigma\in M_{G,\mathbf n}$. 
Therefore, $\deg\sigma^*\mathscr M \leq m+\sum\lambda_\alpha n_\alpha$ 
for all~$\sigma\in M_{G,\mathbf n}$.
By a theorem of Chow (\cite{sga6}, XIII, Cor.~6.11),
this gives only finitely many
possibilities for the Hilbert polynomial (relative to~$\mathscr M$)
of the image of a section~$\sigma$ which belongs to $M_{G,\mathbf n}$.
It is well known that the subschemes of~$X$ with given Hilbert polynomial 
with respect to the ample line bundle~$\mathscr M$ form
a closed and open subscheme of~$\Hilb_X$, which is projective
as a scheme. Consequently,  $M_ {G,\mathbf n}$ is quasi-projective.

It remains to prove that the condition ``$\sigma(C_0)\subset U$''
defines  a constructible subset of~$M_{G,\mathbf n}$.
Indeed, let $T$ be a scheme and let $\sigma\colon C \times T\ra X$
be a morphism. Let $V=\sigma^{-1}(U)$
and let $Z$ be the complement of~$V\cap(C_0\times T)$ in~$C_0\times T$; this
is  a closed subset of $C_0\times T$.
The set of points~$t\in T$ such that $\sigma(C_0\times\{t\})\not\subset U$
is equal to the projection in~$T$ of $Z$, so is constructible,
as claimed.
\end{proof}

\begin{rema}
Assume that for each $\alpha\in\mathscr A$
the divisor~$D_\alpha$ is effective; 
then $M_{U,\mathbf n}$ is actually an  \emph{open and closed subscheme}
of $M_{G,\mathbf n}$. Indeed, let us write $\deg\sigma^*\mathscr L_\alpha$
as the intersection number of~$\sigma_*C $ with $D_\alpha$.
By definition of $M_{U,\mathbf n}$, this is a sum 
of local contributions $(\sigma_*C , D_\alpha)_v$
at all points of $C \setminus C_0$.
Since $D_\alpha$ is effective,
each of these contributions is lower semi-continuous as a function
of~$\sigma$ (it may increase on closed subsets), while their sum
is the constant $n_\alpha$ on $M_{G,\mathbf n}$. This decomposes
$M_{U,\mathbf n}$ as a disjoint union of open and closed 
subschemes defined by prescribing the possible values
for $(\sigma_*C , D_\alpha)_v$.
\end{rema}

\subsection{Clemens complexes}

Let $X$ be a smooth algebraic variety over a field~$K$
and let $D$ be an effective divisor with strict normal crossings on~$X$;
in other words, \tcb{the support of} $D$ is the union of its
irreducible components which are themselves smooth and meet transversally.

The Clemens complex $\Cl(X,D)$ of $(X,D)$ is the simplicial complex
whose points are irreducible components of~$D$,
edges are irreducible components
of intersections of two  distinct irreducible components,
etc. By the normal crossing assumption, all of these schemes are smooth.
Thus, the dimension of the Clemens complex  is the maximal number
of irreducible components of~$D$ whose intersection is non-empty,
minus~$1$.
For every integer~$d$, we also write $\Cl^d(X,D)$ for
the set of simplices (also called faces)
of dimension~$d$ of $\Cl(X,D)$.

The analytic Clemens complex $\Clan(X,D)$ is the subcomplex
of~$\Cl(X,D)$ consisting of those simplices~$Z\in\Cl(X,D)$
such that $Z(K)\neq\emptyset$.
One writes $\Clanmax(X,D)$ for the set of maximal faces of~$\Clan(X,D)$
and $\Clan[d](X,D)$ for the set of faces of dimension~$d$
of $\Clan(X,D)$.

If $L$ is an extension of~$K$, the divisor~$D_L$
on~$X_L$ still has strict normal crossings
and one writes $\Cl_L(X,D)=\Cl(X_L,D_L)$
and $\Clan_L(X,D)=\Clan(X_L,D_L)$.

\Subsection{Motivic residual functions on arc spaces}

Let $k$ be an algebraically closed field of characteristic
zero, let $R$ be the complete discrete valuation ring $k\lbra t\rbra$,
and let $K=k\lpar t\rpar$ be its field of fractions.

Let $\mathscr X$ be a flat $R$-scheme of finite type,
equidimensional of relative dimension~$n$.

For every integer~$m\geq 0$, we write $\mathscr L_m(\mathscr X)$
or~$\mathscr X(m)$ for the $m$th Greenberg space of~$\mathscr X$,
see~\S2.3 of~\cite{loeser/sebag2003}  for the precise general definition.
Let us simply recall that $\mathscr X(m)$
is the algebraic variety over~$k$ 
which represents
the functor $\ell\mapsto \mathscr X(\ell\lbra t\rbra /(t^{m+1}))$
on the category of $k$-algebras. 
There are natural affine morphisms 
$p_m^{m+1}\colon \mathscr X(m+1)\ra\mathscr X(m)$;
consequently, the projective limit 
$\mathscr L(\mathscr X)=\varprojlim_m \mathscr L_m(\mathscr X)$ 
exists as a $k$-scheme.
Let $p_{\tcb{m}} \colon \mathscr L(\mathscr X)\ra \mathscr L_m(\mathscr X)$
be the canonical projection.
When $\mathscr X=X\otimes_k R$, for some $k$-variety~$X$,
then $\mathscr L_m(\mathscr X)$
is the space of $m$-jets of~$X$,
and $\mathscr L(\mathscr X)$ is the arc space of~$X$.

The paper~\cite{cluckers-loeser:2008} introduces a general definition
of constructible motivic functions on arc spaces. 
In this paper, we shall \tcb{mostly} consider  the following 
more restrictive class:
We define
a motivic residual function~$h$ on~$\mathscr L(\mathscr X)$
to be an element of the inductive limit
of all relative Grothendieck groups $\Mot_{\mathscr X(m)}$.
Recall that $\Mot_{\mathscr X(m)}$ is a localization of
the Grothendieck ring $\GVar_{\mathscr X(m)}$;
in particular, a motivic residual function comes from the latter ring
if it is given by a formal linear combination of varieties 
$H\ra \mathscr X(m)$; 
in addition to the cut-and-paste relations at the heart of 
the definition of the Grothendieck groups,
we identify the diagrams
$H\ra \mathscr X(m)$  
and $H\times_{\mathscr X(m)}\mathscr X(m+1)\ra\mathscr X(m+1)$.
The fiber product structure of varieties gives rise to a ring
structure on the set of motivic residual functions 
on~$\mathscr L(\mathscr X)$.

An example of such a motivic residual function
is the characteristic function of 
a constructible subset~$W$ of~$\mathscr L(\mathscr X)$: such a $W$
is of the form $p_m^{-1}(W_m)$ for a constructible subset~$W_m$
of~$\mathscr L_m(\mathscr X)$ and $\mathbf 1_W$ is
given by the obvious diagram $W_m\ra \mathscr X(m)$.
Let $A$ be an algebraic variety over~$k$, 
and let $a$ be its class in~$\Mot_k$; then 
the motivic residual function~$a\mathbf 1_W$
is the diagram $A\times W_m\ra \mathscr X(m)$, the map being
the second projection composed by the inclusion of~$W_m$ into~$\mathscr X(m)$.
Motivic residual functions on~$\mathscr L(\Aff^n)$
are examples of Schwartz-Bruhat motivic functions
in~$K^n$ with support in~$R^n$
(see \S\ref{ss-sbfunctions-general}).

\section{Setup and notation}
\label{s.setup}

\tcr{In this Section, we fix the notation that will be
used for the rest of the paper.  Compared with the introduction,
we denote varieties fibered over the base curve by script letters,
and use capital letters for their generic fiber.
This reflects the fact that, even if models are given in
the statement of Theorem~\ref{theo.main},
its proof requires us to adjust them somewhat.}

\subsection{Algebraic geometry}\label{s.setup-ag}
Let $k$ be an algebraically closed field of characteristic zero,
let $C_0$ be a smooth quasi-projective connected curve over~$k$,
let $C $ be its smooth projective compactification
and let $S=C \setminus C_0$.
Let $F=k(C)=k(C_0)$ be the function field of~$C $; let $\eta_C$
be its generic point.

Let $G$ be the group scheme $\ga^n$ and let $X$ be a smooth
projective equivariant compactification of $G_F$.
In other words, $X$ is a smooth projective $F$-scheme containing
$G_F$ as a dense open subset, and the group law $G_F\times G_F\ra G_F$
extends as a group action $G_F\times X\ra X$.
The boundary $X\setminus G_F$ of~$G_F$ in~$X$
is a divisor. 
In this paper, we make the hypothesis that this divisor has strict
normal crossings. More precisely, we assume
that its irreducible components 
are geometrically irreducible, smooth and meet transversally, 
so that for every~$p$,
the intersection of any~$p$ of those components is either empty
or smooth of dimension~$n-p$. 
This is a slightly stronger assumption that the one done in the
arithmetic case~\cite{chambert-loir-tschinkel:2012},
where we only made this hypothesis after base change to~$\bar F$.
The general case can be treated in a similar way,
by constructing appropriate weak Néron models; we leave
it to the interested reader.

We write $D=X\setminus G_F$ and $(D_\alpha)_{\alpha\in\mathscr A}$ 
for the family of its irreducible components.
The divisors~$D_\alpha$  form a basis of the group~$\Pic(X)$, and a basis
of the monoid $\Lambda_\eff(X)$ of effective divisors in~$\Pic(X)$.
We will freely identify line bundles on~$X$ with divisors whose
support is contained in the boundary, and with their
classes in the Picard group.

Up to multiplication by a scalar, there is a unique $G_F$-invariant
meromorphic differential form~$\omega_X$ on~$X$; its restriction
to~$G_F$ is proportional to the form 
$\mathrm dx_1\wedge\cdots\wedge\mathrm dx_n$.
Its divisor, or its class, is the canonical class~$K_X$ of~$X$.
The divisor $-\div(\omega_X)$ can be written as $\sum\rho_\alpha D_\alpha$
for some integers~$\rho_\alpha\geq 2$ (see~\cite{hassett-t99},
Theorem~2.7).
In particular, the anticanonical class~$K_X^{-1}$ is effective.

The log-canonical class of the pair $(X,D)$ 
in~$\Pic(X)$ is the class of $K_X'=K_X+D$. 
Its opposite, the log-anticanonical class,
is given by $\sum \rho'_\alpha D_\alpha$ with
$\rho'_\alpha=\rho_\alpha-1$ for all~$\alpha$.
Since $\rho_\alpha\geq 2$ for all~$\alpha\in\mathscr A$, 
\tcr{the divisor $-K'_X$ 
can be written as the sum of an ample line divisor
and of an effective divisor (in other words, it is big),
as claimed in the introduction.}

We also recall that $\mathrm H^i(X,\mathscr O_X)=0$ for every integer~$i>0$.

\subsection{Models and heights}
A \emph{model}  of~$X$ over~$C$ is a projective
flat scheme $\pi\colon\mathscr X\ra C$ whose generic fiber
is equal to~$X$.
If, moreover, $\mathscr X$ is regular and if
the sum of the non-smooth fibers of~$\mathscr X$ 
and the 
closures $\mathscr D_\alpha$ of the divisors~$D_\alpha$ 
is a divisor with strict normal crossings on~$\mathscr X$,
then we will say that $\mathscr X$ is a \emph{good model}.
One defines analogously good models of~$X$ over~$C_0$, 
or even over local rings whose field of fractions contains~$k(C)$.

Embedded resolution of singularities
in characteristic zero implies that good models exist.

We choose a good model~$\pi\colon\mathscr X\to C$ of~$X$ over~$C$.

For every point $v\in C (k)$, we write
$\mathscr B_v$ for the set of irreducible components of~$\pi^{-1}(v)$;
for $\beta\in\mathscr B_v$, let $E_\beta$ be the corresponding component
and $\mu_\beta$ be its multiplicty in the special fibre of~$\mathscr X$
at~$v$.
 Let $\mathscr B$ be the disjoint union of all~$\mathscr B_v$,
for $v\in C (k)$.
Let $\mathscr B_1$ be the subset of~$\mathscr B$
consisting of those $\beta$ for which the multiplicity $\mu_\beta$
equals~$1$; let $\mathscr B_{1,v}=\mathscr B_1\cap\mathscr B_v$.

The complement~$\mathscr X_1$ in~$\mathscr X$ 
of the union of the components~$E_\beta$,
for $\beta\in \mathscr B\setminus\mathscr B_1$ and of 
the intersections of distinct vertical components,
is a smooth scheme over~$C $. 
\begin{lemm}\label{lemm.weakNeron}
The $C $-scheme~$\mathscr X_1$ is 
a weak N\'eron model of~$X$:
for every smooth $C$-scheme~$\mathscr Z$,
the canonical map from $\Hom_C(\mathscr Z,\mathscr X_1)$
to $\Hom_{F}(\mathscr Z_F,X)$ is a bijection.
\end{lemm}
\begin{proof}
This follows from the
fact that the $C $-scheme~$\mathscr X_1$ 
is the smooth locus of the proper map $\pi\colon \mathscr X\ra C $,
and that $\mathscr X$ is regular.
See \cite{bosch-l-r90} for details, especially p.~61.
\end{proof}

For every $\alpha\in\mathscr A$, we assume given
a divisor~$\mathscr L_\alpha$ 
on~$\mathscr X$ which extends $D_\alpha$.
There exists a family of integers $(e_{\alpha,\beta})$,
all but finitely many of them being equal to~$0$,
indexed by $\alpha\in\mathscr A$ and $\beta\in\mathscr B$ 
such that
\begin{equation}\label{eq.e-alpha-beta}
 \mathscr L_\alpha =
     \mathscr D_\alpha + \sum_{\beta\in\mathscr B} e_{\alpha,\beta}E_\beta . \end{equation}
We also define integers~$\rho_\beta$, for $\beta\in\mathscr B$, by the formula
\begin{equation}\label{eq.rho-beta}
 -\div(\omega_X)=\sum_{\alpha\in\mathscr A} \rho_\alpha \mathscr D_\alpha
+ \sum_{\beta\in\mathscr B} \rho_\beta E_\beta, \end{equation}
where $\omega_X$ is viewed as a meromorphic section of
the line bundle~$K_{\mathscr X/C}$.

Since $\pi$ is proper and $C $ is a smooth curve,
the map $\sigma\mapsto \sigma(\eta_C)$ is a bijection between 
the set of sections $\sigma \colon C\ra\mathscr X$ of~$\pi$
and 
the set of rational points $X(F)$ of~$X$.
For every line bundle~$\mathscr L$ on~$\mathscr X$ and every
section $\sigma\colon C \ra\mathscr X$,
the degree $\deg_{C }\sigma^*\mathscr L$ 
is the geometric analogue of the \emph{height} 
of the corresponding rational point.

\subsection{Local descriptions}
Let $v\in C (k)$. We write $F_v$ for the completion of~$F=k(C)$ at~$v$.
If $t$ is a local parameter of~$C $ at~$v$, then $F_v\simeq k\lpar t\rpar$
and $\widehat{\mathscr O}_{C,v}\simeq k\lbra t\rbra$.
\tcr{Writing an element~$x$ of $\widehat{\mathscr O}_{C,v}$
as a power series $x_0+x_1t+\dots$,
we consider $k\lbra t\rbra$ as the set of $k$-points
of the scheme $\Spec(k[x_0,x_1,\dots])$; writing
an element of~$F_v$ as a Laurent series $x_{-m}t^{-m}+\dots+x_0+x_1t+\dots$,
we view $k\lpar t\rpar$ as the set
of $k$-points of the ind-scheme whose $m$th term is
$\Spec(k[x_{-m},\dots,x_0,x_1,\dots])$.
Fixing an isomorphism $G\simeq \ga^n$,
we have an identification $G(F_v)\simeq k\lpar t\rpar^n$
of~$G(F_v)$ with the $k$-points of an ind-$k$-scheme. 
We will say that a subset of $G(F_v)$ is definable if it
can be defined in the language~$\mathscr L_{\mathrm{DP},\mathrm P}$ 
of Denef-Pas (see~\cite{cluckers-loeser:2008}, \S2.1, for more details).
In particular, the set of $k$-points of a constructible subset of 
a finite level of this ind-scheme is definable.}

For every point $g\in G(F_v)$, 
one can attach local intersection
degrees $(g,\mathscr D_\alpha)_v$, for $\alpha\in\mathscr A$,
defined as follows. 
By the valuative criterion of properness,
the map~$g\colon \Spec (F_v)\ra G_{F}$
extends to a morphism 
$\tilde g\colon \Spec (\widehat{\mathscr O}_{C ,v})\ra \mathscr X$
and we can consider the pull-back
$\tilde g^*\mathscr D_\alpha$ of~$\mathscr D_\alpha$ 
as an effective Cartier divisor on $\Spec (\widehat{\mathscr O}_{C ,v})$.
We define $(g,\mathscr D_\alpha)_v\in\N$ 
by the formula $\tilde g^*\mathscr D_\alpha=(g,\mathscr D_\alpha)_v [v]$.
For $\beta\in\mathscr B_v$, we define an integer $(g,E_\beta)_v\in\{0,1\}$
similarly,
considering the pull-back of~$E_\beta$.

Observe also that $\sum_{\beta\in\mathscr B_v}\mu_\beta (g,E_\beta)_v=1$.
In particular, for every $g\in G(F_v)$,
there is exactly one index~$\beta\in\mathscr B_v$
for which $(g,E_\beta)_v=1$ and one has $\mu_\beta=1$.

By the valuative criterion properness, every point $g\in G(F)$
extends canonically to a section $\sigma_g\colon C\to\mathscr X$.

\begin{lemm}
For every $g\in G(F)$ and every 
every $\alpha\in\mathscr A$, one has 
\[ \deg_{C }(\sigma_g^*(\mathscr D_\alpha))
 = \sum_{v\in C (k)}  (g,\mathscr D_\alpha)_v. \]
\end{lemm}
\begin{proof}
Since $g\in G(F)$, the Cartier divisor $\sigma_g^*(\mathscr D_\alpha)$
on~$C $
is well-defined and represents the inverse
image by~$\sigma_g$ of the line bundle $\mathscr O_{\mathscr X}(\mathscr
D_\alpha)$. The given formula asserts 
that its degree is the sum of its multiplicities at
all closed points of~$C $.
\end{proof}

For $\mathbf m\in\N^{\mathscr A}$, we define
the \tcb{subset} $G(\mathbf m)_v$  (also denoted $G(\mathbf m)$ if
no confusion can arise concerning the point~$v$) of $G(F_v)$ 
as the set of all points~$g$
such that $(g,\mathscr D_\alpha)_v=m_\alpha$ for
all $\alpha\in\mathscr A$.
For $\mathbf m\in\N^{\mathscr A}$ and $\beta\in\mathscr B_v$,
the \tcb{subset} $G(\mathbf m,\beta)$  of $G(\mathbf m)_v$ 
consists of points~$g$ such that
$(g,E_\beta)_v=1$ (hence $(g,E_{\beta'})_v=0$ for all $\beta'\in\mathscr B_v$
such that $\beta'\neq\beta$).
When $\mathscr B_v$ has a single element, we often call it~$\beta_v$.

\begin{lemm}
For every $\mathbf m\in\N^{\mathscr A}$ 
and every $\beta\in\mathscr B_v$, 
the sets $G(\mathbf m)_v$ and $G(\mathbf m,\beta)$ 
are bounded definable subsets of~$G(F_v)$ 
and 
$\displaystyle G(F_v)=\bigcup_{\mathbf m\in\N^{\mathscr A}}
G(\mathbf m)_v=
\bigcup_{\substack{\mathbf m\in\N^{\mathscr A}\\
  \beta\in\mathscr B_v}} G(\mathbf m,\beta)$
(disjoint unions).
\end{lemm}
\begin{proof}
Since $G$ is affine, $X\setminus G_F$ contains the support of an ample
line bundle. We thus see that the valuations of the coordinates
of the points of~$G(\mathbf m,\beta)$ 
are bounded from below.
Since $\mathscr D_\alpha$ (resp. $E_\beta$) is effective,
the condition~$(g,\mathscr D_\alpha)_v\geq n_\alpha$ 
(resp. the condition $(g,E_\beta)_v\geq 1$)
defines a definable subset.
Taking differences, one gets that the sets $G(\mathbf m)_v$ 
and $G(\mathbf m,\beta)$ are bounded definable.
The last assertion is obvious.
\end{proof}

\begin{lemm}\label{lemm.G0}
There exists a dense open subset~$C_1$ of~$C_0$ such that for every 
closed point $v\in C_1$, the following properties hold:
\begin{enumerate}
\item One has $\mathscr B_v=\mathscr B_{1,v}=\{\beta_v\}$;
\item The set $G({\mathbf 0})_v=G(\mathfrak o_v)$ is a subgroup of $G(F_v)$;
\item For every $\mathbf m\in\N^{\mathscr A}$ and every $\beta\in\mathscr B_v$,
the set $G(\mathbf m,\beta)$ is invariant under the action of $G(\mathbf 0)_v$.
\end{enumerate}
\end{lemm}
\begin{proof}
By assumption, $X$ is a smooth equivariant compactification
of the $F$-group scheme~$G_F$. 
\tcr{By spreading-out,} there exists a dense Zariski open subset~$C_1$ of~$C $
such that $\mathscr X_{C_1}$ is a smooth equivariant compactification
of the $C_1$-group scheme~$G_{C_1}$, more precisely, such that 
the following properties hold:
\begin{itemize}
\item The morphism $\mathscr X_{C_1}\ra C_1$ is proper and smooth,
with geometrically integral fibers;
\item The action $G_F\times X\ra X$ of~$G_F$ on~$X$ extends
to an action $m\colon G_{C_1}\times \mathscr X_{C_1} \ra \mathscr X_{C_1}$;
\item The image of the section
$\sigma_0\in\mathscr X(C_1)$ extending the point $0\in G(F)$
is disjoint from all~$\mathscr D_\alpha$;
\item The morphism 
$g\mapsto m(g,\sigma_0)$ is an isomorphism from~$G_{C_1}$
to an open dense subscheme of~$\mathscr X_{C_1}$;
\item The Cartier
divisors $m^* \mathscr D_\alpha-\pr_2^*\mathscr D_\alpha$ 
on~$G_{C_1}\times \mathscr X_{C_1}$
are trivial, so that
$\mathscr X_{C_1}\setminus G_{C_1}$
is the union of the divisors~$\mathscr D_{\alpha,C_1}$.
\end{itemize}
This open set $C_1$ satisfies the requirements of the Lemma.
\end{proof}

\begin{lemm}\label{lemm.almost-all}
Let $v\in C (k)$.
For every integer~$r$, let $G(\mathfrak m_v^r)$ 
be the bounded definable subgroup of~$G(F_v)$
consisting of points~$g$ such that,
in the identification $G=\ga^n$, 
$\ord_v(g_i)\geq r$ for $i\in\{1,\dots,n\}$.
For every ~$v\in C (k)$, there exists an integer~$r_v$ such that,
for every $\mathbf m\in\N^{\mathscr A}$
and every $\beta\in\mathscr B_v$,
$G({\mathbf m},\beta)$ is invariant under~$G(\mathfrak m_v^{r_v})$.
Moreover, one can take $r_v=0$ for all but
finitely $v\in C (k)$.
\end{lemm}
\begin{proof}
When $v$ belongs to the open subset~$\tcb{C_1}$ constructed by Lemma~\ref{lemm.G0},
one may take $r_v=0$, hence the last claim.

In the remaining of the proof, we fix $v\in C ( k)$.
Fix $\alpha\in\mathscr A$ and let $f_\alpha$ be the canonical
global section of $\mathscr O_{\mathscr X}(\mathscr D_\alpha)$
\tcr{whose zero-divisor is~$\mathscr D_\alpha$}.
\tcr{We need to prove that there exists an integer~$r_v$ such that
$(gg',\mathscr D_\alpha)_v=(g',\mathscr D_\alpha)_v$
and $(gg',E_\beta)_v=(g',E_\beta)_v$,
for every $g\in G(\mathfrak m_v^{r_v})$, 
every $g'\in G(F)\subset\mathscr X(F)$,
every $\alpha\in\mathscr A$, and
every $\beta\in\mathscr B_v$.}

Since $G_F$ fixes~$\mathscr D_\alpha$ on the generic fiber, the line bundles
$m^* \mathscr O_{\mathscr X}(\mathscr D_\alpha)$ and $\pr_2^*\mathscr
O_{\mathscr X}(\mathscr D_\alpha)$
are isomorphic on~$G_F\times \mathscr X_F$ and $u=m^*f_\alpha/\pr_2^*\tcb{f}_\alpha$
is a rational function on~$G_F\times X$. The domain of definition
of~$u$ contains $G_F\times \mathscr X_F$. 
Since $\mathscr X$ is proper over~$C $, 
there exists a closed subset~$Z$ of~$G_{C }$ disjoint
from~$G_F$ such that $u$ is defined on \tcb{the complement of}~$\pr_{\tcb{1}}^{-1}(Z)$.
Moreover, $u(0,x)=1$ on~$X$.
Cover $\mathscr X$ by finitely many affine open subsets~$\Spec (A_i)$.
Then $u$ defines a rational function on 
$\ga^n\times\Spec (A_i)=\Spec (A_i[\mathbf T])$. Since 
$u$ is defined on 
$\Spec (A_i[\mathbf T][1/\varpi_v])$,
there exists an integer~$m$
such that $\varpi_v^m u\in A_i[\mathbf T]$ for all~$i$,
$\varpi_v$ denoting an uniformizer of~$\mathfrak o_v$.

It is now clear that if \tcr{$\ord_v(g_i)>m$} for $i\in\{1,\dots,n\}$,
 then \tcr{$\ord_v(u(g,x))=0$} for every integral point of~$\Spec (A_i)$.
Since every rational point of~$X$ extends to an integral point
of some~$\Spec (A_i)$, we obtain the desired conclusion for~$\mathscr D_{\alpha}$.

Now fix $\beta\in\mathscr B_v$.
Since $E_\beta$ is vertical, $E_\beta\otimes_C F=\emptyset$
and~$E_\beta$ is fixed by~$G_F$ on the generic fiber. Then, the
proof is identical to the one for $\mathscr D_\alpha$.
\end{proof}

\begin{coro}
For every $\mathbf m\in\N^{\mathscr A}$ and every $\beta\in {\mathscr B_v}$,
the characteristic function
of $G(\mathbf m,\beta)$ is a motivic Schwartz-Bruhat function on~$G(F_v)$
in the sense of \S\ref{ss-sbfunctions-general}.
\end{coro}

\Subsection{Integral points}

\begin{lemm}\label{lemm.replaceU}
Let $\mathscr U$ be a flat model of~$G_F$ over~$C_0=C \setminus S$,
let $\mathscr X$ be a flat model of~$X$ over~$C$.
There exists a good model~$\mathscr X'$ of~$X$
over~$C$ whose projection $\pi'\colon \mathscr X'\to C$
factors through~$\mathscr X$,
and a open subset~$\mathscr U'$  of~$\mathscr X'\times_C {C_0}$
such that for every point~$v\in C_0$, 
the intersection~$G(F_v)\cap \mathscr U(\mathfrak o_v)$
(taken in \tcr{$\mathscr U(F_v)$}) coincides with the intersection
$G(F_v)\cap\mathscr U'(\mathfrak o_v)$ taken in~$\mathscr X'(\mathfrak o_v)$.
\tcr{We may also assume that $\mathscr U'$
is the complement to a divisor with strict normal crossings in~$\mathscr X'$.}
\tcr{Moreover, $G(F_v)\cap\mathscr U(\mathfrak o_v)$
is non-empty if and only if $\mathscr U'(\mathfrak o_v)$
is non-empty.}
\end{lemm}
\begin{proof}
Up to replacing~$\mathscr U$ by an adequate blow-up,
we may assume that the open immersion \tcr{$i\colon G_F\hra X$} extends
to a morphism~$p\colon \mathscr U\ra\mathscr X$. 
Then, replacing~$\mathscr X$ by some blow-up~$\mathscr X'$ and
$\mathscr U$ by its strict transform~$\mathscr U'$, we may assume that $p$
is flat (\cite{raynaud-gruson1971}, Théorème~5.2.2);
it is then a open immersion. A further blowing-up
allows to assume that $\mathscr X'\setminus\mathscr U'$
is a divisor.
Applying embedded resolution of singularities, we may also assume
that $\mathscr X'$ is smooth over~$k$, that the fibers of its projection
to~$C $ are divisors with strict normal crossings,
as well as $\mathscr X'\setminus\mathscr U'$.

\tcr{%
Finally, if $G(F_v)\cap\mathscr U(\mathfrak o_v)$
is non-empty, then $\mathscr U'(\mathfrak o_v)$ is non-empty as well.
Conversely, assume that $\mathscr U'(\mathfrak o_v)$ is non-empty.
Then $\mathscr U'$ meets the smooth locus of~$\mathscr X'\to C$,
so that $\mathscr U'(\mathfrak o_v)$ has non-empty interior;
in particular, $G(F_v)\cap \mathscr U'(\mathfrak o_v)$ is non-empty.}
\end{proof}

\begin{lemm}\label{lemm.1U}
Let $\mathscr U$ be a flat model of~$G_{\tcb{F}}$ over $C_0$.
For every $v\in C_0(k)$, $\mathscr U(\mathfrak o_v)$
is a bounded definable subset of $G(F_v)$.
For almost all~$v\in C_0(k)$, one has even
$\mathscr U(\mathfrak o_v)=G(\mathbf 0)_v$.
\end{lemm}
\begin{proof}
\tcr{We may assume that} $\mathscr U$ is an open subset of~$\mathscr X$;
it is then clear that $\mathscr U(\mathfrak o_v)$ is definable
in~$G(F_v)$ and that it equals $\ga^n(\mathfrak o_v)$
for almost all~$v\in C_0(k)$ (Lemma~\ref{lemm.G0}).
Let us now prove its boundedness.

We view the $n$ coordinate functions on~$G_F=\ga[F]^n$ 
as rational functions $f_1,\dots,f_n$ on~$\mathscr U$, 
regular over its generic fiber $\mathscr U_F=G_F$.
Up to resolving the indeterminacies of the~$f_i$ (which
replaces~$\mathscr U$ by some other scheme~$\mathscr U'$ but 
does not change the sets $\mathscr U(\mathfrak o_v)$,
we view the $f_i$ as regular morphisms from~$\mathscr U$
to~$\P^1_C$, such that $f_i^*(\{\infty\})\cap \mathscr U_F=\emptyset$.

Cover $\mathscr U$ by finitely many affine open subsets~$\Spec (A_j)$.
There exists an integer~$r$ such that $\varpi_v^r f_i\in A_j\otimes \mathfrak o_v$
for all~$i$ and~$j$. For every point~$g\in \mathscr U(\mathfrak o_v)$,
there exists~$j$ such that the morphism $g\colon\Spec(\mathfrak
o_v)\ra\mathscr U$ 
restricts to a morphism $\Spec(\mathfrak o_v)\ra \Spec (A_j)$,
because $\mathfrak o_v$ is a local ring.
Then, $\ord_v(f_i(g))\geq -r$,
so that $\mathscr U(\mathfrak o_v)$ is bounded in~$G(F_v)$.

The last assertion follows from the fact that the
equality $\mathscr U_F=G_F$ extends to an
isomorphism over a dense open subset of~$C_0$.
\end{proof}

\subsection{Height zeta functions}

Let $(\lambda_\alpha)$ be a family of positive integers and 
let $\mathscr L$ be the 
line bundle~$\sum_{\alpha\in\mathscr A}\lambda_\alpha \mathscr L_\alpha$
on the chosen good model~$\mathscr X$.
Let $\mathscr U$ be a flat model of~$G_F$ over
the affine curve $C_0=C\setminus S$.
For every integer~$n\in\Z$, 
let $M_n$ be the moduli space of sections~$\sigma\colon C \ra\mathscr X$ 
such that $\sigma(\eta_{C })\in G(F)$,
$\sigma(C_0)\subset\mathscr U$ and
$\deg_{C } (\sigma^*\mathscr L)=n$.
By Proposition~\ref{lemm.hilbert}, this moduli space
exists as a quasi-projective $k$-scheme, and is empty for $n\ll 0$.
The geometric analogue of Manin's \emph{height zeta function} 
is the formal Laurent series in one variable~$T$
with coefficients in $\Mot_k$ given by
\begin{equation}\label{eq.ZT}
 Z_\lambda(T)=\sum_{n\in\Z} [M_{n}]T^{n} \in\GVarM_k\lbra T\rbra [T^{-1}]. 
\end{equation}

As was already the case in number theory, it is convenient
to separate the roles of the various divisors~$\mathscr D_\alpha$
and to introduce a multivariable height zeta function.
So, for every $\mathbf n=(n_\alpha)\in\Z^{\mathscr A}$, let $M_{\mathbf n}$
be the moduli space of sections $\sigma\colon C \ra \mathscr X$
such that $\sigma(\eta_{C })\in G(F)$, $\sigma(C_0)\subset\mathscr U$
and $\deg_{C }(\sigma^*\mathscr L_\alpha)=n_\alpha$
for every $\alpha\in\mathscr A$.
Again by Proposition~\ref{lemm.hilbert}, this moduli space
exists as a quasi-projective $k$-scheme $M_{\mathbf n}$; 
moreover, there exists an integer~$m$ such that $M_{\mathbf n}=\emptyset$
if $n_\alpha<-m$ for some $\alpha\in\mathscr A$.
One then defines the generating series
\begin{equation}\label{eq.ZT2}
 Z(\mathbf T)=\sum_{\mathbf n\in\Z^{\mathscr A}} 
   [M_{\mathbf n}]\mathbf T^{\mathbf n} 
\in \Mot_k \lbra (T_\alpha)\rbra [ \prod_\alpha T_\alpha^{-1} ].
\end{equation}
By definition of~$\mathscr L$, we have
\begin{equation}\label{eq.ZT/ZT2}
 Z_{\lambda}(T) 
= Z( (T^{\lambda_\alpha}) )
=\sum_{m\in\Z} \left( \sum_{\substack{\mathbf n\in\Z^{\mathscr A} \\ \lambda\cdot\mathbf n=m}} [M_{\mathbf n}] \right) T^m  \in \Mot_k \lbra T\rbra [T^{-1}]
.
\end{equation}

\paragraph{}\label{par.B0}
In the sequel, we assume that $\mathscr U$
is an open subset of~$\mathscr X$. Its complement
consists of the union of the divisors $\mathscr D_\alpha$,
and of the vertical  components~$E_\beta$,
for $\beta$ in a finite subset~$\mathscr B^0$ of~$\mathscr B$.
By Lemma~\ref{lemm.replaceU},
this does not restrict the generality.
We then set $\mathscr B^0_v=\mathscr B^0\cap\mathscr B_v$ for
every $v\in C(k)$, and define
\[ \mathscr B_0 = \mathscr B_1 \setminus 
     \left(\bigcup_{v\in C_0} \mathscr B^0_v\right); \]
set also $\mathscr B_{0,v}=\mathscr B_0\cap\mathscr B_v$.
Let $\mathbf m_v\in\N^{\mathscr A}$ and $\beta_v\in\mathscr B_{v}$.
We say that the pair \emph{$(\mathbf m_v,\beta_v)$ is $v$-integral}
if either $v\not\in C_0$,
or if $v\in C_0$,
$\beta_v\in\mathscr B_0$ and $m_{\alpha,v}=0$ for every $\alpha$. 
In other words, the union of the sets $G(\mathbf m_v,\beta_v)$
for all $v$-integral pairs $(\mathbf m_v,\beta_v)$
is equal to 
$\mathscr U(\mathfrak o_v)$ if $v\in C_0$,
and to $G(F_v)$ otherwise.


\subsection{Adelic descriptions}
For every subset~$W$ of $G(\AD_F)$
whose characteristic function is an adelic 
motivic Schwartz-Bruhat function,
the intersection $G(F)\cap W$
is represented by a constructible set~$[W]$ over~$k$.
Our goal now is to describe a family of adelic
sets $G(\mathbf m,\beta)$ which will allow us to recover
the constructible sets~$M_{\mathbf n}$.

Let $\mathbf m=(\mathbf m_v)_v$
and $\beta=(\beta_v)_v$ be families
indexed by $v\in C(k)$,
where $\mathbf m_v=(m_{\alpha,v})\in\N^{\mathscr A}$ and $\beta_v\in \mathscr B_v$
for all~$v$.
We say that \emph{$(\mathbf m,\beta)$ is integral} if 
$(\mathbf m_v,\beta_v)$ is $v$-integral  for every $v$.
For each family $(\mathbf m,\beta)$,
define a set
\[ G(\mathbf m,\beta)
= \prod_{v\in C (k)} G(\mathbf m_v,\beta_v)\]
in the  product of all $G(F_v)$.
If $(\mathbf m,\beta)$ is integral, then 
the characteristic function of $G(\mathbf m,\beta)$
is an adelic motivic Schwartz--Bruhat function, because
then $G(\mathbf m_v,\beta_v)\subset G(\mathbf 0)_v=\ga^n(\mathfrak o_v)$ 
for almost all~$v\in C_0(k)$
(Lemma~\ref{lemm.G0}).

For every $g\in G(F)\cap G(\mathbf m,\beta)$,
one has
\[ \deg_{C } \sigma_g^*(\mathscr D_\alpha)
  = \sum_{v\in C (k)}   m_{\alpha,v} , \]
  and
\[ \deg_{C } \sigma^*_g(\mathscr L_\alpha)
 = \sum_{v\in C (k)} \left( m_{\alpha,v} + e_{\alpha,\beta_v}\right). \]
Such a point~$g$ defines an integral point of~$\mathscr U(C_0)$
if and only if $(\mathbf m,\beta)$ is integral.

To shorten the notation, define, for every $v\in C(k)$,
every $\alpha\in\mathscr A$,
every $\mathbf m_v\in\N^{\mathscr A}$ and every $\beta_v\in\mathscr B_v$
such that $(\mathbf m_v,\beta_v)$ is $v$-integral,
\begin{equation}
\norm{\mathbf m_v,\beta_v}_\alpha= 
m_{\alpha,v} + 	  e_{\alpha,\beta_{\tcb{v}}} 
\qquad\text{and}\qquad
\mathbf T^{\norm{\mathbf m_v,\beta_v}} = \prod_{\alpha\in\mathscr A}
   T_\alpha^{\norm{\mathbf m_v,\beta_v}_\alpha}. \end{equation}
Similarly, for 
every $\mathbf m=(\mathbf m_v)_{v\in C }$ and $\beta=(\beta_v)$
such that $(\mathbf m,\beta)$ is integral, set
\begin{equation}
\norm{\mathbf m,\beta}_\alpha=
\sum_v \norm{\mathbf m_v,\beta_v}_\alpha 
\qquad \text{and}\qquad
\mathbf T^{\norm{\mathbf m,\beta}} = \prod_{\alpha\in\mathscr A}
   T_\alpha^{\norm{\mathbf m,\beta}_\alpha}. \end{equation}

For every subset~$W$ of $G(\AD_F)$
whose characteristic function is an adelic 
motivic Schwartz-Bruhat function,
such as the sets~$G(\mathbf m,\beta)$, the intersection $G(F)\cap W$
is represented by a constructible set~$[W]$ over~$k$.
Consequently,  one has the following adelic description
of the height zeta function $Z(\mathbf T)$ defined by~\eqref{eq.ZT2}:
\begin{equation}\label{eq.ZT.adelic}
Z(\mathbf T)
= \sum_{\text{$(\mathbf m,\beta)$ integral}}
  [G(\mathbf m,\beta)]
          \mathbf T^{\norm{\mathbf m,\beta}}.
\end{equation}

 
We shall prove our main theorem in the next section
by applying the motivic Poisson summation  formula
(Theorem~\ref{theo.poisson})
to each term $[G(\mathbf m,\beta)]$,
assuming the analysis of the local Fourier transforms
of the sets $G(\mathbf m_v,\beta_v)$ in $G(F_v)$.
This local analysis is postponed to Section~\ref{sec.local}
and will use computations of ``motivic oscillatory integrals''
which are the topic of Section~\ref{sec.motivic-osc}.

\section{Proof of the theorem}\label{s.proof}
 
\Subsection{Application of the motivic Poisson summation formula}

Let $W$ be any subset of $G(\AD_F)$ whose characteristic
function $\mathbf 1_W$ is an adelic Schwartz-Bruhat function.
The motivic Fourier transform of~$\mathbf 1_W$,
denoted $\mathscr F(\mathbf 1_{W}, \cdot) $
is also a Schwartz-Bruhat function  on the ``dual'' group~$G(\AD_F)$.
Using Hrushovski-Kazhdan's suggestive notation of 
``sum over $F$-rational points'', 
the motivic Poisson summation formula  (Theorem~\ref{theo.poisson})
is the equality
\begin{equation}
[W]=\sum_{x\in\ga^n(F)}\mathbf 1_W(x) = \Lef^{(1-g) n} 
 \sum_{\xi\in\ga^n(F)} \mathscr F(\mathbf 1_W, \xi). 
\end{equation}
(Recall that $g$ is the genus of~$C $.)
Recall also that when $W$ is of the form $\prod W_v$,
the Fourier transform $\mathscr F(\mathbf 1_W,\cdot)$
can be written as a product
of local Fourier transforms at all points~$v$ of~$C $, 
\[\mathscr F(\mathbf 1_{W},\cdot)
= \bigotimes_{v\in C } \mathscr F_v(\mathbf 1_{W_v}, \cdot) 
;\]
in this expression, almost all factors are equal to~$1$.

We apply this formula  to each of the adelic sets $G(\mathbf m,\beta)$,
where $(\mathbf m,\beta)$ is integral.
From Equation~\eqref{eq.ZT.adelic}, we thus get
\begin{align*}
Z(\mathbf T) 
& = \sum_{\text{$(\mathbf m,\beta)$ integral}}[G(\mathbf m,\beta)] 
  \mathbf T^{\norm{\mathbf m,\beta}} \\
& = 
\sum_ {\text{$(\mathbf m, \beta)$ integral}}
\sum_{x\in  G(F)} \mathbf 1_{G(\mathbf m,\beta)}(x)  
\mathbf T^{\norm{\mathbf m,\beta}} \\
& = \Lef^{(1-g) n} 
\sum_ {\text{$(\mathbf m, \beta)$ integral}}
 \sum_{\xi\in G(F)} \mathscr F(\mathbf 1_{G(\mathbf m,\beta)},\xi) \mathbf T^{\norm{\mathbf m,\beta}} .  
\end{align*} 
Let us define a Laurent series $Z(\mathbf T,\cdot)$ whose coefficients are adelic Schwartz-Bruhat function by the formula
\begin{equation}\label{eq.ZT-xi}
Z(\mathbf T,\xi) 
= \sum_ {\text{ $(\mathbf m, \beta)$ integral}}
\mathscr F(\mathbf 1_{G(\mathbf m,\beta)},\xi)
\mathbf T^{\norm{\mathbf m,\beta}} .
\end{equation}
With this notation, 
the height zeta function~\eqref{eq.ZT2} can be rewritten as
\begin{equation}\label{eq.ZT-poisson}
Z(\mathbf T)  =
\Lef^{(1-g) n} \sum_{\xi\in G(F)}  Z(\mathbf T,\xi).
\end{equation}
In this formula, ``summation over $F$-rational points'' of a Laurent series
has to be understood termwise. 

\Subsection{Restriction of the summation domain}

The following lemma shows that the coefficients
of the Laurent series $Z(\mathbf T,\xi)$  given by Equation~\eqref{eq.ZT-xi}
are ``uniformly'' adelic Schwartz-Bruhat  functions.

\begin{lemm}\label{lemm.xi-E}
There exists a finite dimensional $k$-vector space~$E$,
a \tcr{linear} $F$-morphism $\mathbf a\colon E_F\ra G_F$, and
a finite subset $\Sigma\subset C (k)$ containing~$S$
and satisfying the following
properties: for every integral $(\mathbf m,\beta)$ 
and every $\xi\in G(F)$,
\begin{itemize}
\item If $\xi\notin \mathbf a(E(k))$, then there exists
$v\in C $ such that $\mathscr F_v(1_{G(\mathbf m,\beta)},\xi)=0$;
\item If $\xi\in\mathbf a(E(k))$ and $v\not\in \Sigma$, then
$\mathscr F_v(1_{G(\mathbf m,\beta)},\xi)=1$.
\end{itemize}
\end{lemm}
\begin{proof}
With the notation from Lemma~\ref{lemm.almost-all},
there is, for every point~$v\in C (k)$,
an integer~$r_v$  such that 
the characteristic function of the definable
set $G(\mathbf m_v,\beta_v)$ in $G(F_v)$
is invariant under the action of the subgroup~$G(\mathfrak m_v^{r_v})$.
Consequently, its Fourier transform vanishes outside of the
orthogonal of this subgroup. 
Let $\sum a_v[v]$ be the divisor 
of the global differential form
in~$\Omega_{F/k}$ that has been used to define the global
Fourier transform. For almost all points~$v$, one has $a_v=0$.
Moreover, the orthogonal of $G(\mathfrak m_v^{r_v})$
contains $G(\mathfrak m_v^{-r_v+a_v})$. For every $v\in C (k)$,
set $s_v=-r_v+a_v$;
one has $s_v=0$ for all but finitely many points~$v\in C (k)$.
By the Riemann--Roch theorem, the space~$E$ of points~$\xi\in G(F)$
such that $\xi_v\in G(\mathfrak m_v^{s_v})$ for all~$v$
is a finite dimensional $k$-vector space. 
This proves the first part of the Lemma.

Moreover, for every $(\mathbf m,\beta)$ and every $v\in C_0(k)$
such that $\mathbf m_v=0$ and $\mathscr B_v$ is a singleton,
then
the subset $G(\mathbf m_v,\beta_v)$ of~$G(F_v)$
 identifies with $G(\mathfrak o_v)$; if, moreover,
$a_v=0$, then
the characteristic function of $G(\mathfrak o_v)$ is self-dual.
Up to enlarging the set~$\Sigma$, this implies the second assertion.
\end{proof}

This suggests to introduce, for every place~$v\in \Sigma$,
a Laurent series whose coefficients are 
motivic Schwartz-Bruhat functions  on $G(F_v)$ by
\begin{equation}\label{eq.ZvT}
 Z_v(\mathbf T,\cdot) = \sum_{\text{$(\mathbf m_v,\beta_v)$ integral}}
    \mathscr F_v(\mathbf 1_{G(\mathbf m_v,\beta_v)},\cdot) \mathbf T^{\norm{\mathbf m_v,\beta_v}}. 
\end{equation}
\tcr{By Lemma~\ref{lemm.xi-E}, one has 
$\mathscr F(1_{G(\mathbf m,\beta)},\xi)=0$
if $\xi\not\in \mathbf a(E(k))$,
while
$\mathscr F(1_{G(\mathbf m,\beta)},\xi)=\prod_{v\in\Sigma}
\mathscr F_v(1_{G(\mathbf m,\beta)},\xi)$ otherwise.
Consequently, one has
\begin{equation}\label{eq.Z(T)}
 Z(\mathbf T) = \Lef^{(1-g)n} \sum_{\xi\in  \mathbf a(E(k))}  \prod_{v\in \Sigma} Z_v(\mathbf T,\cdot) .
\end{equation}
}

\Subsection{Local results}
\label{ss.local-results}

In all of this section, we fix a point~$v\in\Sigma$
and state the properties of the Laurent series
$Z_v(\mathbf T,\cdot)$\tcr{, and of its specialization 
$Z_{\lambda,v}(T,\cdot)=Z((T^{\lambda_\alpha}),\cdot)$.}
\tcr{They will be proved in Section~\ref{sec.local}.
We fix a finite dimensional $k$-vector space~$E$
and a linear $F$-morphism $\mathbf a\colon E_F\ra G_F$
satisfying the conditions of Lemma~\ref{lemm.xi-E}.}
\tcr{Recall also that $\mathscr U$ is the good model of~$G_F$ over~$C$
of which we study the integral sections of bounded height.}

\begin{prop}\label{prop.almost}
Assume that $v\in \Sigma\cap C_0$.
Then $Z_v(\mathbf T,\cdot)$ is a polynomial in~$\mathbf T$.
\tcr{Moreover, $Z_{\lambda,v}(\Lef^{-1},0)$} is 
a \tcr{non-zero} effective element
of $\Mot_k$.
\end{prop}


 The following result is a motivic analogue of  Proposition~4.6
of~\cite{chambert-loir-tschinkel2010}.
In that paper, 
some formalism of ``residue measures''
was introduced, which is useful  for  \emph{describing}
the kind of integrals that appear in the right hand side.
Observe indeed that this is a sum of motivic integrals
on arc spaces $\mathscr L_v(\mathscr D_A)$ attached to
the faces of dimension~$d$ of the analytic Clemens  complex
of~$(X,D)$ at the place~$v$.

\begin{prop}\label{prop.v.trivial}
Assume that $v\in C \setminus C_0$. Then the Laurent series
$Z_v(\mathbf T,0)$ is a rational function. 
More precisely, there exists a family $(P_{v,A})$ of Laurent polynomials with
coefficients in~${\Mot}_k$, a family $(u_{v,A})$
of motivic, integer valued functions, indexed by
the set of maximal faces~$A$ of the analytic Clemens
complex~$\Clan_v(X,D)$ such that
\[Z_v(\mathbf T,0) = \sum_{A\in\Clanmax_{v}(X,D)}
P_{v,A}(\mathbf T)
\prod_{\alpha\in A}
   \frac1{1-\Lef^{\rho_\alpha-1}T_\alpha} \]
and
\[ P_{v,A}(\mathbf T) \equiv  (1-\Lef^{-1})^{\Card(A)}
\int_{\mathscr L_v(\mathscr D_A)} \Lef^{u_{v,A}(x)} \,\mathrm dx \]
modulo the ideal generated by the polynomials
$1-\Lef^{\rho_\alpha-1}T_\alpha$, for $\alpha\in A$.
\end{prop}

\begin{coro}\label{coro.v.trivial}
Assume that $v\in C\setminus C_0$ and that $\lambda=(\rho_\alpha-1)_\alpha$.
Let $d_v=1+\dim\Clan_v(X,D)$.
The Laurent series
$Z_{\lambda,v}(T,0)$
in the variable~$T$ is a rational function.
More precisely, for every \tcr{non-zero} 
common multiple~$a$ of the \tcr{integers}~$\rho_\alpha-1$,
\tcr{for $\alpha\in\mathscr A$, then}
$P_{\lambda,v}(T) = (1-\Lef^a T^a)^{d_v}Z_{\lambda,v}(T,0)  $
belongs to $\tcb{{\Mot}}_k[T,T^{-1}]$
and satisfies
\[P_{\lambda,v} (\Lef^{-1}) = (1-\Lef^{-1})^{{d_v}} 
\sum_{\substack{A\in\Clanmax(X,D) \\ \Card(A)=d_v}} 
\prod_{\alpha\in A} \frac{a}{\rho_\alpha-1}
\int_{\mathscr L_v(\mathscr D_A)} \Lef^{u_{v,A}(x)}\,\mathrm dx. \]
\end{coro}

\begin{prop}\label{prop.v.nontrivial}
Let $v\in C \setminus C_0$ and let $d_v=1+\dim\Clan_v(X,D)$.
There exists a constructible partition $(U_{v,i})$ of $E\setminus\{0\}$
and, for every $i$, 
an element $P_{v,i}\in\ExpMot_{U_{v,i}}[\mathbf T,\mathbf T^{-1}]$
and finite families $(a_{v,i,j})$, $(b_{v,i,j})$
where $a_{v,i,j}\in\N$, $b_{v,i,j}\in\N^{\mathscr A}$,
such that the restriction to $U_{v,i}$ of $ Z_v(\mathbf T,\mathbf a(\cdot)) $
equals
\[    \prod_j (1-\Lef^{a_{v,i,j}}\mathbf T^{b_{v,i,j}})^{\tcb{-1}} P_{v,i}(\mathbf T;\cdot). \]
Moreover, assuming that $\lambda=(\rho_\alpha-1)_\alpha$,
there exist integers~$a_{v,i}\geq 1$ and $d_{v,i}\in [0,d_v-1]$ such that  
the restriction to~$U_{i,v}$ of
$ (1-(\Lef T)^{a_{v,i}})^{d_{v,i}} Z_{\lambda,v} (T,\mathbf a(\cdot)) $
belongs to  $\ExpMot_{U_{v,i}}\{T\}^\dagger$. 
\end{prop}

For a moment, we take these three propositions 
for granted and complete the proof of Theorem~\ref{theo.main}.

\Subsection{Conclusion: Proof of Theorem~\ref{theo.main}}

\tcr{Recall from Equation~\eqref{eq.Z(T)}
that our goal is to evaluate the sum
\[  Z(\mathbf T) = \Lef^{(1-g)n} \sum_{\xi\in  \mathbf a(E(k))}  \prod_{v\in \Sigma} Z_v(\mathbf T,\cdot) . \]
}

For every $v\in C\setminus C_0$, let $d_v=1+\dim\Clan_v(X,D)$;
let also $d=\sum_{v\in C\setminus C_0} d_v$.
Propositions~\ref{prop.almost}, \ref{prop.v.trivial} 
and~\ref{prop.v.nontrivial}
show that for every $\xi\in\mathbf a(E(k))$, 
the Laurent series
$Z(\mathbf T,\xi)=\prod_{v\in C} Z_v(\mathbf T,\xi)$ 
with coefficients in~$\ExpMot_k$
is a rational function of~$\mathbf T$, and admits
a denominator of the form $\prod (1-\Lef^a \mathbf T^b)$.

We set
\[ \Clanmax_\infty (X,D) = \prod_{v\in C\setminus C_0}\Clanmax_v(X,D). \]
For $\xi=0$,
with the notation of Proposition~\ref{prop.v.trivial}, one has
\[ Z(\mathbf T,0) = \sum_{A=(A_v)\in \Clanmax_\infty(X,D)}
   \prod_{v\in C\setminus C_0} \prod_{\alpha\in A_v}\frac1{1-\Lef^{\rho_\alpha-1}T_\alpha}
 P_{v,A_v}(\mathbf T) \prod_{v\in C_0} Z_v(\mathbf T,0). \]
In particular, if $\lambda=(\rho_\alpha-1)_\alpha$, one has
\begin{align*}
 Z_\lambda(T,0) & = Z((T^{\rho_\alpha-1}),0) \\
&= \sum_{A=(A_v)\in \Clanmax_\infty(X,D) }
   \prod_{v\in C\setminus C_0} \prod_{\alpha\in A_v}\frac1{1-(\Lef T)^{\rho_\alpha-1}}
 P_{v,A_v}((T^{\rho_\alpha-1})) \prod_{v\in C_0} Z_v((T^{\rho_\alpha-1}),0) \\
&= \sum_{A=(A_v)\in \Clanmax_\infty(X,D) }
 P_{A}(T)
   \prod_{v\in C\setminus C_0} \prod_{\alpha\in A_v}\frac1{1-(\Lef T)^{\rho_\alpha-1}}
\end{align*}
where the polynomial~$P_A\in\Mot_k[T]$ is defined by 
\[ P_A(T) =  \prod_{v\in C\setminus C_0} P_{v,A_v}((T^{\rho_\alpha-1})) 
\prod_{v\in C_0} Z_v((T^{\rho_\alpha-1}),0). 
\]
Consequently, $Z_\lambda(T,0)$
is both a rational function, and an element of $\Mot_k\{T\}$; moreover,
\begin{multline*} (1-\Lef^aT^a)^d Z_\lambda(T,0) \\
 = \sum_{A=(A_v)\in\Clanmax_\infty(X,D)}
 P_{A}(T)
   \prod_{v\in C\setminus C_0} (1-(\Lef T)^a)^{d_v-\Card(A_v)}
   \prod_{\alpha\in A_v}  \frac{1-(\Lef T)^a}{1-(\Lef T)^{\rho_\alpha-1}}
.
\end{multline*}
The right hand side of the preceding formula is a polynomial in~$T$
with coefficients in~$\tcb{\Mot}_k$; when one sets $T=\Lef^{-1}$, 
only the terms remain for which $\Card(A_v)=d_v$ for every~$v$, and one gets
\begin{equation}
\label{eq.Z(L-1)}
\sum_{\substack{A=(A_v)\in\Clanmax_\infty(X,D)\\ \Card(A_v)=d_v}}
P_A(\Lef^{-1})
\prod_{v\in C\setminus C_0}\prod_{\alpha\in A_v}\frac{a}{\rho_\alpha-1}
. 
\end{equation}
It then follows from Propositions~\ref{prop.almost}
and~\ref{prop.v.trivial} 
that this is an effective element of~$\Mot_k$,
which is nonzero since, by assumption,
$\mathscr U(\mathfrak o_v)$ is non-empty for every $v\in C_0$.
In this case, one concludes that $Z_\lambda(T,0)$ 
has a pole of order exactly~$d$ at $T=\Lef^{-1}$.

For $\xi\neq 0$, one deduces in a similar way
from Proposition~\ref{prop.v.nontrivial}
that the Laurent series $Z(\mathbf T,\xi)$ is rational,
as well as its specializations.
By uniformity, the same property holds 
when one takes the sum over $F$-rational points,
so that the height zeta function $Z_\lambda(T)$
is a rational function which belongs to $\ExpMot_k\{T\}$.

Since $\Lef$ belongs to~${\Mot}_k$
and the natural map from~${\Mot}_k$ to~${\ExpMot}_k$ is injective
(Lemma~\ref{lemm.injective}),
the power series $Z(\mathbf T)$
is rational when viewed as a Laurent series with coefficients in~${\Mot}_k$.
In particular, the specialization $Z_\lambda(T)$ 
is a rational function too.

For every $\xi\neq 0$, the specialization
$Z_\lambda(T,\xi)=Z((T^{\rho_\alpha-1}),\xi)$
is a rational function, and an element of~$\ExpMot_k\{T\}$.
Moreover, Proposition~\ref{prop.v.nontrivial}
asserts that the order of its pole at $T=\Lef^{-1}$
is strictly smaller than~$d$.
Taking for the integer~$a$ any common multiple of the~$\rho_\alpha-1$
and of the integers~$a_{v,i}$ appearing in the statement of 
Proposition~\ref{prop.v.nontrivial}, 
and summing over rational points~$\xi\in G(F)$,
we obtain that
\[ (1-(\Lef T)^a)^d Z_\lambda(T) \in \ExpMot_k\{T\}^\dagger \]
and its value at $T=\Lef^{-1}$ is given by 
Equation~\eqref{eq.Z(L-1)}, multiplied by~$\Lef^{(1-g)n}$.

This concludes the proof of Theorem~\ref{theo.main}.

\section{Motivic oscillatory integrals}
\label{sec.motivic-osc}\label{s.oscill}

In this section, we consider a field~$k$ of characteristic
zero and let $K$ be the local field~$ k\lpar t\rpar$.
We write $\ord$ for the valuation of~$K$, normalized
by $\ord(t)=1$, $R$ for the valuation ring of $K$ \tcb{and $\mathfrak m$ for its maximal ideal}.
The angular component map~$\ac\colon K\to k$ is the
unique multiplicative map which is trivial on $1+k\lbra t\rbra$,
on~$t$, and maps constants $a\in k$ to themselves.
We also fix a real number $q > 1$ and set $\abs x = q^{- \ord (x)}$.

With the notation of Section~\ref{s.poisson},
let $r\colon K\ra k$ be the linear map, 
given by $r(1)=1$ for $n=0$ and $r(t^n)=0$ otherwise,
so that $r(a)=\res_0(a\,\mathrm dt/t)$.
Set $\mathrm e(\cdot)=\psi(r(\cdot))$; it is an analogue
of a non-trivial character of $R/\mathfrak m$.

\Subsection{Decay of motivic integrals}

\begin{lemm}\label{lem1}
Let $d$ be a positive integer 
and let $\xi\in K$ be such that $\abs\xi= 1$.
Then, for every $a\in K$ and every
$n  \in \mathbf{N}$ such that $\ord(a) +n\leq 0 <\ord(a)+2n$,
one has
\[
\int_{\xi + t^n R} \mathrm e (a x^d) \mathrm dx =0.
\]
\end{lemm}
\begin{proof}
We follow the arguments of Lemma~2.3.1 in~\cite{chambert-loir-tschinkel:2012}.
One can write
\[\int_{\xi+t^n R} \mathrm e(ax^d)\,\mathrm dx
= \Lef^{-n} \int_R \mathrm e(a \xi^d (1+t^n u)^d)\,\mathrm du.\]
For $u\in R$, 
all terms starting from the third one  in the binomial expansion
\[ a\xi^d (1+t^nu) ^d  = a\xi^d + \binom d1 a\xi^d t^n u +\binom d2 a\xi^dt^{2n}u^2+\dots
+\binom dd a\xi^d t^{dn}u^d \]
belong to~$\mathfrak m$,
since $\ord(a) > -2n$ and $\ord(\xi)=0$. Therefore
\[r(a \xi^d (1+t^nu)^d)=r(a \xi^d ) + d r(a \xi^d t^n u) \]
and
\[
\int_R \mathrm e(a \xi^d (1+t^n u)^d)\,\mathrm du
= [\Spec (k), r(a \xi^d )] \int_R \mathrm e (da \xi^d t^n u)\,\mathrm du. \]
Since 
$\ord(a\xi^dt^n)=\ord(a)+n\leq 0$, Proposition~\ref{prop.local-vanishing}
implies that
\[\int_R \mathrm e(da \xi^d t^n u)\,\mathrm du=0, \]
and the lemma follows.
\end{proof}

For $m\in\Z$, let $C_m$ be the annulus defined by $\ord(x)=m$. 
For $d\in\Z$, $d\neq0$, and $a\in K^*$, set
\begin{equation}
I (m, d, a)= \int_{C_m} \mathrm e (a x^d) \,\mathrm dx
\end{equation} 
in $\Mot_k$. 
\begin{lemm}\label{lem2}
The integrals $I(m,d,a)$ satisfy the following properties:
\begin{enumerate}\def\labelenumi{\upshape(\theenumi)}
\item
Let $m\in\Z$ and $d\in\Z$, $d\neq0$. Let $a,b\in K^*$ 
be such that 
$\ord(b)=\ord(a)+md$ and $\ac(b)=\ac(a)\pmod {(k^*)^d}$.
Then $I(m,d,a)=\Lef^{-m}I(0,d,b)$.
\item
Assume that $k$ is algebraically closed.
Then \[ I(m,d,a)=\Lef^{-m} I(0,d,t^{\ord(a)+md}).\]
 In particular,
$I (m, d, a)$ depends only on~$m$, $d$, and~$\ord(a)$.

\item
If $\ord(a)+md<0$, 
then $I (m, d, a) = 0$.

\item 
If $\ord(a) + md>0$, then $I(m,d,a)=\Lef^{-m}(\Lef-1)/\Lef$.
\end{enumerate}
\end{lemm}
\begin{proof}
(1) Let $u\in k^*$ be such that $\ac(a)u^d=\ac(b)$.
By assumption, there exists $v_1\in 1+t k\lbra t\rbra$ such that
$b=au^dt^{md} v_1$; since $k$ has characteristic zero,
there exists $v\in 1+t k\lbra t\rbra$ such that $v_1=v^d$.
Let us make the change of variables $\tcb{x}=uvt^m\tcb{y}$. This gives
\begin{align*}
 I(m,d,a) & =\int_{C_m} \mathrm e(ax^d)\,\mathrm dx
= \Lef^{-m} \int_{C_0} \mathrm e(au^dv^dt^{md} y^d)\,\mathrm dy \\
&= \Lef^{-m} \int_{C_0} \mathrm e(by^d)\,\mathrm dy
= I(0,d,b). \end{align*}

Assertion~(2) follows at once.

Let us prove (3). Since $I(m,d,a)=\Lef^{-m} I(0,d,at^{md})$
we only need to prove that $I(0,d,a)=0$ for $\ord(a) < 0$.
Let $n=-\ord(a)$.
Observe that
\[ I(0,d,a)= \int_{C_0} \mathrm e(ax^d)\,\mathrm dx
= \int_{C_0} \int_{R} {\mathrm e(a(x+t^ny)^d)}\,\mathrm dy\,\mathrm dx
= 0. \]
Since $\ord(a)<0$, $\ord(a)+2n=-\ord(a)>0$,
hence by Lemma~\ref{lem1},
$\int_{x+t^n R} \mathrm e(ay^d)\mathrm dy=0$
for every $x\in K$ such that $\ord(x)=0$.
The statement follows.


(4) It suffices to prove that $I(0,d,a)=(\Lef-1)/\Lef$ for $\ord(a)> 0$.
In this case, one has $r(ax^d)=0$ for every $x\in R^*$,
hence the claim.
\end{proof}

Let $u\in K\lpar x\rpar$ be a Laurent series of positive
radius of convergence; write $u=\sum u_n x^n$.
Let $\mu\in\Z$ be such that $u$ converges on the closed disk~$D_\mu$ defined by 
the inequality $\ord(x)\geq \mu$ deprived from~$0$; in other words, 
$\mu$ is such that $\ord(u_n)+n\mu\to+\infty$ when $n\to+\infty$. 
Let $m$ be an integer such that $m\geq\mu$;
let $\nu\geq 0$
be such that $\ord(u_n) +nm>0 $ for $n>\nu$. By construction,
$ \ord(u_n x^n)=\ord(u_n)+n\ord(x)  >0$ for $n>\nu$ and $\ord(x)=m$,
so that $r(u(x))= r(u^{\nu}(x))$, for $x\in D_\mu$
such that $\ord(x)=m$, where $u^\nu(x)=\sum_{n\leq\nu} u_nx^n$.
Therefore, for $m\geq\mu$, we can define the motivic integrals
$ \int_{C_m} \mathrm e(u(x))\, \mathrm dx $
as given by 
$ \int_{C_m} \mathrm e(u^\nu(x))\,\mathrm dx$
in~$\Mot_k$.
More generally, for every definable subset~$W$ of~$D_\mu$,
one can define $\int_W \mathrm e(u(x))\,\mathrm dx$ 
as an element of a suitable completion of~$\Mot_k$,
and as an element of~$\Mot_k$ itself if $\ord(x)$ is bounded
from above on~$W$.

\begin{prop}\label{prop.annulus-0}
Let $u\in K\lpar x\rpar$ be a Laurent series of positive
radius of convergence; let $d=-\ord_x(u)$ and $a=\lim_{x\to 0}u(x)x^{d}$. 
Assume that $d>0$.
The motivic integrals
\[ \int_{C_m} \mathrm e(u(x))\,\mathrm dx \]
vanish for every large enough integer~$m$; more 
precisely, it suffices that \tcb{$u$ converges
and has no root in the punctured disk defined by $\ord(x)\geq m$},
and that $\ord(a)<md$.
\end{prop}
\begin{proof}
Since $K$ has characteristic zero, there exists $a\in K^*$
and a power series $v\in K\lbra x\rbra$ such that $v(0)=1$
and $u(x)=a x^{-d} v(x)^{-d}$.
Let $m_0\in\N$ be a large enough integer such that
$x^du$ converges on the disk $\{\ord(x)\geq m_0\}$ and 
does not vanish on this disk.
If one writes $u=\sum u_n x^n$, we thus have the following properties:
\begin{itemize}
\item one has $u_{-d}=a$ and $u_n=0$ for $n<-d$;
\item for $n>{-d}$,  $\ord(u_n)+nm_0>\ord(a) $;
\item when $n\to+\infty$, one has $\ord(u_n)+nm_0\to +\infty$.
\end{itemize}
Writing $v=\sum_{n\geq0} v_n x^n$,
it follows that $\ord(v_n)+nm_0>0$ for every $n\in\N_{>0}$.
Consequently, the change of variables $y=xv(x)$ 
maps the annuli $C_m$ to themselves, for $m\geq m_0$,
and preserves the motivic measure.
Therefore, for $m\geq m_0$, one has
\[ \int_{C_m} \mathrm e(u(x))\,\mathrm dx
 = \int_{C_m} \mathrm e(a x^{-d}  v(x)^{-d}) \,\mathrm dx 
=\int_{C_m} \mathrm e(a y^{-d})\,\mathrm dy. \]
According to Lemma~\ref{lem2}, (3),
this integral vanishes if $\ord(a)-md<0$. This concludes 
the proof of the proposition.
\end{proof}

\Subsection{Motivic Igusa integrals with exponentials---the regular case}

\paragraph{Setup}\label{setup-mIie}
Let $\mathscr X$ be a flat $R$-scheme of finite type,
equidimensional of relative dimension~$n$,
let $\mathscr D$ be a relative divisor on~$\mathscr X$.
We assume that $\mathscr X$ is smooth, everywhere of relative
dimension~$n$, and that $\mathscr D$ has strict normal crossings \tcb{over $R$}.
Let also $X=\mathscr X_k$ and $D=\mathscr D_k$
be their special fibers. 
Let $\mathscr A$ be the set of irreducible components of~$\mathscr D$; 
for $\alpha\in\mathscr A$,
let $\mathscr D_\alpha $ be the corresponding irreducible component,
and let $D_\alpha$ be its special fiber.
For every $A\subset\mathscr A$, let $\mathscr D_A=\bigcap_{\alpha\in A}\mathscr D_\alpha$  and let $\mathscr D_A^\circ = \mathscr D_A\setminus \bigcup_{\alpha\notin A}\mathscr D_\alpha$; one defines $D_A$ and $D_A^\circ$ in a similar way.
By definition of a divisor with normal crossings, 
every irreducible component of~$D_A$ has codimension~$\Card(A)$.

For every constructible subset~$W$ of~$X$, 
let $\mathscr L(\mathscr X;W)$
be the constructible subset of~$\mathscr L(\mathscr X)$  parameterizing arcs
$x\in\mathscr X(R)$ \tcb{whose origin lies in} $W$. 
For every $\mathbf m\in \N^{\mathscr A}$, we write $W(\mathbf m)$
for the constructible subset of $\mathscr L(\mathscr X)$ consisting
of arcs~$x$ such that $\ord_{\mathscr D_\alpha}(x)= m_\alpha$
for every $\alpha\in \mathscr A$.

Let $h$ be a motivic residual function on~$\mathscr L(\mathscr X)$.
Let $f$ be a meromorphic function on~$\mathscr X$ 
such that the polar divisor $\div_\infty(f_K)$
of the restriction~$f_K$ to~$\mathscr X_K$ is contained in the union 
$\bigcup_{\alpha\in \mathscr A} \mathscr D_{\alpha,K}$.
Let $(d_\alpha)_{\alpha\in\mathscr A}$ be nonnegative integers
such that  on~$\mathscr X_K$,
\begin{equation}
 \div_\infty(f_K) = \sum_{\alpha\in\mathscr A} d_\alpha \mathscr D_{\alpha,K}.
\end{equation}

For a family $\mathbf T=(T_\alpha)_{\alpha\in\mathscr A}$ of indeterminates, define
the \emph{motivic Igusa integral with exponentials}
\begin{equation}\label{eq.motivic-Igusa}
 Z(\mathscr X,h\mathrm e(f);\mathbf T)  = 
\int_{\mathscr L(\mathscr X)}  \prod_{\alpha\in\mathscr A} T_\alpha^{\ord_{\mathscr D_\alpha}(x)} h(x)\,\mathrm e(f(x))\,\mathrm dx ,\end{equation}
a power series in~$\mathbf T$ with coefficients in ${\ExpMot}_k$. 
Although $f$ is only a rational function on~$\mathscr X$,
note that $r(f)$  is a well defined residual function on~$W(\mathbf m)$ 
for each $\mathbf m\in\N^{\mathscr A}$, so that we have
\begin{equation} 
 Z(\mathscr X,h\mathrm e(f);\mathbf T)  = 
  \sum_{\mathbf m\in\N^{\mathscr A}}
  \prod_{\alpha\in \mathscr A} T_\alpha^{m_\alpha}
  \int_{W(\mathbf m)} 
   h(x) \mathrm e(f(x))\,\mathrm dx.
\end{equation}
This power series is an analogue of the classical motivic Igusa zeta integrals
which would correspond to the case $f=0$.

More generally, for every subset $A\subset\mathscr A$, let
\begin{equation}
 Z_A(\mathscr X,h\mathrm e(f);\mathbf T)  = 
\int_{\mathscr L(\mathscr X;D_A^\circ)}  \prod_{\alpha\in\mathscr A} T_\alpha^{\ord_{\mathscr D_\alpha}(x)} h(x)\,\mathrm e(f(x))\,\mathrm dx .\end{equation}
When $A$ runs among~$\Clan(X,D)$, the subsets $\mathscr L(\mathscr X;D_A^\circ)$
form a partition of~$\mathscr L(\mathscr X)$ into constructible subsets
and we decompose the motivic integral defining~$Z(X,h\mathrm e(f);\mathbf T)$
as the sum of  motivic integrals over each of them,
so that
\[ Z(\mathscr X,h\mathrm e(f);\mathbf T) = \sum_{A\subset\mathscr A}
 Z_A(\mathscr X,h\mathrm e(f);\mathbf T)  .\] 
For every $\mathbf m\in\N_{>0}^A$, let  $W_A(\mathbf m)$  be the 
constructible subset of~$\mathscr L(\mathscr X;D_A^\circ)$ 
defined by the conditions
$\ord_{\mathscr D_\alpha}(x)=m_\alpha$ for $\alpha\in A$
and $\ord_{\mathscr D_\alpha}(x)=0$ for $\alpha\not\in A$.
With this notation, one has
\[ 
 Z_A(\mathscr X,h\mathrm e(f);\mathbf T)  = 
  \sum_{\mathbf m\in\N_{>0}^{A}} \prod_{\alpha\in A} T_\alpha^{m_\alpha}
  \int_{W_A(\mathbf m)} h(x) \mathrm e(f(x))\,\mathrm dx.\]


\begin{lemm}\label{lemm.hensel}
Let $A$ be a subset of~$\mathscr A$ and let $B$ be a set of
cardinality equal to $n-\Card(A)$.
There exists a measure-preserving definable isomorphism~$\theta$
from $D_A^\circ\times \mathscr L(\Aff^1;0)^A\times \mathscr L(\Aff^1)^{B} $,
with coordinates $x_\alpha$ (for $\alpha\in A$) and $y_\beta$
(for $\beta\in B$),
to $\mathscr L(\mathscr X;D_A^\circ)$ 
such that $\ord_{\mathscr D_\alpha}(\theta(x))=\ord(x_\alpha)$
for $\alpha\in A$, and $\ord_{\mathscr D_\alpha}(\theta(x))=0$ for $\alpha\notin A$.
\end{lemm}
\begin{proof}
This is a standard fact in the theory of motivic zeta functions.
We may assume that $\mathscr D_\alpha=\emptyset$ for $\alpha\not\in A$,
and that there exist regular functions $u_\alpha$
(for $\alpha\in A$) on~$\mathscr X$ such that $\div(u_\alpha)
=\mathscr D_\alpha$.
By definition of a divisor with strict normal crossings,
the morphism $u=(u_\alpha)\colon \mathscr X\to (\Aff^1)^A$
is then smooth.
Hence we may assume that there exists regular functions
$v_\beta$ (for $\beta\in B$) in~$\mathscr X$
such that the morphism $(u,v)=((u_\alpha);(v_\beta))$
from $\mathscr X$ to~$(\Aff^1)^A\times (\Aff^1)^B$
is étale.  
Both of these assumptions  are only valid up
to replacing~$\mathscr X$ by a Zariski dense open subset
containing any given point of the special fibre.
Since we only seek for a definable isomorphism, they do
not restrict the generality.

It then follows from the definition of an étale morphism
that the induced morphism
$\mathscr L(\mathscr X;D_A^\circ)\to 
D_A^\circ\times \mathscr L(\Aff^1;0)^A\times\mathscr L(\Aff^1)^B$
is an isomorphism. 
It preserves the motivic measure by construction of  latter.
Moreover, denoting the standard coordinates on
$\mathscr L(\Aff^1;0)$ by $x_\alpha$, for $\alpha\in A$, this
isomorphism maps the definable function $\ord_{\mathscr D_\alpha}$
to the function $\ord(x_\alpha)$.
\end{proof}

In this section and the next one, we study the rationality 
and the poles of the Igusa integral with exponentials.
We first consider the special case where $f$ is regular
on the generic fiber~$\mathscr X_K$.

\begin{prop}\label{prop.ZA-rational}
Assume that $f_K$ is regular on the generic fiber~$\mathscr X_K$.
Let $A$ be a subset of~$\mathscr A$.
The power series $Z_A(\mathscr X,h\mathrm e(f);\mathbf T)$ 
is a rational function.
More precisely, there exists a  polynomial $Q_A$
with coefficients in~$\ExpMot_k$, 
such that
\[ Z_A(\mathscr X,h\mathrm e(f);\mathbf T) =  Q_A(\mathbf T)
 \prod_{\alpha\in A} 
 \frac{1}{1-\Lef^{-1}T_\alpha}  \]
and such that
\begin{equation}
 Q_A(\mathbf T)- 
 (1-\Lef^{-1})^{\Card(A)}
 \int_{\mathscr L(\mathscr D_A^\circ)} h(x)\mathrm e(f(x))\,\mathrm dx
\end{equation}
belongs to the ideal generated by the polynomials $1-\Lef^{-1}T_\alpha$,
for $\alpha\in A$.
\end{prop}
\begin{proof}
By Lemma~\ref{lemm.hensel}, there is a
measure-preserving definable isomorphism
from $\mathscr L(\mathscr X;D_A^\circ)$ 
to $D_A^\circ\times\mathscr L(\Aff^1;0)^A\times
\mathscr L(\Aff^1)^{B} $,  where $B$ is some set of cardinality
$n-\Card(A)$,  with coordinates $x_\alpha$ (for $\alpha\in A$), $y_\beta$
(for $\beta\in B$)
under which $\ord_{\mathscr D_\alpha}(x)=\ord(x_\alpha)$
for $\alpha\in A$, and $\ord_{\mathscr D_\alpha}(x)=0$ for $\alpha\notin A$.
In the sequel, we identify a point $x\in\mathscr L(\mathscr X;D_A^\circ)$
with a triple $(\xi,x,y)$, where $\xi\in D_A^\circ$,
$x\in\mathscr L(\Aff^1;0)^A$ and $y\in\mathscr L(\Aff^1)^B$.

Fix an integer~$a$ and a regular function~$g$ on~$\mathscr X$ 
such that $f=t^a g$.
On $\mathscr L(\mathscr X;D_A^\circ)$, we can expand the function~$g$
as a power series 
\[ g_A (x,y)= \sum_{\substack{p\in\N^A\\ q\in\N^B}} g_{p,q} x^p y^q,\]
where $g_{p,q}\in \mathscr O(D_A^\circ)\lbra t\rbra $.
Then
\[ \ord(g_A(x,y)-g_A(0,y))\geq \min_{\alpha\in A}\ord(x_\alpha). \]
In particular, we see that $r(t^a g_A(x,y))=r(t^a g_A(0,y))$
if $a+\min_{\alpha\in A}\ord(x_\alpha)>0$.
Let $\mu$ be a positive integer such that  $\mu > -a$
and such that the Schwartz-Bruhat function~$h$ 
factors through $\mathscr L_\mu(\mathscr X)$.

If one has $m_\alpha\geq\mu$ for every $\alpha\in A$, 
it then follows that
\begin{align*}
\int_{W_A(\mathbf m)} h(x)\mathrm e(f(x))\,\mathrm dx &=
\big(\prod_{\alpha\in A} \Lef^{-m_\alpha}\big)
    \int_{\substack{\ord(x'_\alpha)=0 \\ x_\alpha=t^{m_\alpha} x'_\alpha }}
          h(\xi,x,y)\mathrm e(t^a g_A(x,y))\,\mathrm dx'\mathrm dy \\
&=  \big(\prod_{\alpha\in A} \Lef^{-m_\alpha}\big)
       \int_{\substack{\ord(x'_\alpha)=0 \\ x_\alpha=t^{m_\alpha}x'_\alpha}}
          h(\xi,0,y)\mathrm e(t^ag_A(0,y))\,\mathrm dx' \mathrm dy\\
& = \prod_{\alpha\in A} \big(\Lef^{-m_\alpha} (1-\Lef^{-1})\big)
   \int_{ D_A^\circ\times \mathscr L(\Aff^1)^B} 
          h(\xi,0,y)\mathrm e(t^ag_A(0,y))\, \mathrm dy\\
& = \prod_{\alpha\in A} \big(\Lef^{-m_\alpha} (1-\Lef^{-1})\big)
 \int_{\mathscr L(\mathscr D_A;D_A^\circ)} h(x)\mathrm e(f(x))\,\mathrm dx.
\end{align*}
In general, let 
\[ A_1=\{\alpha\in A\sozat m_\alpha<\mu\} 
\quad\text{and}\quad
A_2=\{\alpha\in A\sozat m_\alpha\geq\mu\} , \]
so that $A$ is the disjoint union $A=A_1\cup A_2$.
Write $x=(x_1,x_2)$, where $x_1=(x_\alpha)_{\alpha\in A_1}$
and $x_2=(x_\alpha)_{\alpha\in A_2}$,
and split $\mathbf m=(\mathbf m_1,\mathbf m_2)$ accordingly.
Analogously, one has 
\[
\int_{W_A(\mathbf m)} h(x)\mathrm e(f(x))\,\mathrm dx  
= \prod_{\alpha\in A_2} \big(\Lef^{-m_\alpha} (1-\Lef^{-1})\big)
 \int_{W'_{A_2}(\mathbf m_1)}  h(x)\mathrm e(f(x))\,\mathrm dx,
\]
where $W'_{A_2}(\mathbf m_1)$ is the definable
subset of~$\mathscr L(\mathscr D_{A_2})$ consisting
of arcs $x$ on~$\mathscr D_{A_2}$ \tcb{with origin on~$D_{A_2}$ and}
such that $\ord_{\mathscr D_\alpha}(x)=m_\alpha$
for $\alpha\in A_1$.
We can then write
\begin{multline*}
Z_A(\mathscr X,h\mathrm e(f);\mathbf T)  
= \sum_{\mathbf m\in\N_{>0}^A} \prod_{\alpha\in A}
  T_\alpha^{m_\alpha} \int_{W_A(\mathbf m)} h(x) \mathrm e(f(x))\,\mathrm dx \\
= \sum_{\substack{A_1\subset A \\ A_2=A\setminus A_1}}
\sum_{\mathbf m_1\in (0,\mu)^{A_1}}
 \prod_{\alpha\in A_1} T_\alpha^{m_\alpha}
\sum_{\mathbf m_2\in [\mu,\infty)^{A_2}}
  \prod_{\alpha\in A_2} (1-\Lef^{-1}) ( \Lef^{-1}T_\alpha)^{m_\alpha}
\int_{W'_{A_2}(\mathbf m_1)} \hskip-1cm h(x)\mathrm e(f(x))\,\mathrm dx \\
 = \sum_{\substack{A_1\subset A \\ A_2=A\setminus A_1}}
\sum_{\mathbf m_1\in (0,\mu)^{A_1}}
\left(\int_{W'_{A_2}(\mathbf m_1)} \hskip-1cm h(x)\mathrm e(f(x))\,\mathrm dx \right)
 \prod_{\alpha\in A_1} T_\alpha^{m_\alpha}
 \prod_{\alpha\in A_2}  \frac{ (1-\Lef^{-1})(\Lef^{-1}T_\alpha)^\mu}{1-\Lef^{-1}T_\alpha}.
\end{multline*}

It follows from this computation that the power series $Q_A(\mathbf T)$ defined by
\[ Q_A(\mathbf T) =  Z_A(\mathscr X,h\mathrm e(f);\mathbf T) \prod_{\alpha\in A} (1-\Lef^{-1}T_\alpha) \]
is in fact a polynomial, 
which establishes the first assertion of the proposition.
Moreover, if we compute $Q_A(\mathbf T)$ modulo the ideal generated by the polynomials $1-\Lef^{-1}T_\alpha$ for $\alpha\in A$, only the term corresponding to
$A_1=\emptyset$ and $A_2=A$ survives in the sum. In this
case, $W'_{A_2}(\mathbf m_1)=\mathscr L(\mathscr D_A^\circ)$, and
we get
\begin{align*}
Q_A(\mathbf T) & \equiv 
\left( \int_{\mathscr L(\mathscr D_A^\circ)}
h(x)\mathrm e(f(x))\,\mathrm dx\right)
\prod_{\alpha\in A} (1-\Lef^{-1}) (\Lef^{-1}T_\alpha)^\mu \\
& \equiv 
(1-\Lef^{-1})^{\Card(A)} \left( \int_{\mathscr L(\mathscr D_A^\circ)}
h(x)\mathrm e(f(x))\,\mathrm dx\right),
\end{align*}
as claimed.
\end{proof}
\begin{coro}\label{coro.Z-rational}
Assume that $f_K$ is a regular function on the generic fiber~$\mathscr X_K$.
There exists a family $(P_A)$ of polynomials
with coefficients in~$\ExpMot_k$, 
indexed by $\Clanmax(X,D)$,
such that 
\begin{equation}
 Z(\mathscr X,h\mathrm e(f);\mathbf T) = \sum_{A\in\Clanmax(X,D)} P_A(\mathbf T) \prod_{\alpha\in A} \frac{1}{1-\Lef^{-1}T_\alpha}  \end{equation}
and such that for each~$A\in\Clanmax(X,D)$, 
\begin{equation}
 P_A(\mathbf T)- 
 (1-\Lef^{-1})^{\Card(A)}
 \int_{\mathscr L(\mathscr D_A)} h(x)\mathrm e(f(x))\,\mathrm dx
\end{equation}
belongs to the ideal generated by the polynomials $1-\Lef^{-1}T_\alpha$,
for $\alpha\in A$.
\end{coro}
\begin{proof}
By definition, one has
\[
 Z(\mathscr X, h\mathrm e(f);\mathbf T)
 = \int_{\mathscr L(\mathscr X)}  \prod_{\alpha\in\mathscr A} T_\alpha^{\ord_{\mathscr D_\alpha}(x)} h(x)\,\mathrm e(f(x))\,\mathrm dx  
 = \sum_{A\subset \mathscr A} 
Z_A(\mathscr X,h\mathrm e(f);\mathbf T).
\]
For each $A\subset\mathscr A$ such that $D_A^\circ(k)\neq\emptyset$,
choose a maximal subset $A'\in\Clanmax(X,D)$ such that $A\subset A'$.
In the previous formula for~$Z(\mathscr X,h\mathrm e(f);\mathbf T)$,
we can collect terms according to the chosen maximal subset.
Applying Proposition~\ref{prop.ZA-rational}, we obtain
\begin{align*}
 Z(\mathscr X,h\mathrm e(f);\mathbf T) 
& = \sum_{A\in\Clanmax(X,D)} 
       \sum_{\substack{B\subset A \\ B'=A}} 
             Q_B(\mathbf T) \big( \prod_{\alpha\in B} \frac{1}{1-\Lef^{-1}T_\alpha}\big)\\
&  = \sum_{A\in\Clanmax(X,D)} \prod_{\alpha\in A} \frac1{1-\Lef^{-1}T_\alpha}
       \sum_{\substack{B\subset A \\ B'=A}} Q_B(\mathbf T) \prod_{\alpha\in A\setminus B} (1-\Lef^{-1}T_\alpha).
\end{align*}
For every $A\in\Clanmax(X,D)$, we set
\[ P_A(\mathbf T)  = 
       \sum_{\substack{B\subset A \\ B'=A}} Q_B(\mathbf T) \prod_{\alpha\in A\setminus B} (1-\Lef^{-1}T_\alpha).\]
Then we have
\[ Z(\mathscr X,h\mathrm e(f);\mathbf T) 
 = \sum_{A\in\Clanmax(X,D)} P_A(\mathbf T) \prod_{\alpha\in A} \frac1{1-\Lef^{-1}T_\alpha}. \]
Moreover, modulo the ideal generated by the polynomials $1-\Lef^{-1}T_\alpha$
(for $\alpha\in A$), $P_A(\mathbf T)$ is congruent  to $Q_A(\mathbf T)$,
which is itself congruent to
\[ (1-\Lef^{-1})^{\Card(A)}
    \int_{\mathscr L(\mathscr D_A^\circ)}  h(x) \mathrm e(f(x))\,\mathrm dx. \]
This proves the corollary 
since $\mathscr D_A^\circ=\mathscr D_A$ for $A\in\Clanmax(X,D)$.
\end{proof}

\Subsection{Motivic Igusa integrals with exponentials---the general case}
\label{ss.mIie}
We keep the setup and notation as described
in Section~\ref{setup-mIie}.

In the previous section, we assumed that $f$ was regular on~$\mathscr X_K$.
In the general case where,
on~$\mathscr X_K$, 
the polar divisor $\div_\infty(f)$ of~$f$ is contained in the union 
$\bigcup_{\alpha\in \mathscr A} \mathscr D_{\alpha,K}$,
we shall prove in Proposition~\ref{coro.mIie-rational}
that the motivic Igusa  integral with exponentials 
is a rational function.
Under the additional condition that $f_K$ extends to a regular
morphism from~$\mathscr X_K$ to~$\P^1_K$, we have the following more
precise result.

\begin{prop}\label{prop.mIie-poles}
Assume that $f_K$ extends to a regular map from~$\mathscr X_K$ to~$\P^1_K$.
Let $\Clan(X,D)_0$ be the subcomplex of $\Clan(X,D)$
where we only keep vertices $\alpha\in\mathscr A$
such that $d_\alpha=0$.
Then there is a family~$(P_A)$ of polynomials   
\tcr{with coefficients in~$\ExpMot_k$,}
indexed by the set $\Clanmax(X,D)_0$ of maximal faces of $\Clan(X,D)_0$,
such that
\begin{equation}
Z(\mathscr X,h\mathrm e(f);\mathbf T) = \sum_{A\in\Clanmax(X,D)_0} P_A(\mathbf T) \prod_{\alpha\in A}
\frac1{1-\Lef^{-1}T_\alpha}.
\end{equation}
\end{prop}
\begin{proof}
Let us first assume that $f$ extends to a regular map
from~$\mathscr X$ to~$\P^1_R$.

The special case $\Clan(X,D)_0=\Clan(X,D)$ is treated by
Proposition~\ref{coro.Z-rational}.
As in the previous section, we write 
\begin{equation}
Z(\mathscr X,h\mathrm e(f);\mathbf T) = 
\sum_{A\in\Clan(X,D)} 
 Z_A(\mathscr X,h\mathrm e(f);\mathbf T) . 
\end{equation}
For every $A\in\Clan(X,D)_0$, $Z_A(\mathscr X,h\mathrm e(f);\mathbf T)$
can be evaluated by the same computation as 
the one that we performed in the proof of Proposition~\ref{prop.ZA-rational}: 
there is
a polynomial~$Q_A(\mathbf T)$ 
\tcr{with coefficients in~$\ExpMot_k$}
such that
\[
 \int_{\mathscr L(\mathscr X;D_A^\circ)}
        \prod_{\alpha\in\mathscr A} T_\alpha^{\ord_{\mathscr D_\alpha}(x)}
  h(x)   \mathrm e(f(x))\,\mathrm dx \\
 = Q_A(\mathbf T) 
 (1-\Lef^{-1})^{\Card(A)} \prod_{\alpha\in A} \frac{1}{1-\Lef^{-1}T_\alpha}
\]
and such that
$Q_A(\mathbf T)$ is congruent
to
\[ \int_{\mathscr L(\mathscr D_A;D_A^\circ)} h(x)\mathrm e(f(x))\,\mathrm dx\]
modulo the ideal generated by the polynomials
$1-\Lef^{-1}T_\alpha$.

\tcr{The general case is treated by adapting the arguments given
in the proof of Proposition~\ref{prop.ZA-rational}.
Let indeed} $A\subset\mathscr A$, let 
$A_0=\{\alpha\in A\sozat d_\alpha=0\}$
and let $A_1=A\setminus A_0$.
Define  a function~$g_A$ 
on~$\mathscr L(\mathscr X;D_A^\circ)$ by
\[ f = g_A \prod_{\alpha\in A} x_\alpha^{-d_\alpha}. \]
Since $f$ extends to a regular map to~$\P^1$, the divisors
of zeroes and of poles of~$f$ do not meet. Consequently, 
one has $d_\alpha\tcb{>}0$ for every $\alpha\in A_1$,
and $\ord\circ g_A$ is bounded \tcr{from above.}
As in the proof of Proposition~\ref{prop.ZA-rational},
$g_A$ can be expanded as a converging power series.
\tcr{Then, applying Proposition~\ref{prop.annulus-0},  
we observe that there exists an integer~$m$
such that the integral 
\[ \int_{\mathscr L(\mathscr X;D_A^\circ)} \prod_{\alpha\in\mathscr A}
T_\alpha^{\ord_{\mathscr D_\alpha(x)}}  h(x) \mathrm e(f(x))\,\mathrm dx\]
is equal to the analogous integral but restricted 
to the subset defined by the inequalities
$\ord(x_\alpha)\leq m$ for all $\alpha\in A_1$.}
By a similar argument 
to the proof of Proposition~\ref{prop.ZA-rational}, we conclude that
the power series~$Q_A(\mathbf T)$ defined by
\[
Q_A(\mathbf T)= Z_A(\mathscr X,h\mathrm e(f);\mathbf T) \prod_{\alpha\in A_0}
(1-\Lef^{-1}T_\alpha)
\]
is in fact a polynomial.

As in the proof of Corollary~\ref{coro.Z-rational},
the proposition follows by 
choosing, for every subset $A\subset\mathscr A$ some 
maximal subset $A'\in \Clan(X,D)_0$ such that $A_0\subset A'$
and regrouping the terms according to the chosen subset.

\color{blue}
This concludes the proof when $f$ extends to a regular
morphism from~$\mathscr X$ to~$\P^1_R$. To treat the general case,
recall that there exists a proper birational morphism
$\pi\colon\mathscr Y\to\mathscr X$ which is a composition
of blowing-ups with smooth centers contained in 
the special fiber such that $\pi^*f$
extends to a regular map from~$\mathscr Y$ to~$\P^1_R$.
Since $\pi$ induces an isomorphism on the generic fiber,
it does not modify the set~$\mathscr A$,
nor the analytic Clemens complexes $\Clan(X,D)$
and $\Clan(X,D)_0$.

Let $\mathscr D'_\alpha$ be the strict transform of~$\mathscr D_\alpha$.
Then $h_\alpha=\ord_{\mathscr D_\alpha}-\ord_{\mathscr D'_\alpha}$
is a constructible function on~$\mathscr L(\mathscr Y)$ which
takes only finitely many values.
By the change of variable formulas, one has
\begin{align*}
 Z(\mathscr X,h\mathrm e(f),\mathbf T)
& = \int_{\mathscr X} \prod_{\alpha\in\mathscr A} T_\alpha^{\ord_{\mathscr D_\alpha}(x)}  h(x) \mathrm e(f(x))\,\mathrm dx \\
& = \int_{\mathscr Y} \prod_{\alpha\in\mathscr A} T_\alpha^{\ord_{\mathscr D_\alpha}(\pi(y))}  \pi^*h(y) \mathrm e(\pi^*f(y))\Lef^{\ord J_\pi(y)}\,\mathrm dy \\
& = \int_{\mathscr Y}
\prod_{\alpha\in\mathscr A} T_\alpha^{\ord_{\mathscr D'_\alpha}(y)}  
\prod_{\alpha\in\mathscr A} T_\alpha^{h_\alpha(x)} \pi^*h(y) \mathrm e(\pi^*f(y))\Lef^{\ord J_\pi(y)}\,\mathrm dy.
\end{align*}
We then decompose this integral according to the values of~$h_\alpha$
and~$\ord J_\pi$ and compute  each individual part as in the first
part of the proof.
Combining the the various contributions implies the proposition. 
\end{proof}

\begin{prop}\label{prop.mIie2}\label{coro.mIie-rational}
The power series~$Z(\mathscr X,h\mathrm e(f);\mathbf T)$ is a rational function.
\end{prop}
\begin{proof}
Let $\pi\colon\mathscr Y\ra\mathscr X$ be a proper birational morphism
such that the rational map $\pi^*f$  extends to a
regular morphism from~$\mathscr Y_K$ to~$\P^1_K$
and such that the horizontal part of~$\pi^*\mathscr D$ is  a 
relative divisor with strict normal crossings.
For every $\alpha\in\mathscr A$,
we write $\mathscr D'_{\alpha}$ 
for the strict transform of~$\mathscr D_{\alpha}$;
let $(\mathscr E_{\beta})_{\beta\in\mathscr B}$ be the family
of horizontal exceptional divisors.

Let $\alpha\in\mathscr A$. 
There exists a family $(m_{\alpha,\beta})_{\beta\in\mathscr B}$
of nonnegative integers such that 
\[ \pi^*\mathscr D_{\alpha,K}=
\mathscr D'_{\alpha,K}+\sum_{\beta\in\mathscr B}
    m_{\alpha,\beta}  \mathscr E_{\beta,K}. \]
Consequently, there exists a bounded constructible function~$u_\alpha$
on~$\mathscr L(\mathscr Y)$ such that
\begin{equation}
  \pi^*\ord_{\mathscr D_\alpha} = \ord_{\mathscr D'_\alpha}
 + \sum_{\beta\in\mathscr B} m_{\alpha,\beta} \ord_{\mathscr E_\beta} + u_\alpha. 
\end{equation}

Let also $(\nu_{\beta})_{\beta\in\mathscr B}$ be the family
of positive integers such that
\[ K_{\mathscr Y_K/\mathscr X_K}= \sum_{\beta\in\mathscr B} (\nu_\beta-1) \mathscr E_{\beta,K}. \]
This implies that there exists a bounded constructible function~$v$ 
on~$\mathscr L(\mathscr Y)$ such that
\[ \ord_{K_{\mathscr Y/\mathscr X}}
     = \sum_{\beta\in\mathscr B} (\nu_\beta-1) \ord_{\mathscr E_\beta}+v. \]

Then, using the change of variables formula 
(Theorem 13.2.2 in \cite{cluckers-loeser:2008}), we find
\begin{align*}
Z(\mathscr X,h\mathrm e(f);\mathbf T) 
&= \int_{\mathscr L(\mathscr X)} \prod_{\alpha\in\mathscr A}
 T_\alpha^{\ord_{\mathscr D_\alpha}(x)} h(x) \mathrm e(f(x))\,\mathrm dx \\
&= \int_{\mathscr L(\mathscr Y)} \prod_{\alpha\in\mathscr A}
 T_\alpha^{\ord_{\mathscr D'_\alpha}(y)+\sum_\beta m_{\alpha,\beta}\ord_{\mathscr E_\beta}(y)+u_\alpha(y)}  \times {} \\
& \hskip.2\textwidth {} \times 
     \Lef^{v(y)+\sum_\beta (\nu_\beta-1)\ord_{\mathscr E_\beta}(y)} 
  \pi^*h(y) \mathrm e(\pi^*f(y))\,\mathrm dy \\
& = \int_{\mathscr L(\mathscr Y)} 
\prod_{\alpha\in\mathscr A} T_\alpha^{\ord_{\mathscr D'_\alpha}(y)+u_\alpha(y)}
\prod_{\beta\in\mathscr B} P_\beta(\mathbf T)^{\ord_{\mathscr E_\beta}(y)}
\Lef^{v(y)} 
\pi^*h(y) \mathrm e(\pi^*f(y))\,\mathrm dy ,
\end{align*}
where, for every $\beta\in\mathscr B$, we have set
\begin{equation}\label{eq.SfromT}
 P_\beta (\mathbf T)= \Lef^{\nu_\beta-1} \prod_{\alpha\in\mathscr A} T_\alpha^{m_{\alpha,\beta}}.
\end{equation}

Let us introduce a family $\mathbf S=(S_\beta)_{\beta\in\mathscr B}$
of indeterminates. For every
\tcr{motivic residual} function~$w$ on~$\mathscr L(\mathscr Y)$
and every rational function~$g$ on~$\mathscr Y_K$,
let us also define, 
analogously to Equation~\eqref{eq.motivic-Igusa},
a generating series
\[ Z(\mathscr Y, w\mathrm e(g);(\mathbf T,\mathbf S))
 = \int_{\mathscr L(\mathscr Y)} \prod_{\alpha\in\mathscr A} T_\alpha^{\ord_{\mathscr D'_\alpha}(y)} \prod_{\beta\in\mathscr B} S_\beta^{\ord_{\mathscr E_\beta}(y)}  w(y)\mathrm e(g(y))\,\mathrm dy. \]
\tcr{If $g$} induces a morphism from~$\mathscr Y_K$ to~$\P^1_K$,
it follows from Proposition~\ref{prop.mIie-poles} \tcb{and the proof of Proposition~\ref{prop.mIie2}}
\tcr{that $Z(\mathscr Y,w \mathrm e(g);(\mathbf T,\mathbf S))$}
is a rational function of $(\mathbf T,\mathbf S)$,
for every~$w$.

For every $p\in \Z^{\mathscr A}$, let $W_p$
be the constructible subset of~$\mathscr L(\mathscr Y)$
on which the family $(u_\alpha)$ equals~$p$.
These subsets form a finite partition of~$\mathscr L(\mathscr Y)$
and one has
\[
Z(\mathscr X,h\mathrm e(f);\mathbf T) = \sum_p \prod_{\alpha\in\mathscr A} T_\alpha^{p_\alpha}
  Z(\mathscr  Y,\mathbf 1_{W_p}\Lef^v \pi^* h\mathrm e(\pi^*f);(\tcb{\mathbf T},(P_\beta(\mathbf T))), \]
where the polynomials~$P_\beta$ are defined in~\eqref{eq.SfromT}.
Consequently, $Z(\mathscr X,h\mathrm e(f);\mathbf T)$ is a rational
function of~$\mathbf T$.
\end{proof}

\section{Local Fourier transforms}\label{sec.local}

In this section, we finally prove the propositions
stated in Section~\ref{ss.local-results}.

Fix a place  $v\in C (k)$.
We have to study 
the Laurent series~\eqref{eq.ZvT} given by
\begin{align*}
Z_v(\mathbf T,\xi)
& = \sum_{\text{$(\mathbf m_v,\beta_v)$ integral}}
\mathscr F(\mathbf 1_{G(\mathbf m_v,\beta_v)},\xi)
\prod_{\alpha\in \mathscr A}T_\alpha^{\norm {\mathbf m_v,\beta_v}_\alpha}\\
&= \sum_{\text{$(\mathbf m_v,\beta_v)$ integral}} 
 \prod_{\alpha\in \mathscr A}T_\alpha^{\norm {\mathbf m_v,\beta_v}_\alpha}
\int_{G(\mathbf m_v,\beta_v)} \mathrm e(\langle g,\xi\rangle)\,\mathrm dg.
\end{align*}

Our first step will be to replace 
the motivic Haar integral with respect to~$\mathrm dg$
by a motivic integral with respect to the motivic measure on
the arc space~$\mathscr L(\mathscr X)$.
Then, we will split the motivic integral according to the natural
stratification of the special fibre.

Since the place~$v$ is fixed,
we often omit the index~$v$ from the notation, writing
$K\simeq k\lpar t\rpar$ for $F_v=k(C)_v$
and $R\simeq k\lbra t\rbra$ for the ring of integers of~$F_v$.

\Subsection{The Haar integral as a motivic measure}
View the invariant top-differential form~$\mathrm dg$ on~$G_F$
as a meromorphic top-differential form~$\omega_X$ on~$\mathscr X$.
Since $\mathscr X$ is proper over~$C $,
one has  $\mathscr X(k\lbra t\rbra)=X(K)$.
\tcr{The order of contact
of an arc with the divisor of~$\omega$}
induces  an order function 
$\ord_\omega\colon X(K)\ra\Z\cup\{\infty\}$,
which  takes finite values on $G(K)$.

\tcr{The Poisson summation formula involves (motivic) integrals
on $k\lpar t\rpar^n=G(F)$.
The injection $G(F)\subset \mathscr X(k\lbra t\rbra)$
allows to view any Schwartz-Bruhat function~$\Phi$ on~$G(F)$
as an exponential motivic function on the arc space~$\mathscr L(\mathscr X)$
of~$\mathscr X$. 
The following lemma shows how both motivic integrals are related.
It is a motivic analogue of the standard fact that the Lebesgue mesure
on the real line~$\R$ is the volume-form associated
with the differential form $\mathrm dx$, or with the corresponding singular
differential form on~$\P^1(\R)$.}
\begin{lemm}\label{lemm.haar-motivic}
\tcr{Let $\Phi\in\mathscr S(F^n)$.}
Then the motivic integral
$ \int_{G(K)} \Phi(g) \,\mathrm dg $
can be rewritten as
\[ \int_{\mathscr L(\mathscr X)} \Phi(x) \Lef^{-\ord_\omega(x)}\,\mathrm dx \]
where $\mathrm dx$ denotes the motivic measure on 
the arc space~$\mathscr L(\mathscr X)$.
\end{lemm}
\begin{proof}
Let us begin with a remark.
Let $Z$ be a smooth projective $K$-scheme and let $\omega$
be a meromorphic differential form on~$Z$. 
Let $f$ be a motivic function on~$Z$.  \tcr{By this, we mean that
one is given a proper flat $R$-model~$\mathscr Z$ of~$Z$,
an integer~$m$, and a class~$\phi\in\ExpMot_{\mathscr L_m(\mathscr Z)}$.
We identify the function associated with a triple $(\mathscr Z,m,\phi)$
and the function associated with the triple $(\mathscr Z,m+1,\pi^*\phi)$,
where $\pi$ is the canonical morphism from~$\ExpMot_{\mathscr L_m(\mathscr Z)}$
to~$\ExpMot_{\mathscr L_{m+1}(\mathscr Z)}$ defined by base change;
similarly, if $p\colon\mathscr Z'\to\mathscr Z$ is a morphism
of models, we identify the functions associated with triples
$(\mathscr Z,m,\phi)$ and $(\mathscr Z',m,p^*\phi)$.}
Then one defines 
the integral $\int_{Z(K)} f \abs\omega$ of the motivic function~$f$ with respect
to~$\omega$ by the formula
\[  \int_{Z(K)} f \abs\omega = \int_{\mathscr L(\mathscr Z)} \phi(z) \Lef^{-\ord_\omega(z)}\,\mathrm dz, \]
where $\mathrm dz$ is the motivic measure 
on the arc space~$\mathscr L(\mathscr Z)$.
\tcr{We may even assume that the smooth locus~$\mathscr Z^0$
of~$\mathscr Z$ is a weak Néron model of~$Z$ and restrict
the integral over~$\mathscr L(\mathscr Z^0)$.}
The \tcr{definition of motivic integration and the} change of variables formula 
(Theorem 13.2.2 in \cite{cluckers-loeser:2008})
shows that this integral is independent of the choice
\tcr{of the triple~$(\mathscr Z,m,\phi)$ that defines~$f$.}

Let us show how this remark implies our lemma.

Let $\mathscr P=\P^n=\Proj R[x_0,\dots,x_n]$ 
be the natural compactification of~$G_R=\Aff^n =\Spec (R[x_1,\dots,x_n])$,
let $P=\mathscr P_K$.
Let $\omega$ be the differential form
$\mathrm dx_1\wedge\dots\wedge\mathrm dx_n$ on~$G_R$;
we will write $\omega_P$ or~$\omega_X$
according to~$\omega$ being viewed as a meromorphic differential form 
on~\tcr{$\mathscr P$} or on~$\mathscr X$.
Since $\mathscr X$ is regular, one has 
\[
  \int_{X(K)} \Phi(x) \abs{\omega_X}
= \int_{\mathscr L(\mathscr X)} \Phi(x) \Lef^{-\ord_{\omega_X}(x)}\,\mathrm dx .
 \]
However, this integral can be computed starting from
the model~$\mathscr P$, and one has
\[ \int_{X(K)} \Phi(x) \abs{\omega_{X} }
=\int_{P(K)} \Phi(p) \abs{\omega_P}. \]
Since $\div(\omega_P)=(n+1)\div(x_0)$, 
$\ord_{\omega_P}(x)=(n+1)\min (0,\ord(x_1),\dots,\ord(x_n))$
for every $(x_1,\dots,x_n)\in K^n$. 
Consequently, by the very definition of the motivic integral on~$G(K)$,
this last integral equals $\int_{G(K)} \Phi(g)\,\mathrm dg$.
It suffices to show this equality for 
\tcr{simple functions, as in~\S\ref{par.simple-functions}.}
Let $(m_1,\dots,m_n)$ be a family of integers and
let $\Omega$ be the set of points in~$K^n$
such that $\ord(x_i)=m_i$ for every~$i\in\{1,\dots,n\}$. 
Set $m_0=0$, let $m=\min(m_0,\dots,m_n)$, and let $j\in\{0,\dots,n\}$
be such that $m=m_j$.
We view~$\Omega$ in the affine chart $\{x_j=1\}$ of~$\mathscr P$
and identify it with $\prod_{\substack{0\leq i\leq n\\ i\neq j}} 
      t^{m_i-m}R^\times$.
Therefore, its $\mathrm dp$-measure (that is, its motivic volume
with respect to the arc space of~$\mathscr P$) equals
\[ \int_{\Omega}\mathrm dp = \Lef^{\sum_{i\neq j} (m_i-m)} (1-\Lef^{-1})^n
 = (1-\Lef^{-1})^n  \Lef^{\sum m_i } \Lef^{-(n+1)m} . \]
On the other hand, $\Omega$ is also viewed as the set 
$\prod_{i=1}^n t^{-m_i}R^\times$ 
in~$K^n$ hence its $\mathrm dg$-integral is given by
\[ \int_{\Omega} \mathrm dg
  = \Lef^{\sum m_i} (1-\Lef^{-1})^n. \]
Consequently, 
\[ \int_\Omega\mathrm dg = \int_\Omega \Lef^{-\ord_\omega(p)}\,\mathrm dp.\]

This concludes the proof of the lemma.
\end{proof}

\subsection{Partitions of unity}
Let $\mathscr B_1$ be the subset of~$\mathscr B_v$
consisting of those $\beta$ for which the multiplicity $\mu_\beta$
equals~$1$.
Let~$\mathscr X_1$  be the complement in~$\mathscr X$ 
of the union of the components~$E_\beta$,
for $\beta\in \mathscr B_v\setminus\mathscr B_1$, and of 
the intersections of distinct vertical components.
As we know from Lemma~\ref{lemm.weakNeron}, this is a weak
N\'eron model of~$X$ over~$\Spec(R)$.

For every subset $A\subset\mathscr A$ and every~$\beta\in \mathscr B_{1}$,
let $\Delta(A,\beta)$
be the locally closed subset of the special
fibre~$\mathscr X_v$ 
corresponding to points~$\tilde x$ which 
belong to the horizontal divisors $\mathscr D_\alpha$, for $\alpha\in A$,
and to no other, as well as to the vertical divisor~$E_\beta$, but no other.
Let $\Omega(A,\beta)$ be its preimage in  $\mathscr L(\mathscr X)$
by the specialization morphism $\mathscr L(\mathscr X)\ra \mathscr X_v$.

Recall that we have defined in Section~\ref{par.B0}
a subset $\mathscr B_0$ of~$\mathscr B$ such that
for every point $x\in G(\tcb{F_v})$, the integrality condition 
``$x\in\mathscr U(\mathfrak o_{\tcb{F_v}})$ \tcr{if $v\in C_0$}'' 
holds if and only if there exists
$\beta\in\mathscr B_0$ such that
$x\in\Omega(\emptyset,\beta)$.

Let $\mathscr L(\Aff^1)\simeq \Spec (k[a_0,a_1,\dots])$
be the arc space of the affine line, and
let $\mathscr L(\Aff^1;0)$
be its closed subspace consisting of arcs based at the origin.
\tcr{Let $A$ be a subset of~$\mathscr A$.}
By Lemma~\ref{lemm.hensel}, there is a definable isomorphism 
which preserves the motivic measure
from~$\Omega(A,\beta)$ to the product of
$ \mathscr L(\Aff^1;0)^A \times_k\Delta(A,\beta)$
with $\mathscr L(\Aff^1)^{n-\Card(A)}$. 
Moreover, this isomorphism induces the following equalities\tcr{, for $g\in G(F)$}:
\begin{equation}
(g,\mathscr D_\alpha)_v  = \begin{cases}
 \ord_v(x_\alpha) &\text{if $\alpha\in A$} \\
 0 & \text{otherwise,} \end{cases} \end{equation}
\tcr{(where, as in Lemma~\ref{lemm.hensel}, $x_\alpha$ is the $\alpha$-component of the image of~$g$)}
and
\begin{equation}
(g,E_\gamma)_v = \begin{cases} 1 & \text{if $\gamma=\beta$} \\
  0 & \text{otherwise.} \end{cases}
\end{equation}

Recall that in Equation~\eqref{eq.rho-beta},
we had defined integers~$\rho_\beta$
such that
\[-\div(\omega_X)=\sum_{\alpha\in\mathscr A}\rho_\alpha
\mathscr D_\alpha+\sum_{\beta\in\mathscr B} \rho_\beta E_\beta. \]
This implies
\begin{equation}
-\ord_\omega(x) = \rho_\beta+ \sum_{\alpha\in A} \rho_\alpha \ord_v(x_\alpha).
\end{equation}
Recall also the definition of integers~$e_{\alpha,\beta}$
in Equation~\ref{eq.e-alpha-beta}:
\[ \mathscr L_\alpha =
     \mathscr D_\alpha + \sum_{\beta\in\mathscr B} e_{\alpha,\beta}E_\beta .
\]
To shorten the notation below, we then let 
\begin{equation}
\rho(A,\beta) = \rho_\beta + \sum_{\alpha\in\mathscr A} \rho_\alpha e_{\alpha,\beta}.
\end{equation}
We also define $e_\alpha$ to be the constructible function 
\begin{equation}
e_\alpha(\cdot) = \sum_{\beta\in\mathscr B} e_{\alpha,\beta} \ord_{E_\beta}(\cdot) 
\end{equation}
on $\mathscr L(\mathscr X)$ so that $e_\alpha(g)=e_{\alpha,\beta}$ 
for every $g\in G(\tcb{F_v})$ such that $(g,E_\beta)_v=1$.

Using this notation and applying Lemma~\ref{lemm.haar-motivic},
we can rewrite
the motivic Fourier transforms $Z_v(\mathbf T,\xi)$ 
as sums of motivic integrals over arc spaces $\Omega(A,\beta)$.

\begin{lemm}\label{lemm.pu}
For every motivic residual function~$h$ on~$\mathscr L(\mathscr X)$
and every $\xi\in G(\tcb{F_v})$, one has
\begin{multline}\label{eq.pu}
\int_{G(\tcb{F_v})}\prod_{\alpha\in\mathscr A} T_\alpha^{(g,\mathscr L_\alpha)_v} h(g)\mathrm e(\langle g,\xi\rangle) \,\mathrm dg \\
=
\sum_{\substack{A\subset\mathscr A\\ \beta\in\mathscr B_1}}
\prod_{\alpha\in\mathscr A} T_\alpha^{e_{\alpha,\beta}}
\Lef^{\rho_\beta}
\int_{\Omega(A,\beta)}  
\prod_{\alpha \in A} (\Lef^{\rho_\alpha}T_\alpha)^{\ord(x_\alpha)}
h(x)\mathrm e(\langle x,\xi\rangle)\,\mathrm dx.
\end{multline}
\end{lemm}

\Subsection{Places in~$C_0$: Proof of Proposition~\ref{prop.almost}}

Let $v$ be a place in~$C_0$.
In this case, $Z(\mathbf T,\xi)$
is given by Lemma~\ref{lemm.pu}, taking for motivic function~$h$
the characteristic function of the set~$\mathscr U(\mathfrak o)$
within~$G(\tcb{F_v})$.
In other words, one has $h\equiv 0$ on $\Omega(A,\beta)$ if $A\neq\emptyset$
or $\beta\not\in\mathscr B_0$, and $h\equiv 1$ otherwise.
Consequently,
\[
Z(\mathbf T,\xi)= 
\sum_{\beta\in\mathscr B_0}
  \prod_{\alpha\in\mathscr A}  T_\alpha^{e_{\alpha,\beta}}
\Lef^{\rho_\beta} \int_{\Omega(\emptyset,\beta)} \mathrm e(\langle x,\xi\rangle)\,\mathrm dx.
\]
We see in particular  that it is a polynomial.
Assume that $\xi=0$. Then, the factor $\mathrm e(\langle x,\xi\rangle)$
equals~$1$, so that
\[
Z(\mathbf T,0)= 
\sum_{\beta\in\mathscr B_0}
  \prod_{\alpha\in\mathscr A}  T_\alpha^{e_{\alpha,\beta}}
\Lef^{\rho_\beta} [\Delta(\emptyset,\beta)] \Lef^{-n} .
\]
In particular, 
\[Z_\lambda(\Lef^{-1},0)
= Z((\Lef^{1-\rho_\alpha}),0)
= \sum_{\beta\in\mathscr B_0}
 \prod_{\alpha\in\mathscr A} \Lef^{(1-\rho_\alpha)e_{\alpha,\beta}}
\Lef^{\rho_\beta} [\Delta(\emptyset,\beta)]\Lef^{-n}.
\]
This is an effective element of ${\Mot}_k$,
non-zero unless $\mathscr U(\mathfrak o)=\emptyset$.
\tcr{By the last assertion of Lemma~\ref{lemm.replaceU}
and the hypothesis of Setting~2,}
this concludes the proof of Proposition~\ref{prop.almost}.


\Subsection{Trivial character (places in~$C \setminus C_0$):
Proof of Proposition~\ref{prop.v.trivial}}
Assume that $v\in C \setminus C_0$.
Then, $Z(\mathbf T,0)$ is given by Lemma~\ref{lemm.pu}, applied with
$h\equiv 1$. This leads to the following computation.
\begin{align*}
Z(\mathbf T,0) 
&=  \sum_{\substack{A\subset\mathscr A \\ \beta \in\mathscr B_1}} 
\prod_{\alpha\in \mathscr A} T_{\alpha}^{e_{\alpha,\beta}}
\Lef^{\rho_\beta}
\int_{\Omega(A,\beta)} 
\prod_{\alpha\in A} (\Lef^{\rho_\alpha} T_\alpha)^{\ord(x_\alpha)}
\, \mathrm dx \\
& = 
\sum_{\substack{A\subset\mathscr A \\ \beta \in\mathscr B_1}} 
\prod_{\alpha\in \mathscr A} T_{\alpha}^{e_{\alpha,\beta}}
\Lef^{\rho_\beta}
[\Delta(A,\beta)]\Lef^{-n+\Card(A)}
\prod_{\alpha\in A} \int_{\mathscr L(\Aff^1;0)} (\Lef^{\rho_\alpha}T_\alpha)^{\ord(x)}\,\mathrm dx.
\end{align*}
In the last formula, the integral over~$\mathscr L(\Aff^1;0)$ is
given by a geometric series, familiar in the theory
of motivic Igusa functions. Indeed, for every $\alpha\in A$,
\begin{align*}
 \int_{\mathscr L(\Aff^1;0)} (\Lef^{\rho_\alpha}T_\alpha)^{\ord(x)}\,\mathrm dx& =
\sum_{m=1}^\infty   (\Lef^{\rho_\alpha}T_\alpha)^m\int_{\ord(x)=m}\,\mathrm dx\\
&= 
\sum_{m=1}^\infty   (\Lef^{\rho_\alpha}T_\alpha)^m\Lef^{-m} (1-\Lef^{-1}) \\
& =
(1-\Lef^{-1})  \frac{\Lef^{\rho_\alpha-1}T_\alpha}{1-\Lef^{\rho_\alpha-1}T_\alpha}.
\end{align*}
Consequently,
\begin{equation}\label{eq.ZT0-1}
Z(\mathbf T,0) 
 = 
\sum_{\substack{A\subset\mathscr A \\ \beta \in\mathscr B_1}} 
\prod_{\alpha\in \mathscr A} T_{\alpha}^{e_{\alpha,\beta}}
\Lef^{\rho_\beta}
[\Delta(A,\beta)]\Lef^{-n+\Card(A)}(1-\Lef^{-1})^{\Card(A)}
\prod_{\alpha\in A} 
\frac{\Lef^{\rho_\alpha-1}T_\alpha}{1-\Lef^{\rho_\alpha-1}T_\alpha}
.
\end{equation}
 
\tcr{For every pair~$(A,\beta)$ 
such that $\Delta(A,\beta)\neq\emptyset$,
fix a maximal subset~$A_0$ of~$\mathscr A$ such that $A\subset A_0$
and $\Delta(A_0,\beta)\neq\emptyset$.}
\tcr{Let us fix such a maximal set~$A_0$ and let us then collect the  terms
of Equation~\eqref{eq.ZT0-1}
corresponding to pairs~$(A,\beta)$ that are associated with~$A_0$.}
The corresponding subseries of~$Z(\mathbf T,0)$
is then given by 
\[
Z_{A_0}(\mathbf T,0)
 = 
\sum_{\substack{\beta\in\mathscr B_1 \\ (A,\beta)\mapsto A_0} }
\prod_{\alpha\in \mathscr A} T_{\alpha}^{e_{\alpha,\beta}}
\Lef^{\rho_\beta}
[\Delta(A,\beta)]\Lef^{-n+\Card(A)}(1-\Lef^{-1})^{\Card(A)}
\prod_{\alpha\in A} 
\frac{\Lef^{\rho_\alpha-1}T_\alpha}{1-\Lef^{\rho_\alpha-1}T_\alpha}
. \]
Consequently, there exists a Laurent polynomial~$P_{A_0}(\mathbf T)$ 
such that 
\[ Z_{A_0}(\mathbf T,0) \prod_{\alpha\in A_0} (1-\Lef^{\rho_\alpha-1}T_\alpha)
=P_{A_0}(\mathbf T).\]
From this expression, we also see that 
modulo the ideal generated by the polynomials $1-\Lef^{\rho_\alpha-1}T_\alpha$,
for $\alpha\in \mathscr A$, only the terms corresponding to $A=A_0$
remain, so that we have
\[ P_{A_0}(\mathbf T) \equiv 
       \sum_{\beta\in\mathscr B_1} 
\prod_{\alpha\in\mathscr A} \Lef^{(1-\rho_\alpha)e_{\alpha,\beta}}
\Lef^{\rho_\beta}
[\Delta(A_0,\beta)] \Lef^{-n+\Card(A_0)} (1-\Lef^{-1})^{\Card(A_0)}
.
\]
Let $u_{A_0}$ be the motivic residual function on~$\mathscr L(\mathscr X)$
which is given by
\[\rho_\beta + \sum_{\alpha\in\mathscr A} (1-\rho_\alpha)e_{\alpha,\beta}
\]
on $\Omega(A_0,\beta)$. 
By definition of motivic integration, one has
\begin{align*}
\int_{\mathscr L(\mathscr D_{A_0})} \Lef^{u_{A_0}}
&=  \sum_{\beta\in\mathscr B_1} \int_{\Omega(A_0,\beta)}
  \Lef^{\rho_\beta} \prod_{\alpha\in\mathscr A} \Lef^{(1-\rho_\alpha)e_{\alpha,\beta}}
 \\
&=   \sum_{\beta\in\mathscr B_1} 
  \Lef^{\rho_\beta} \prod_{\alpha\in\mathscr A} \Lef^{(1-\rho_\alpha)e_{\alpha,\beta}}
\Lef^{-n+\Card(A_0)} [\Delta(A,\beta)] 
\end{align*}
so that
\[  P_{A_0}(\mathbf T) \equiv 
(1-\Lef^{-1})^{\Card(A_0)}
\int_{\mathscr L(\mathscr D_{A_0})} \mathbf L^{u_{A_0}} . \]
The right-hand-side of the preceding congruence being
a non-zero effective element of ${\Mot}_k$,
this concludes the proof of Proposition~\ref{prop.v.trivial}.

\begin{coro}
Let $d=1+\dim\Clan(X,D)$ and
\tcr{let $a$ be a non-zero multiple of the integers $\rho_\alpha-1$,
for $\alpha\in\mathscr A$.}
The Laurent series
$Z_{\lambda}(T,0)$ in one variable~$T$ is a rational function
which belongs to~$\Mot_k\{T\}$. Moreover,
$(1-\Lef^aT^a)^d Z_\lambda(T,0)$ belongs
to $\Mot_k[T,T^{-1}]$ and
\[ (1-\Lef^a T^a)^d Z_{\lambda}(T,0)\Big|_{T=\Lef^{-1}} 
= (1-\Lef^{-1})^{d} 
\sum_{A\in\Clan[d](X,D) } 
\prod_{\alpha\in A} \frac{a}{\rho_\alpha-1}
\int_{\mathscr L(\mathscr D_A)} \Lef^{u_A(x)}\,\mathrm dx. \]
\end{coro}

\Subsection{Non-trivial characters (places in $C \setminus C_0$):
Proof of Proposition~\ref{prop.v.nontrivial}}

Assume that $v$ belongs to~$C \setminus C_0$.
In this Section, we establish the behavior
of the Fourier transform $Z_v(\mathbf T,\mathbf a(\xi))$
for non-zero $\xi\in E$. 
Since the place~$v$ is fixed, the motivic Igusa integrals
$Z_v(\mathbf T,\cdot) $ and $Z_{\lambda,v}(T,\cdot)$ are denoted~$Z(\mathbf T,\cdot)$ and~$Z_\lambda(T,\cdot)$ respectively.

By Proposition~\ref{prop.mIie2}, we already know
that for each~$\xi$, $Z(\mathbf T,\mathbf (\xi))$ is rational,
with denominator given by 
products of polynomials of the form $1-\Lef^n \mathbf T^{\mathbf m}$
for some $n\in\N$ and $\mathbf m\in\N^{\mathscr A}$
which are described through some Clemens complex.
To prove Proposition~\ref{prop.v.nontrivial}, we have
to prove two more properties: first, that this rationality  property
holds uniformly on the strata of 
some \tcr{constructible} partition $(U_{i})$, and second, that
\tcr{for each~$i$, 
there exists an integer $e\in[0,d-1]$ and an integer $a\geq 1$
such that the restriction to~$U_{i}$
of~$(1-\Lef^a T^a)^e Z_\lambda(T,\cdot)=
(1-\Lef^aT^a)^e Z((T_\alpha^{\lambda_\alpha}),\cdot)$
belongs to $\ExpMot_{U_{i}}\{T\}^\dagger$.}

The following analysis thus refines the proof of Proposition~\ref{prop.mIie2}.
It is very close to the one of~\cite{chambert-loir-tschinkel:2012},
except for the replacement of $p$-adic oscillatory integrals
by the motivic integrals of Section~\ref{ss.mIie}.


Recall that any point $\xi\in G(\tcb{F_v})$ gives rise to a linear
form $f_\xi=\langle \xi,\cdot\rangle$ on~$G_{\tcb{F_v}}$, hence to
a rational function on~$X$, or even on~$\mathscr X$.
More generally, the morphism~$\mathbf a\colon E_{\tcb{F_v}}\to G$
induces a regular function~$f_E$ on $G\times_{\tcb{F_v}} E_{\tcb{F_v}}$, hence
a rational function on~$\mathscr X\times_k E$.
We view $f_E$ as a family
of regular functions on~$G$, resp. as a family
of rational functions on~$\mathscr X$, both indexed by~$E$
and study the variation, for $p\in E$,
of the divisors $\div(f_{\mathbf a(p)})$. 

\begin{lemm}
Let $P$ be a reduced $k$-scheme of finite type and let
$\mathbf a\colon P_{\tcb{F_v}} \ra G$  be an $\tcb{F_v}$-morphism.
Let $f_P$ be the associated rational function on $\mathscr X\times_k P$.
There exists a decomposition of~$P$
as a disjoint union of smooth locally closed subsets~$P_i$,
and for each~$i$, a map 
$\pi_i\colon\mathscr Y_i\ra \mathscr X\times_k P_i$ which is
a composition of blowings-up whose centers
are smooth over~$P_i$, with generic fiber invariant under the
action of~$G_{\tcb{F_v}}\times_k P_i$, and do not meet~$G_{\tcb{F_v}}\times_k P_i$
such that 
the rational function~$\pi_i^*f_P$ on~$\mathscr Y_i$
defines a regular morphism from~$\mathscr Y_i\times_k P_i$ to~\tcr{$\P^1_{R}$}
\tcr{and such that the horizontal part of $\pi_i^*(\mathscr D\times_k P_i)$
is a relative divisor with strict normal crossings}.
\end{lemm}
\begin{proof}
The arguments of Lemma~3.4.1 of~\cite{chambert-loir-tschinkel:2012}
furnish a partition~$(P_i)$ of~$P$ by smooth locally closed
subsets and morphisms $\pi_i\colon Y_{i,\tcb{F_v}}\to X\times_{\tcb{F_v}} P_{i,\tcb{F_v}}$ which are compositions of blowing-ups whose centers are smooth
over~$P_{i,\tcb{F_v}}$ and do not meet $G_{\tcb{F_v}}\times_{\tcb{F_v}} P_{i,\tcb{F_v}}$
such that the rational function~$\pi_i^*f_P$ defines
a regular morphism $Y_{i,\tcb{F_v}}\times_{\tcb{F_v}} P_{i,\tcb{F_v}}\to\P^1_{\tcb{F_v}}$.

Let us fix such a stratum~$P_i$ and call it~$P$; similarly,
we write $\pi$ for~$\pi_i$ and $Y$ for~$Y_{i}$.
Up to replacing~$P$ by a dense open subset~$P'$, and
applying embedded resolution of singularities to the Zariski
closures of the centers of the blowing-ups which define~$\pi$,
we obtain a morphism $\pi'\colon\mathscr Y\to\mathscr X\times_k P'$,
composition of blowing-ups with smooth centers over~$P'$,
such that $\pi'_{\tcb{F_v}}\colon\mathscr Y_{\tcb{F_v}}\to\mathscr X_{\tcb{F_v}}\times_{\tcb{F_v}} P'_{\tcb{F_v}}$
satisfies the preceding assumptions.
We then view the morphism $f_P$
as a rational map $\mathscr Y\times_k P'\dashrightarrow \mathbf P^1$.
Above a dense open subset~$P''$ of~$P'$
we resolve its indeterminacies  by a further composition of blowing-ups
whose centers  are smooth over~$P'$ and do not meet the generic fiber.
We can now repeat these arguments for~$P\setminus P''$, 
so that the lemma follows by noetherian induction.
\end{proof}

We apply the preceding lemma to $E\setminus\{0\}$
and fix such a stratum, which we call~$P$.
Set $Y=\mathscr Y_{\tcb{F_v}}$. For $p\in P$, let $\mathscr Y_p$ be the fiber
of~$\mathscr Y$ above~$p$ under the composition $\mathscr Y\to \mathscr X\times_k P\to P$; let $Y_p=(\mathscr Y_p)_{\tcb{F_v}}$.

For $p\in P$, by the change of variable formula 
(Theorem 13.2.2 in \cite{cluckers-loeser:2008}),
both $Z(\mathbf T,0)$
and $Z(\mathbf T,\tcb{\mathbf a}(p))$ can be computed as motivic integrals
on~$\mathscr L(\mathscr Y_p)$. Since the rational function~$f_{\tcb{\mathbf a}}$
extends to a regular morphism from~$\mathscr Y_{\tcb{F_v}}\times P_{\tcb{F_v}}$ to~$\P^1_{\tcb{F_v}}$,
Proposition~\ref{prop.mIie-poles} 
gives an explicit form for $Z(\mathbf T,\tcb{\mathbf a}(p))$ as
a rational function, whose denominator
is controlled in terms of the Jacobian divisor of~$Y/X$
and the sub-complex $\Clan(Y_p,Y_p\setminus G)_0$
of $\Clan(Y_p,Y_p\setminus G)$
corresponding to the irreducible components of~$Y\setminus G$
along which $f_{\tcb{\mathbf a}(p)}$ has no pole. 
Since the sub-complex $\Clan(Y_p,Y_p\setminus G)_0$ 
of $\Clan(Y_p,Y_p\setminus G)$
depends constructibly on~$p\in P$,
this implies \tcr{the} first part of Proposition~\ref{prop.v.nontrivial}.

\color{blue}
The final part of the proof follows by applying the geometric
arguments leading to the proof of 
Proposition~3.4.4 of~\cite{chambert-loir-tschinkel:2012}.
For every $p\in P$, 
the Clemens complex  of $(Y_p,Y_p\setminus G)$
has distinguished vertices, those coming from $X$ which
we denote by~$D'_\alpha$,
and exceptional ones,
corresponding to divisors contracted by~$\pi_p$,
which we denote by $E'_\beta$. 
Since $Y_p\to X$ is $G$-equivariant,
the rational function~$f_{\mathbf a(p)}$ has a divisor
$\div_X(f_{\mathbf a(p)})$ on~$X$ whose strict transform
dominates 
the divisor of~$f_{\mathbf a(p)}$ on~$Y_p$: their difference
is an effective linear combination of the divisors~$E'_\beta$
(Lemma~1.4 of~\cite{chambert-loir-t2002}).
Moreover, the relative Jacobian divisor of $Y_p/X_p$
is a linear combination of these divisors, with positive coefficients.

As in Section~3.4 of~\cite{chambert-loir-tschinkel:2012},
these geometric facts imply 
that only the distinguished vertices
of $\Clan(Y_p,Y_{p}\setminus G)$,
that is, those of $\Clan(X,X\setminus G)$, intervene.
Moreover,
letting $d_p=1+\dim(\Clan(X,X\setminus G)_0)$ 
and $d=1+\dim(\Clan(X,X\setminus G))$,
then one has $d_p<d$
(cf. Lemmas~3.4.5 and 3.5.4 of~\cite{chambert-loir-tschinkel:2012}).

Applying Proposition~\ref{prop.mIie-poles},
we conclude that there exists an integer~$a$ such that
$(1-\Lef^a T^a)^{d_p} Z_\lambda(T,\tcb{\mathbf a(p)})\in\GVarM_k\{T\}^\dagger$
for every $p\in P$.
\color{black}

This concludes the proof of Proposition~\ref{prop.v.nontrivial}.

\bibliographystyle{smfplain}
\bibliography{aclab,acl,mot}
%
%
%
%
%
%

%

\end{document}